\documentclass[12pt]{amsart}
\usepackage{latexsym}
\usepackage{amssymb} 
\usepackage{amscd}  
\renewcommand\a{\alpha}
\renewcommand\b{\beta}
\newcommand\g{\gamma}
\renewcommand\d{\delta}
\newcommand\la{\lambda}
\newcommand\z{\zeta}
\newcommand\e{\eta}
\renewcommand\th{\theta}
\newcommand\io{\iota}

\newcommand\s{\sigma}
\newcommand\x{\chi}
\newcommand\f{\phi}
\newcommand\vf{\varphi}
\newcommand\p{\psi}
\renewcommand\t{\tau}
\renewcommand\r{\rho}

\newcommand\w{\omega}

\newcommand\vS{\varSigma}

\newcommand\vD{\varDelta}

\newcommand\vT{\varTheta}
\newcommand\vG{\varGamma}
\newcommand\ve{\varepsilon}

\newcommand\Fq{{\mathbf F}_q}

\newcommand\Ql{\bar{\mathbf Q}_l}

\newcommand\BP{\mathbf P}

\newcommand\BZ{\mathbf Z}

\newcommand\BH{\mathbf H} 

\newcommand\BG{\mathbf G}

\newcommand\Bm{\mathbf m}

\newcommand\Bk{\mathbf k}

\newcommand\Bla{\boldsymbol\lambda}

\newcommand\Bmu{\boldsymbol\mu}

\newcommand\CA{\mathcal{A}}
\newcommand\CB{\mathcal{B}}
\newcommand\ZC{\mathcal{C}}
\newcommand\CH{\mathcal{H}}

\newcommand\CE{\mathcal{E}}
\newcommand\CL{\mathcal{L}}
\newcommand\DD{\mathcal{D}}
\newcommand\CS{\mathcal{S}}
\newcommand\CM{\mathcal{M}}
\newcommand\CN{\mathcal{N}}
\newcommand\CO{\mathcal{O}}

\newcommand\CP{\mathcal{P}}

\newcommand\CU{\mathcal{U}}

\newcommand\CF{\mathcal{F}}
\newcommand\CG{\mathcal{G}}

\newcommand\CZ{ \mathcal{Z}}
\newcommand\CX{ \mathcal{X}}
\newcommand\CY{ \mathcal{Y}}
\newcommand\CW{ \mathcal{W}}

\newcommand\Fg{\mathfrak g}

\newcommand\iv{^{-1}}
\newcommand\wh{\widehat}
\newcommand\wt{\widetilde}
\newcommand\wg{^{\wedge}}
\newcommand\we{\wedge}

\newcommand\ol{\overline}

\newcommand\spade{\spadesuit}
\newcommand\hra{\hookrightarrow}
\newcommand\lra{\leftrightarrow}

\newcommand\IC{\operatorname{IC}}
\newcommand\Ch{\operatorname{Ch}}

\newcommand\Hom{\operatorname{Hom}}
\newcommand\End{\operatorname{End}}

\newcommand\Ind{\operatorname{Ind}}

\newcommand\Res{\operatorname{Res}}

\newcommand\codim{\operatorname{codim}\,}
\newcommand\supp{\operatorname{supp}\,}

\newcommand\rk{\operatorname{rank}\,}

\newcommand\ps{\operatorname{ps}}

\newcommand\reg{_{\operatorname{reg}}}

\newcommand\unip{\operatorname{uni}}
\newcommand\uni{_{\operatorname{uni}}}
\newcommand\nil{_{\operatorname{nil}}}
\newcommand\id{\operatorname{id}}

\newcommand\lp{\operatorname{\!\langle\!}}
\newcommand\rp{\operatorname{\!\rangle}}
\renewcommand\Im{\operatorname{Im}}

\newcommand\dw{\dot w}

\newcommand{\isom}{\,\raise2pt\hbox{$\underrightarrow{\sim}$}\,}
\numberwithin{equation}{section}

\newtheorem{thm}{Theorem}[section]
\newtheorem{lem}[thm]{Lemma}
\newtheorem{cor}[thm]{Corollary}
\newtheorem{prop}[thm]{Proposition}

\def \para#1{\par\medskip\textbf{#1}
              \addtocounter{thm}{1}}

\def \remark#1{\par\medskip\noindent
                \textbf{Remark #1}
                \addtocounter{thm}{1}}

\begin{document}
\setlength{\baselineskip}{4.9mm}
\setlength{\abovedisplayskip}{4.5mm}
\setlength{\belowdisplayskip}{4.5mm}
%%%
%%%
\renewcommand{\theenumi}{\roman{enumi}}
\renewcommand{\labelenumi}{(\theenumi)}
\renewcommand{\thefootnote}{\fnsymbol{footnote}}
%%%
\renewcommand{\thefootnote}{\fnsymbol{footnote}}
%%%
\allowdisplaybreaks[2]
%\NoBlackBoxes
\parindent=20pt
%%%%%%%%%%%%%%%%%%%%
%%%%%%%%%%%%%%%%%%%%
%%%%%%%%%%%%%%%%%%%%
\medskip
\begin{center}
 {\bf Exotic symmetric space over a finite field, III} 
\\
\vspace{1cm}
Toshiaki Shoji
 and Karine Sorlin\footnote{ supported by 
ANR JCJC REPRED, ANR-09-JCJC-0102-01.}
\\
\vspace{0.3cm}
\end{center}
\title{}
\begin{abstract}
Let $\CX = G/H \times V$, where $V$ is a symplectic space 
such that $G = GL(V)$ and $H  = Sp(V)$.  
In previous papers, the authors constructed 
character sheaves on $\CX$, based on the explicit data.
On the other hand, there exists a conceptual definition of  
character sheaves on $\CX$ based on the idea of Ginzburg in the
case of symmetric spaces. Our character sheaves form a subset of 
Ginzburg type character sheaves.  In this paper we show that 
these two definitions actually coincide, which implies 
a classification of Ginzburg type character sheaves on $\CX$. 
 
\end{abstract}
\maketitle
\markboth{SHOJI AND SORLIN}{ EXOTIC SYMMETRIC SPACE, III}
\pagestyle{myheadings}

\par\bigskip

\begin{center}
{\sc Introduction}
\end{center}
\par\medskip
\addtocounter{section}{-1}

Let $\Bk$ be an algebraic closure of a finite field 
$\Fq$ with odd characteristic.  Let  $V$ be a $2n$-dimensional
vector space over $\Bk$ and let $G = GL(V)$.  
There exists an involutive automorphism $\th$ on $G$ such that
$H = G^{\th} \simeq Sp_{2n}$. Put
$G^{\io\th} = \{ g \in G \mid \th(g) = g\iv \}$, which is isomorphic 
to the symmetric space $G/H$.  
 We consider the variety 
$\CX = G^{\io\th} \times V$ on which $H$ acts diagonally.  
In [SS], [SS2], the intersection cohomology complexes associated to 
$H$-orbits on $\CX$ were studied.  In particular, the set of 
character sheaves $\wh\CX$ on $\CX$ was defined in [SS] as a certain set of
$H$-equivariant simple perverse sheaves on $\CX$.   
We consider the $\Fq$-structure on $\CX$ with Frobenius map 
$F: \CX \to \CX$, and denote by $\wh\CX^F$ the set of $F$-stable character 
sheaves $A \in \wh \CX$ (i.e.  such that $F^*A \simeq A$).   
It was shown in [SS2] that the set of characteristic functions 
of character sheaves in $\wh\CX^F$ forms a basis of the space 
$\ZC_q(\CX)$ of $H^F$-invariant functions on $\CX^F$, as far as $q$ is large enough.  
\par
Our definition of character sheaves on $\CX$ was 
based on the explicit data of simple perverse sheaves on $\CX$.
However, there exists a more conceptual approach for defining 
character sheaves.  In fact, Ginzburg [Gi] defined character 
sheaves on symmetric spaces in a way 
independent of the classification, which is a generalization of 
Lusztig's definition of character sheaves, and Grojnowski [Gr]
classified thus defined character sheaves in the case of $G/H$.
As pointed out by Henderson and Trapa [HT], it is possible to
extend Ginzburg type definition of character sheaves 
on the symmetric space $G/H$ to our variety $\CX \simeq G/H \times V$.      
We denote the set of thus defined character sheaves by $\wh\CX'$, 
tentatively.  Then we have $\wh\CX \subset \wh\CX'$, and in fact,
$\wh\CX$ corresponds to the so-called principal series part 
of $\wh\CX'$.  
As an analogue of the theory of character sheaves on reductive groups, 
it is natural to expect that the set of $F$-stable character sheaves
$\wh\CX'^F$ will produce a basis of $\ZC_q(\CX)$.  
So we conjectured in [SS2] that $\wh \CX = \wh\CX'$.
The aim of this paper is to show that the equality holds true. 
\par
On the other hand, for any finite dimensional vector space $V$ over 
$\Bk$, we consider the variety $\CX_1 =  G \times V$ with $G = GL(V)$, 
on which $G$ acts diagonally.  The set of 
 character sheaves on 
$\CX_1$, which is a certain set of $G$-equivariant 
simple perverse sheaves on $\CX_1$, was introduced 
by Finkelberg, Ginzburg  and Travkin [FGT]. 
We denote it as $\wh\CX_1'$. 
They call these character sheaves as mirabolic character sheaves.  
Also a certain explicitly classified subset $\wh\CX_1$ of  
$\wh\CX_1'$ was introduced in [FGT], and studied extensively by them  in connection with 
affine Grassmannian and Hall algebras, and by Achar and Henderson [AH] 
in connection with the geometry of the enhanced nilpotent cone.  
They conjecture that $\wh\CX_1' = \wh\CX_1$ in [FGT, Conjecture 1]. 
In this paper, we prove their conjecture, and give a positive answer to Lusztig's 
question whether $\wh\CX'^F_1$ will produce a basis of $\ZC_q(\CX_1)$. 
\par
Our strategy is quite similar to the case of character sheaves 
developed by Lusztig [L3, I].  We consider a variety  
$\CX$ which is a direct product of various $\CX, \CX_1$ as above
and of various $GL_{2m}/Sp_{2m}$, $GL_m$. 
We construct induction functors 
and restriction functors on a certain category related to $\CX$, 
and define  a cuspidal character sheaf as in the case of 
original character sheaves, by making use of restriction functors.  
We show that in our case cuspidal 
character sheaves only occur in the case 
where $G$ is a torus, and that any character sheaf in $\wh\CX'$ 
is obtained as a simple direct summand of the induction from 
the Borel subgroup.  This gives the required identity 
$\wh\CX = \wh\CX'$.
\par
The authors would like to thank the referees for valuable suggestions, especially 
for informing them  about Braden's theorem, which contributed to improve the content of the
paper.

\par\bigskip
\begin{center}
\sc{ Table of Contents}
\end{center}

\par\medskip

0. Preliminaries on perverse sheaves
\par
1. Intersection cohomology on $G^{\io\th} \times V$-revisited
\par
2. Intersection cohomology on $GL(V) \times V$
\par
3. Definition of character sheaves
\par
4. Induction
\par
5. Restriction
\par
6. Classification of character sheaves

\par\bigskip
\section{Preliminaries on perverse sheaves}

\para{0.1.}
In this section, we list up some known properties of 
perverse sheaves which will be used later (mainly after Section 3).  
The basic references for them are [BBD] and [BL]. 
Let $\Bk$ be an algebraic closure of a finite field $\Fq$.
For an algebraic variety $X$ over $\Bk$, let  
$\DD X = \DD_c^b(X, \Ql)$ be the bounded 
derived category of constructible $\Ql$-sheaves.
Let $\DD X^{\le 0}$ be the full subcategory of $\DD X$ whose objects 
are those $K$ in $\DD X$ such that, for any integer $i$, 
$\dim\supp \CH^iK \le -i$. 
Let $\DD X^{\ge 0}$ be the full subcategory of $\DD X$ whose 
objects are those $K$ such that $DK \in \DD X^{\le 0}$, where 
$DK$ is the Verdier dual of $K$.  The category of perverse 
sheaves $\CM X$ is defined as the full subcategory
of $\DD X$ whose objects are those 
$K \in \DD X^{\le 0} \cap \DD X^{\ge 0}$.
Hence $\CM X$ is the heart of the $t$-structure of $\DD X$ 
defined by $(\DD X^{\le 0}, \DD X^{\ge 0})$, and is an abelian 
category in which all objects have finite length.
\par
We put $\DD X^{\le n} = \DD X^{\le 0}[-n]$ and 
$\DD X^{\ge n} = \DD X^{\ge 0}[-n]$ for an integer $n$.
The inclusion $\DD X^{\le n} \to \DD X$ has a right adjoint
${}^p\t_{\le n} : \DD X \to \DD X^{\le n}$, and the inclusion
$\DD X \to \DD X^{\ge n}$ has a left adjoint 
${}^p\t_{\ge n}: \DD X^{\ge n} \to \DD X$. 
We have a cohomological functor 
${}^pH^0 := {}^p\t_{\ge 0}{}^p\t_{\le 0} \simeq 
   {}^p\t_{\le 0}{}^p\t_{\ge 0} :  \DD X \to \CM X$.
The perverse cohomology functor ${}^pH^i: \DD X \to \CM X$ is 
defined by ${}^pH^iK = {}^pH^0(K[i])$  for $i \in \BZ$.
For $K \in \DD X$, a simple perverse sheaf 
$A \in \CM X$ is said to be a perverse constituent of $K$ if 
$A$ is isomorphic to a simple constituent of some ${}^pH^i K$. 
\par\medskip\noindent
(0.1.1)
Let $K \in \DD X$.  Then $K$ belongs to $\DD X^{\le n}$ if and only
if ${}^pH^iK = 0$ for $i > n$.  A similar results holds also for $\DD X^{\ge n}$.   
\par
\medskip\noindent
(0.1.2) \ Let $F$ be a functor from $\DD X$ to $\DD Y$ between 
triangulated categories. 
Let $K \in \CM X$, and assume that $F(A) \in \DD Y^{\le n}$ for 
any simple constituent $A$ of $K$.  Then we have $F(K) \in \DD Y^{\le n}$.
Similar results hold also for $\DD Y^{\ge n}$ and for $\CM Y$. 

\par\medskip\noindent
(0.1.3) \  Assume that $A \in \DD X^{\le 0}$ and $B \in \DD X^{\ge 1}$.
Then 
\begin{equation*}
\Hom (A, B) = 0.
\end{equation*} 
In fact this follows from the adjointness property of ${}^p\t_{\le 0}$.

\para{0.2.}
Let $f: X \to Y$ be a morphism between algebraic varieties 
$X$ and $Y$.  We have functors $f_*, f_!: \DD X \to \DD Y$ and 
$f^*, f^! : \DD Y \to \DD X$ which satisfy the following adjointness
properties; for any $A \in \DD X, B \in \DD Y$, 
\begin{align*}
\tag{0.2.1}
\Hom (f^*B, A) &= \Hom (B, f_*A), \\
\tag{0.2.2}
\Hom (f_!A, B) &= \Hom (A, f^!B). 
\end{align*} 
If $f$ is proper, $f_* = f_!$.  If $f$ is   
smooth with connected fibres
of dimension $d$, $f^! = f^*[2d]$. 
In this case, we set $\wt f = f^*[d]$. 
\par
Assume that $f$ is smooth with connected fibres of dimension $d$.
The following properties are well-known (see [BBD, 4.2.5]);
let $K \in \DD Y$. Then $K \in \DD Y^{\le 0}$ 
if and only if $\wt fK \in \DD X^{\le 0}$, and 
$K \in \DD Y^{\ge 0}$ if and only if $\wt fK \in \DD X^{\ge 0}$.
Hence $K \in \CM Y$ if and only if
$\wt fK \in \CM X$.  Moreover, the perverse cohomology functor 
${}^pH^i$ commutes with $\wt f$. We have 
\par\medskip\noindent
(0.2.3) \ If $K \in \DD Y^{\le 0}$, $K' \in \DD Y^{\ge 0}$, then
\begin{equation*}
\Hom_{\DD Y}(K,K') = \Hom_{\DD X}(\wt fK, \wt fK').
\end{equation*} 
In particular, $\wt f: \CM Y \to \CM X$ is fully faithful. 

\para{0.3.}
If $G$ is an algebraic group acting on $X$, 
we denote by $\DD^G(X)$ the $G$-equivariant derived category on $X$ 
in the sense of Bernstein and Lunts [BL], which is a triangulated 
category with $t$-structure.  As its heart, we have 
a category $\CM^G(X)$ of $G$-equivariant perverse sheaves.
We have forgetful functors $\DD^G(X) \to \DD X$, 
$\CM^G(X) \to \CM X$.
If $G$ is connected, $\CM^G(X) \to \CM X$
is fully faithful. 

\par
If $G$ acts on $X, Y$, and $f : X \to Y$ is a $G$-equivariant map, 
the functors $f_!: \DD X \to \DD Y$, $f^* : \DD Y \to \DD X$ can be 
lifted to the functors between $\DD^G(X)$ and $\DD^G(Y)$, which 
we denote by the same symbols $f_!, f^*$. 
If $f$ is a principal $G$-bundle, then $f^*: \DD^G(Y) \to \DD^G(X)$
factors through an equivalence $\DD Y \isom \DD^G(X)$ by
[BL, 2.2.5].  We denote by $f_{\flat}: \DD^G(X) \to \DD Y$
the inverse of this equivalence.  

\para{0.4.}
Let $T$ be a torus acting freely on $X$, and $\CL$ be a local system 
of $T$ of rank one. We denote by $\DD_{\CL}(X)$ the full subcategory
of $\DD X$ whose objects are complexes $K$ such that 
$a^*K \simeq \CL\boxtimes K \in \DD(T\times X)$, where 
$a: T \times X \to X$ is the action of $T$.
$\DD_{\CL}(X)$ inherits the $t$-structure of $\DD X$, and we denote
its heart by $\CM_{\CL}(X)$.   We have forgetful functors 
$\DD_{\CL}(X) \to \DD X$, $\CM_{\CL}(X) \to \CM X$, and the latter
is fully faithful.  These properties come from  [BL]. 
If an algebraic group $G$ acts on $X$, commuting with  the action of $T$, 
then one can define $\DD_{\CL}^G(X)$ and $\CM_{\CL}^G(X)$ in an 
obvious way. 

\par
The following remark ([Gr, 0.6]) is basic.
\par\medskip
(0.4.1) \ Let $T$ be a functor $\DD X \to \DD Y$ 
which sends distinguished triangles in $\DD X$ to those in 
$\DD Y$, and commutes with degree shift (such as the composite 
of $f_!, f^*, f_{\flat}$ and degree shift).  Take $K \in \DD X$.
If $A$ is a perverse constituent of $T(K)$, then  
$A$ is a perverse constituent of $T(A')$ for some perverse 
constituent $A'$ of $K$.  
\par\medskip
In fact, this follows easily from the perverse cohomology 
long exact sequence associated to the distinguished triangle
$(T(K_1), T(K_2), T(K_3))$, where 
\begin{equation*}
(K_1, K_2, K_3) = ({}^pH^iK[-i], {}^p\t_{ \ge i}K, {}^p\t_{\ge  i+1}K) 
\end{equation*} 
is the distinguished triangle obtained by applying the functor
${}^p\t_{\ge 0}$ to the distinguished triangle 
$({}^p\t_{\le 0}(K[i]), K[i], {}^p\t_{\ge 1}(K[i]))$.  

\para{0.5.}
Assume that $\BG_m$ acts on a variety $X$.  An object $K \in \DD(X)$ 
is called weakly equivariant if $K \in \DD_{\CL}(X)$ for some 
local system $\CL$ on $\BG_m$, where $\DD_{\CL}(X)$ is defined as in 0.4, 
without assuming the freeness of the action. 
Let $X^0$ be the closed 
subvariety of $X$ consisting of fixed points by $\BG_m$.  We define 
varieties
\begin{align*}
X^+ &= \{ x \in X \mid \lim_{\substack{t \to 0}}t\cdot x \in X^0 \} \\
X^- &= \{ x \in X \mid \lim_{\substack{t \to \infty}}t\cdot x \in X^0\}.  
\end{align*}
Let $i^{\pm}: X^0 \hra X^{\pm}$ and $p^{\pm} : X^{\pm} \hra X$ be the 
inclusion maps. 
The following result of Braden is crucial for later discussions. 
%%%%
%%%%
\begin{thm}[{Braden [B, Theorem 1, Theorem 2]}]
Assume that $X$ is a normal variety with $\BG_m$-action. For 
a weakly equivariant object $K$ in $\DD(X)$, there exists a natural isomorphism 
\begin{equation*}
(i^+)^!(p^+)^*K \simeq (i^-)^*(p^-)^!K.
\end{equation*}
Moreover, if $K$ is a weakly equvariant simple perverse sheaf on $X$, 
then $(i^+)^!(p^+)^*K$ is a semisimmple complex on $X^0$.
\end{thm}

The following lemma is also useful to apply Braden's theorem to our setting
in later discussions (see also Braden [B, Lemma 6]).

\begin{lem}[{Springer [Sp, Proposition 1]}]
Let $q : Y \to Z$ be a vector bundle with $\BG_m$-action, where 
$\BG_m$ acts trivially on $Z$. Assume that $\BG_m$ acts linearly on 
each fibre of $q$ with strictly positive weights. Let $i: Z \hra Y$ be the 
inclusion of the zero section.  Then for a weakly equivariant object 
$K \in \DD(Y)$, there exists natural isomorphisms $i^!K \isom q_!K$, 
$q_*K \isom i^*K$.  
\end{lem}

\par\bigskip

\section{Intersection cohomology on $G^{\io\th} \times V$ -revisited}
\para{1.1.}
We follow the notation in [SS].  In particular, 
$G = GL(V)$ with $\dim V = 2n$, and $\th : G \to G$ is an 
involutive automorphism on $G$ such that $H = G^{\th} \simeq Sp_{2n}$.
Let $B$ be a $\th$-stable Borel subgroup of $G$ containing a 
$\th$-stable maximal torus $T$. As in [SS], put 
$\CX = G^{\io\th} \times V$ on which $H$ acts diagonally, where 
$G^{\io\th} = \{ g\in G \mid \th(g) = g\iv\}$, which coincides with 
the set $\{ g\th(g)\iv \mid g \in G\}$. 
Let $M_0 \subset M_1 \subset \cdots \subset M_{2n} = V$ be the total 
flag in $V$ stabilized by $B$, here $M_0 \subset \cdots \subset M_n$ 
is the isotropic flag corresponding to $B^{\th}$, and we have 
$M_{n+i} = M_{n-i}^{\perp}$ for $i = 1, \dots, n$.   
Extending the notations in [SS, 3.1], 
we define, for any integer $0 \le m \le 2n$,  
\begin{align*}
\wt\CX_m &= \{ (x,v,gB^{\th}) \in G^{\io\th} \times V \times H/B^{\th}
                 \mid g\iv xg \in B^{\io\th}, g\iv v \in M_m\}, \\
\CX_m &= \bigcup_{g \in H}g(B^{\io\th}\times M_m), \\
\wt\CY_m &= \{ (x,v, gB^{\th}) \in G^{\io\th}\reg \times V \times H/B^{\th}
                 \mid g\iv xg \in B^{\io\th}\reg, g\iv v \in M_m\}, \\
\CY_m &= \bigcup_{g \in H}g(B^{\io\th}\reg \times M_m), 
\end{align*}
where $G^{\io\th}\reg, B^{\io\th}\reg$ are as in [SS, 1.8].
We consider the diagram 
\begin{equation*}
\tag{1.1.1}
\begin{CD}
T^{\io\th} @<\a^{(m)} <<  \wt\CX_m @>\pi^{(m)} >> \CX_m,
\end{CD}
\end{equation*}
where  $\pi^{(m)}(x,v,gB^{\th}) = (x,v)$ and  
$\a^{(m)}(x,v,gB^{\th}) = p(g\iv xg)$ ($p: B^{\io\th} \to T^{\io\th}$
is the projection).
For a tame local system $\CE$ on $T^{\io\th}$, consider the
complex $K_m = K_{m, T,\CE} = (\pi^{(m)})_*(\a^{(m)})^*\CE[\dim \CX_m]$.
In the case where $m = n$, the structure of $K_n$ was described in 
[SS, Theorem 4.2].  In this section, we shall extend it to the 
case $K_{n'}$ where $n'$ is any integer $0 \le n' \le 2n$. 

\para{1.2.}  
First we consider the case where $n' \le n$.
(What we need in later discussions is the case where $n' \ge n$.
We include the case $n' < n$ just for the reference.)
We fix an integer $n' \le n$, and write 
$\CX_{n'}, \CY_{n'}, \pi^{(n')}$, etc. as $\CX', \CY', \pi'$, etc. 
Let $\p': \wt\CY' \to \CY$ be the restriction of $\pi'$ 
on $\wt\CY'$. 
As in [SS, 3.2], for each subset $I \subset [1,n']$ 
such that $|I| = m$, we 
define a subset $M_I$ of $M_{n'}$,
and a map $\p_I : \wt\CY_I \to \CY_m^0$, where 
$\CY_m^0 = \CY_m\backslash \CY_{m-1}$ and 
\begin{equation*}
\wt\CY_I = H \times^{B^{\th} \cap Z_H(T^{\io\th})}(T^{\io\th}\reg \times M_I).
\end{equation*}
Then $\wt\CY_I$ can be identified with a subset of $\wt\CY'$.
Put $(\wt\CY_m^+)' = (\p')\iv(\CY_m^0)$.  
We define $\CW = N_H(T^{\io\th})/Z_H(T^{\io\th}) \simeq S_n$, and 
let $\CW'$ be a subgroup of $\CW$ corresponding to the subgroup 
$S_{n'}\times S_{n-n'}$ 
of $S_n$.  Then $\CW'$ acts naturally on $(\wt\CY_m^+)'$, and we have 
as in [SS, (3.4.1)]
\begin{equation*}
(\wt\CY_m^+)' = \coprod_{\substack{I \subset [1,n']\\ |I| = m}}
\wt\CY_I = \coprod_{w \in \CW'/\CW'_{\Bm}}w(\wt\CY_{[1,m]}),
\end{equation*}  
where $\CW'_{\Bm}$ is the stabilizer of $[1,m]$ in $\CW'$
(i.e., $\CW'_{\Bm} \simeq S_m \times S_{n'-m} \times S_{n-n'}$).
We denote by $\p'_m: (\wt\CY_m^+)' \to \CY^0_m$ the restriction of 
$\p'$ on $(\wt\CY_m^+)'$.  Then $\p'_m$ is $\CW'$-equivariant
with respect to the trivial action of $\CW'$ on $\CY_m^0$. 
Let $\a_0$ (resp. $\a_I$) be the restriction of $\a : \wt\CX \to T^{\io\th}$
to $\wt\CY$ (resp. to $\wt\CY_I$).  Let $\CE$ be a tame local system 
on $T^{\io\th}$.  Then as in [SS, (3.4.2)], we have
\begin{equation*}
(\p'_m)_*\a_0^*\CE|_{\wt\CY'} \simeq 
        \bigoplus_{\substack{I \subset [1,n'] \\ |I| = m}}
             (\p_I)_*\a_I^*\CE.
\end{equation*}
In the following discussion, we denote the restriction of $\a_0^*\CE$ on 
various subvarieties such as $\wt\CY', (\wt\CY_m^+)'$ etc. simply by $\a_0^*\CE$ 
if there is no fear of confusion.
\para{1.3.}
By [SS, 3.2], the map $\p_I$ is factorized as 
\begin{equation*}
\begin{CD}
\p_I: \wt\CY_I  @>\xi_I>> \wh \CY_I @>\e_I>> \CY_m^0,
\end{CD}
\end{equation*}
and the map $\xi_I$ is a locally trivial fibration with fibre 
isomorphic to $(SL_2/B_2)^{I'} \simeq \BP_1^{I'}$, and the map 
$\e_I$ is a finite Galois covering with Galois group $\CW_I$, where 
$I'$ is the complement of $I$ in $[1,n]$, and $\CW_I$ is the 
stabilizer of $I$ in $\CW$, i.e., $\CW_I \simeq S_m \times S_{n-m}$.
The map $\a_I$ is factored as $\a_I = \b_I\circ \xi_I$ by the 
map $\b_I: \wh\CY_I \to T^{\io\th}$. 
Let $\CE_I = \b_I^*\CE$ be a local system on $\wh\CY_I$, and $\CW_{\CE_I}$
the stabilizer of $\CE_I$ in $\CW_I$.  We put 
$\CW_{\CE_I} = \CW_{\Bm,\CE}$ if $I = [1,m]$.  
Then by
[SS, (3.4.3)] and by the remark below (3.4.3), we have 
$\End((\e_I)_*\CE_I) \simeq \Ql[\CW_{\CE_I}]$, the group algebra
of $\CW_{\CE_I}$, and $(\e_I)_*\CE_I$ is decomposed as 
\begin{equation*}
(\e_I)_*\CE_I \simeq \bigoplus_{\r \in \CW_{\CE_I}^{\we}}\r \otimes \CL_{\r},
\end{equation*}  
where $\CL_{\r} = \Hom (\r, (\e_I)_*\CE_I)$ is a simple local system 
on $\CY_m^0$.
\par 
Put $\wt\CW' = \CW' \ltimes (\BZ/2\BZ)^{n'}$, which is the direct product of
the Weyl group 
of type $C_{n'}$ and $S_{n-n'}$.  Let $\CW'_{\CE} = \CW' \cap \CW_{\CE}$ and  
$\CW'_{\Bm,\CE} = \CW' \cap \CW_{\CE_I}$ for $I = [1,m]$. 
We define 
$\wt\CW'_{\CE} = \CW'_{\CE} \ltimes (\BZ/2\BZ)^{n'}$, 
$\wt\CW'_{\Bm,\CE} = \CW'_{\Bm, \CE} \ltimes (\BZ/2\BZ)^{n'}$.
Then by a similar argument as in [SS, 3.5], we see that 
\begin{equation*}
\tag{1.3.1}
(\p'_m)_*\a_0^*\CE \simeq 
   H^{\bullet}(\BP_1^{n-n'})\otimes \bigoplus_{\r \in \CW_{\Bm,\CE}^{\wg}}
  \Ind_{\wt\CW'_{\Bm,\CE}}^{\wt\CW'_{\CE}}
         (H^{\bullet}(\BP_1^{n'-m})\otimes \r)\otimes \CL_{\r},
\end{equation*} 
where $\r$ is considered as a $\CW'_{\Bm,\CE}$-module by the restriction, 
and then is extended to $\wt\CW'_{\Bm,\CE}$-module (trivial extension), 
and $H^{\bullet}(\BP_1^{n'-m})$ is regarded as an 
$S_{n'-m}\ltimes (\BZ/2\BZ)^{n'-m}$-module through the Springer action 
of $(\BZ/2\BZ)^{n'-m}$ and the $S_{n'-m}$ action arising from permutations
of factors in $\BP^{n'-m}_1$. 
We extend $\r$ to a $\wt\CW'_{\Bm,\CE}$-module $\wt\r$ in such a way that 
the $i$-th factor $\BZ/2\BZ$ acts trivially on $\wt\r$ if $i \le m$,
and 
acts non-trivially otherwise. We define a 
$\wt\CW'_{\CE}$-module $\wt V'_{\r}$ by 
$\wt V'_{\r} = \Ind_{\wt\CW'_{\Bm,\CE}}^{\wt\CW'_{\CE}}\wt\r$. 
Note that $\wt V'_{\r}$ is not necessarily 
irreducible.  
The formula (1.3.1) can be rewritten as
\begin{equation*}
\tag{1.3.2}
(\p'_m)_*\a_0^*\CE \simeq 
            \biggl(\bigoplus_{\r \in \CW_{\Bm,\CE}\wg}
     \wt V_{\r}'\otimes \CL_{\r}\biggr)[-2(n-m)] + \CN'_m,
\end{equation*}
where $\CN'_m$ is a sum of various $\CL_{\r}[-2i]$ for $\r \in \CW_{\Bm,\CE}\wg$ 
with $0 \le i < n-m$ (cf. Appendix (a)). 

\par
For each $m \le n'$, let $\ol\p'_m$ be the restriction of $\p'$ on $(\p')\iv(\CY_m)$. 
We show the following formula, which is an analogue of (3.6.1*) in Appendix.
Note that $d_m = \dim \CY_m$.
\begin{equation*}
\tag{1.3.3}
\begin{split}
(\ol\p'_m)_*\a_0^*\CE
   &\simeq 
      \bigoplus_{0 \le m' \le m}\bigoplus_{\r \in \CW_{\Bm',\CE}\wg}
                \wt V'_{\r}\otimes\IC(\CY_{m'},\CL_{\r})
                       [-2(n -m')] + \ol\CN''_m,
\end{split}
\end{equation*}
where $\ol\CN''_m$ is a sum of various $\IC(\CY_{m'}, \CL_{\r})[-2i]$ for 
$\r \in \CW_{\Bm',\CE}\wg$ with $0 \le i < n - m'$.
\par
We show (1.3.3) by induction on $m$.  In the case where $m = 0$, 
$(\ol\p_m)_*\a_0^*\CE$ coincides with $(\p_m)_*\a_0^*\CE$.  Hence 
(1.3.3) holds by (1.3.2). 
Assume that (1.3.3) holds for $m-1$. Here $\CY_{m-1}$ is a closed 
subset of $\CY_m$ and $\CY_m^0 = \CY_m \backslash \CY_{m-1}$. 
The restriction of $(\ol\p'_m)_*\a_0^*\CE$ on $\CY_m^0$ (resp. on $\CY_{m-1}$) 
coincides with $(\p'_m)_*\a_0^*\CE$ (resp. $(\ol\p'_{m-1})_*\a_0^*\CE$).  
Since $(\ol\p'_m)_*\a_0^*\CE$ is a semisimple complex, it is a direct sum of 
various $A[i]$, where $A$ is a simple perverse sheaf. As in the discussion in 
the proof of (3.6.1*) in Appendix, 
if $\supp A \cap \CY_m^0 \ne \emptyset$, such $A$ can be given by 
the formula for $(\p'_m)_*\a_0^*\CE$ in (1.3.2).  If 
$\supp A \cap \CY_m^0 = \emptyset$, then $\supp A \subset \CY_{m-1}$, and 
such $A$ can be given by the formula (1.3.3) by induction hypothesis.
Now take $A$ such that $\supp A \cap \CY_m^0 \ne \emptyset$.  As discussed in Appendix, 
if there is no contribution of $A|_{\CY_{m-1}}$ to the former factors
of (1.3.3) for $(\ol\p'_{m-1})_*\a_0^*\CE$, we obtain the formula (1.3.3) 
for $m$. 
But by comparing (1.3.1) and (3.5.2) in Appendix (or in [SS]),
such $A$ is exactly the same $A$ appearing in the decomposition of $(\ol\psi_m)_*\a_0^*\CE$.
Then by the discussion in the proof of (3.6.1*) in Appendix, we see that 
$A|_{\CY_{m-1}}$ gives no contribution to the former factors  in 
(3.6.1*) for $(\ol\p_{m-1})_*\a_0^*\CE$.  Hence it gives no contribution for 
the former factors of (1.3.3) as asserted.  This proves (1.3.3) for $m$.       
\par
By considering the case where $m = n'$, we obtain the following result, which is 
an analogue of Proposition 3.6 in [SS] (see Appendix).
%%%%
%%%%
\begin{prop}  %%% Proposition 1.4.
$(\p')_*\a_0^*\CE[d_{n'}]$ is a semisimple complex on 
$\CY'$, equipped with $\wt\CW'_{\CE}$-action, and is 
decomposed as
\begin{equation*}
\p'_*\a_0^*\CE[d_{n'}] \simeq 
     \bigoplus_{0 \le m \le n'}
       \bigoplus_{\r \in \CW_{\Bm,\CE}^{\wg}}\wt V'_{\r} \otimes 
            \IC(\CY_m,\CL_{\r})[d_m - 2(n-n')] + \ol\CN'_{n'},
\end{equation*} 
where $\ol\CN'_{n'}$ is a sum of various $\IC(\CY_m, \CL_{\r})[d_m - 2i]$
for $\r \in \CW_{\Bm,\CE}\wg$ with $m - n' \le i < n - n'$. 
\end{prop}

The following result is a generalization of Theorem 4.2 in [SS].

%%%%
%%%%
\begin{thm}  %%%% Theorem 1.5.
Assume that $n' \le n$ and let 
$K_{n', T, \CE} = \pi'_*\a^*\CE[d_{n'}]$.  Then 
under the notation in Proposition 1.4, we have
\begin{equation*}
K_{n',T,\CE} \simeq 
      \bigoplus_{0 \le m \le n'}\bigoplus_{\r \in \CW_{\Bm,\CE}\wg}
             \wt V'_{\r} \otimes\IC(\CX_m,\CL_{\r})[d_m -2(n - n')] + \CM'_{n'},
\end{equation*}
where $\CM_{n'}$ is a sum of various $\IC(\CX_m, \CL_{\r})[d_m - 2i]$ 
for $\r \in \CW_{\Bm, \CE}\wg$ with $m - n' \le i < n - n'$.
\end{thm}

\begin{proof}
Let $\CX_m^0 = \CX_m \backslash \CX_{m-1}$. 
Put $\wt\CX^+_m = (\pi^{(n)})\iv(\CX_m^0)$ and 
$(\wt\CX_m^+)' = (\pi')\iv(\CX_m^0)$.
Hence $(\wt\CX_m^+)'$ is a closed subvariety of $\wt\CX_m^+$.
In [SS, 4.3], for each subset $I$ of $[1,n]$ such that $|I| = m$, 
a subset $\CX_I$ of $\wt\CX_m^+$ was introduced, and it was proved
in [SS, Lemma 4.4] that $\wt\CX_m^+$ is a disjoint union of 
$\wt\CX_I$, where $I$ runs over the set of subsets of $[1,n]$ such that 
$|I| = m$, and $\wt\CX_I$ gives an irreducible component of $\wt\CX_m^+$.
This argument can be applied also for the case of $(\wt\CX_m^+)'$, 
and we have
\begin{equation*}
\tag{1.5.1}
(\wt\CX_m^+)' = \coprod_{\substack{ I \subset [1,n'] \\  |I| = m}}
                          \wt\CX_I 
       \subset \coprod_{\substack{ I \subset [1,n] \\ |I| = m}}
                          \wt\CX_I = \wt\CX_m^+.
\end{equation*}
In particular, the first equality gives the decomposition of $(\wt\CX_m^+)'$
into (a disjoint union of) irreducible components. 
\par
Let $\pi_m$ be the restriction of 
$\pi$ on $\wt\CX_m^+$, and $\pi_m'$ that of 
$\pi'$ on $(\wt\CX_m^+)'$.    
As in the case $\a_0^*\CE$, we denote the restriction of $\a^*\CE$ on 
various varieties such as $\wt\CX', (\wt\CX_m^+)'$, etc. by $\a^*\CE$. 
Then the decomposition of $\wt\CX_m^+$ into irreducible components
implies that $(\pi_m)_*\a^*\CE
   \simeq \bigoplus_I(\pi_I)_*\a^*\CE$, where 
$\pi_I: \wt\CX_I \to \CX_m^0$ is the restriction of $\pi$ 
to $\wt\CX_I$ for each $I \subset [1,n]$. 
Then it follows from (1.5.1) that 
\par\medskip\noindent
(1.5.2) \ $(\pi'_m)_*\a^*\CE$ is a direct summand 
of $(\pi_m)_*\a^*\CE$.    
\par\medskip
The following formula is an analogue of Proposition 4.8 in [SS].
\begin{equation*}
\tag{1.5.3}
(\pi'_m)_*\a^*\CE \simeq 
H^{\bullet}(\BP_1^{n-n'})\otimes 
    \bigoplus_{\r \in \CW_{\Bm,\CE}\wg} H^{\bullet}(\BP_1^{n'-m})
             \otimes V'_{\r}\otimes \IC(\CX_m^0, \CL_{\r}).
\end{equation*}
(Here $V'_{\r} = \Ind_{\CW'_{\Bm,\CE}}^{\CW'_{\CE}}\r$
is regarded as a vector space, ignoring the $\CW'_{\CE}$-module
structure, which coincides with $\wt V'_{\r}$ as a vector space.)  
We show (1.5.3).    $\CY_m^0$ is an open dense smooth subset of 
$\CX_m^0$, and the restriction of $(\pi_m')_*(\a')^*\CE$ on $\CY_m^0$
coincides with $(\p'_m)_*(\a'_0)^*\CE$.  By (1.3.1), 
we have a similar formula as (1.5.3) for 
$(\p'_m)_*(\a_0')^*\CE$.   Note that $(\pi_m')_*(\a')^*\CE$ is a semisimple 
complex. Hence, in order to show (1.5.3), 
it is enough to see that $\supp A \cap \CY_m^0 \ne \emptyset$ 
for any direct summand $A[i]$ ($A$: simple perverse sheaf) 
of $(\pi'_m)_*(\a')^*\CE$. But by (1.5.2), and Proposition 4.8 in [SS],
we see that any simple perverse sheaf appearing in the decomposition of
$(\pi_m')_*(\a')^*\CE$ is of the form $\IC(\CX_m^0,\CL_{\r})$ up to shift, 
hence its support 
has a non-trivial intersection with $\CY_m^0$.   Thus (1.5.3) is
proved. 
\par
Let $\ol\pi'_m$ be the restriction of $\pi'$ on $(\pi')\iv(\CX_m)$. 
By a similar discussion as in the proof of (1.3.3), by using 
(1.5.3) and (Appendix (4.9.1*)) instead of (1.3.1) and 
(Appendix (3.6.1*)), 
one can show 
the following formula for any $m$ such that $ 0 \le m \le n'$. 
\begin{equation*}
\tag{1.5.4}
\begin{split}
(\ol\pi'_m)&_*\a^*\CE[d_m]  \\ 
  &\simeq \bigoplus_{0 \le m' \le m}\bigoplus_{\r \in \CW_{\Bm',\CE}\wg}
    \wt V'_{\r}\otimes \IC(\CX_{m'}, \CL_{\r})[d_{m'}-2(n -m)] + \CM'_m, 
\end{split}
\end{equation*}
where $\CM'_m$ is a sum of various $\IC(\CX_{m'}, \CL_{\r})[d_{m'} - 2i]$ for 
$\r \in \CW_{\Bm',\CE}\wg$ with $m' - m \le i < n - m$. 
The theorem  follows from (1.5.4) by putting $m = n'$.
\end{proof}

\para{1.6.}
We consider the case where $n' > n$.
By fixing such an $n'$, we consider the objects 
$\CX', \CY', \pi', \psi'$, etc. as before. 
$M_{n'}$ can be written as 
$M_{n'} = M_{n_0}^{\perp}$ for an integer $n_0$, 
and if we choose a suitable symplectic basis 
$\{ e_1, \dots, e_n, f_1, \dots, f_n\}$ of $V$ consisting of 
eigenvectors for $T$, we have 
$M_{n'} = \lp e_1, \dots, e_{n_0}, e_{n_0+1}, f_{n_0+1}, 
\dots, e_n, f_n \rp$. 
For each $I \subset [1,n_0], J \subset [n_0+1, n]$, 
we define a subset $M_{I,J}$ of $M_{n'}$ by  
\begin{equation*}
M_{I,J} = M_I \times \prod_{j \in J}\lp e_j, f_j\rp^* 
\end{equation*}
where $M_I$ is as in 1.2, and 
$\lp e_i, f_i\rp^* = \lp e_i, f_i\rp \backslash  \{ 0\}$. 
Then $M_{I,J}$ is a $B^{\th} \cap Z_H(T^{\io\th})$-stable subset
of $M_{n'}$.  
We have a partition $M_{n'} = \coprod_{I,J}M_{I,J}$, and 
$\bigcup_{g \in H}g(T^{\io\th}\reg \times M_{I,J}) \subset \CY^0_m$ 
for $m = |I| + |J|$.
Put
\begin{equation*}
\wt\CY_{I,J} = H\times^{B^{\th}\cap Z_H(T^{\io\th})}
           (T^{\io\th}\reg \times M_{I,J}). 
\end{equation*}
Then $\wt\CY_{I,J}$ can be identified with a subset of $\wt\CY$. 
Put $(\wt\CY_m^+)' = (\psi')\iv(\CY_m^0)$ for $0 \le m \le n$.  
Then we have 
\begin{equation*}
  (\wt\CY_m^+)' = \coprod_{|I| + |J| = m}\wt\CY_{I,J}.
\end{equation*}
One can check that $\wt\CY_{I,J}$ are irreducible, and this gives
a decomposition of $(\wt\CY_m^+)'$ into irreducible components. 
But note that $\wt\CY_{I,J}$ are not equidimensional. 
As in 1.2, we denote by $\psi_m': (\wt\CY_m^+)' \to \CY_m^0$ the 
restriction of $\psi'$ on $(\wt\CY_m^+)'$, and 
$\psi_{I,J}: \wt\CY_{I,J} \to \CY_m^0$ the restriction of 
$\psi'$ on $\wt\CY_{I,J}$.
Then under a similar notation as in 1.2, we have
\begin{equation*}
\tag{1.6.1}
(\psi_m')_*(\a_0')^*\CE \simeq \bigoplus_{\substack{I \subset [1,n_0], 
J \subset [n_0+1, n] \\ |I| + |J| = m }}
     (\psi_{I,J})_*\a_{I,J}^*\CE.
\end{equation*}

\par
For $I \subset [1,n_0]$, we define 
a parabolic subgroup $Z_H(T^{\io\th})_I$ of $Z_H(T^{\io\th})$ 
as in [SS, 3.2].  Then $M_{I,J}$ is $Z_H(T^{\io\th})_I$-stable, and 
one can define 
\begin{equation*}
\wh \CY_{I,J} = H \times^{Z_H(T^{\io\th})_I}(T^{\io\th}\reg \times M_{I,J}). 
\end{equation*}
As in [SS, (3.2.1)], the map $\psi_{I,J}$ is decomposed as 
\begin{equation*}
\begin{CD}
\psi_{I,J}: \wt\CY_{I,J} @>\xi_{I,J}>>  \wh\CY_{I,J} @>\e_{I,J}>>  \CY_m^0,
\end{CD}
\end{equation*}
where $\xi_{I,J}$ is a locally trivial fibration with fibre isomorphic to
$\BP_1^{I'} \simeq \BP_1^{n-|I|}$.
Here we consider the variety 
\begin{equation*}
\wh\CY_{I\cup J} = H \times^{Z_H(T^{\io\th})_{I\cup J}}
                    (T^{\io\th}\reg \times M_{I\cup J}) 
\end{equation*}
 and the map
$\e_{I\cup J} : \wh\CY_{I\cup J} \to \CY_m^0$ as in [SS, 3.4], 
which is a finite Galois covering with group $\CW_{I\cup J}$. 
Note that $M_{I\cup J} \subset M_{I,J}$ and 
$Z_H(T^{\io\th})_{I\cup J} \subset Z_H(T^{\io\th})_I$.
It follows that the inclusion 
$T^{\io\th}\reg \times M_{I\cup J} \hra T^{\io\th}\reg \times M_{I,J}$
induces a map $\vf: \wh\CY_{I\cup J} \to \wh \CY_{I,J}$
such that $\e_{I\cup J} = \e_{I,J}\circ \vf$.  
Here we have 
\begin{equation*}
\wh\CY_{I\cup J} \simeq H/P_1 \times T^{\io\th}\reg,
\quad
\wh\CY_{I,J} \simeq H/P_2 \times T^{\io\th}\reg,
\end{equation*} 
where $P_1 = Z_{Z_H(T^{\io\th})_{I\cup J}}(z)$, 
$P_2 = Z_{Z_H(T^{\io\th})_I}(z)$ for 
$z \in T^{\io\th}\reg \times M_{I\cup J}$.
Since $P_1 = P_2$, we see that $\vf$ gives an isomorphism 
$\wh\CY_{I\cup J} \isom \wh\CY_{I,J}$. 
The map $\b_{I,J}: \wh\CY_{I,J} \to T^{\io\th}$ is defined 
similarly to $\b_{I\cup J}$ as in [SS, 3.4].  We put 
$\CE_{I,J} = \b_{I,J}^*\CE$ and $\CE_{I\cup J} = \b_{I\cup J}^*\CE$.
Then we have 
\begin{equation*}
\tag{1.6.2}
(\e_{I, J})_*\CE_{I, J} \simeq  (\e_{I\cup J})_*\CE_{I \cup J}.
\end{equation*}  
Thus as in (3.5.2) in [SS], we have by (1.6.1) and (1.6.2),
\begin{equation*}
  \tag{1.6.3}
(\psi_m')_*(\a'_0)^*\CE|_{(\wt\CY^+_m)'}
  \simeq \bigoplus_{\substack{I,J \\
      |I|+|J| = m }}
         \bigoplus_{\r \in \CW\wg_{I\cup J,\CE}}
    H^{\bullet}(\BP_1^{I'})\otimes\r\otimes \CL_{\r}.
\end{equation*}

\para{1.7.}
For each $1 \le m \le n$, let $\CX_m^0 = \CX_m\backslash \CX_{m-1}$
as before.
For each $(x,v) \in G^{\io\th} \times V$, one can associate a subspace
$W(x,v)$ of $V$ as in [SS, 4.3], and we have by [SS, (4.3.3)] 
\begin{equation*}
\CX_m^0 = \{ (x,v) \in G^{\io\th} \times V \mid \dim W(x,v) = m\}.
\end{equation*}
Put $(\wt\CX_m^+)' = (\pi')\iv(\CX_m^0)$.
For $(x,v,gB^{\th}) \in (\wt\CX_m^+)'$, we define its level $(I,J)$
($I \subset [1,n_0], J \subset [n_0+1, n]$) as follows;
assume that $(x,v) \in B^{\io\th} \times M_{n'}$ and that $x = su$, 
the Jordan decomposition of $x$ with $s \in T^{\io\th}$.  
$V$ is decomposed into eigenspaces of $s$ as 
$V = V_1 \oplus \cdots \oplus V_t$.  Then $v \in M_{n'}$ can 
be written as $v = v_1 + \cdots + v_t$
with $v_i \in M_{n'} \cap V_i$ since $M_{n'}$ is $s$-stable.
Let $u_i$ be the restriction of $u$ on $V_i$. 
Thus $(u_i,v_i) \in (G_i)\uni^{\io\th} \times V_i$ with $G_i = GL(V_i)$.
Now $M_{n'} \cap V_i$ has a basis which is  
a subset of the basis $\{ e_1, \dots, e_n, f_1, \dots, f_n\}$.
By using this basis, we decompose  
$M_{n'} \cap V_i = E'_i \oplus E''_i$ with $E_i' = M_{n_0} \cap V_i$.   
Let $\ol v_i$ be the image of $v_i$ in $(M_{n'}\cap V_i)/E_i'$, and $\ol u_i$ 
the linear map on $(M_{n'}\cap V_i)/E_i'$ induced from $u_i$.
Let $m_i = \dim \Bk[u_i]v_i$ and $m_i'' = \dim \Bk [\ol u_i]\ol v_i'$, where 
$\Bk[u_i]v_i = \lp v_i, u_iv_i, u_i^2v_i, \dots \rp$, and similarly 
for $\Bk[\ol u_i]\ol v_i'$. 
We have $m_i'' \le m_i$, and $\sum_im_i = m$ since $(x,v) \in \CX_m^0$.
Put $m_i' = m_i - m_i''$. 
Then $m_i' = \dim (\Bk[u_i]v_i \cap E_i')$. 
The basis of $V_i$ is given by $\{e_j, f_j \mid j \in K_i\}$ for a 
subset $K_i$ of $[1,n]$. 
Then the basis of $E'_i$ is given by $\{ e_j \mid j \in K_i'\}$ and 
the basis of $E_i''$ is given by $\{ e_j, f_j \mid j \in K_i''\}$ 
for subsets $K'_i = [1,n_0] \cap K_i, 
K_i'' = [n_0+1, n] \cap K_i$.  
As in [SS, 4.3], we define a subset $I_i$ of $[1,n_0]$ as the first 
$m_i'$ letters in $K_i'$, and put $I = \coprod_i I_i$.  
Similarly, we define a subset $J_i$ of $[n_0+1, n]$ as the first 
$m_i''$ letters in $K_i''$, and put $J = \coprod_iJ_i$. 
From the construction, we have 
$I \subset [1,n_0], J \subset [n_0+1, n]$ and 
$|I| + |J| = m$,  and the attachment $(x,v) \mapsto (I,J)$ depends 
only on the $B^{\th}$-conjugate of $(x,v)$. Thus we have a well-defined
map $(x,v,gB^{\th}) \mapsto (I,J)$.  We define a subvariety 
$\wt\CX_{I,J}$ of $(\wt\CX_m^+)'$ by 
\begin{equation*}
\wt\CX_{I,J} = \{ (x,v, gB^{\th}) \in (\wt\CX_m^+)'
                   \mid (x,v,gB^{\th}) \mapsto (I,J) \}.
\end{equation*} 
Then $\wt\CY_{I,J}$ is an open dense subset of $\wt\CX_{I,J}$ for 
each $I,J$.  The following lemma can be proved in a similar way 
as [SS, Lemma 4.4].

\begin{lem}  %%%  Lemma 1.8.
%$(\wt\CX_m^+)'$ is decomposed as 
\begin{equation*}
(\wt\CX_m^+)' = \coprod_{\substack{
       I \subset [1,n_0], J \subset [n_0+1,n] \\
                            |I| + |J| = m }}  \wt\CX_{I,J}.
\end{equation*}
\end{lem}

\para{1.9}
As in [SS, 4.5] we consider the spaces 
$W_0 = M_m, \ol V_0 = W_0^{\perp}/W_0$, 
and put $G_1 = GL(W_0), G_2 = GL(\ol V_0)$. Then $\ol V_0$ 
has a natural symplectic structure, and $G_2$ is identified with 
a $\th$-stable subgroup of $G$.
Let $B_1$ be a Borel subgroup of $G_1$ which is the stabilizer of 
the flag $(M_k)_{0 \le k \le m}$, and $B_2$ be a $\th$-stable Borel
subgroup of $G_2$ which is the stabilizer of the flag 
$(M_{m+1}/M_m \subset \cdots \subset M_m^{\perp}/M_m)$ in $G_2$.
Put 
\begin{align*}
\wt G_1 &= \{ (x,gB_1) \in G_1 \times G_1/B_1 \mid g\iv x g \in B_1\}, \\
\wt G_2^{\io\th} &= \{ (x,gB_2^{\th}) \in G_2^{\io\th} 
          \times G_2^{\th}/B_2^{\th} \mid g\iv xg \in B_2^{\io\th}\},
\end{align*}
and define maps $\pi^1: \wt G_1 \to G_1, 
       \pi^2: \wt G_2^{\io\th} \to G_2^{\io\th}$ by first projections.
For a fixed $m \le n$, we consider the varieties 
\begin{align*}
\CG_m^0 &= \{ (x,v, W) \mid (x,v) \in G^{\io\th} \times V, 
                     W = W(x,v), \dim W = m\} \simeq \CX_m^0, \\
\CH^0_m &= \{(x,v, W, \f_1, \f_2) \mid (x,v,W) \in \CG^0_m,  \\
             &\phantom{*****} 
           \f_1: W \isom W_0, \f_2: W^{\perp}/W \isom \ol V_0 
   \text{ (symplectic isom.) }\},  
\end{align*} 
and morphisms
\begin{align*}
q_1&: \CH_m^0 \to \CG_m^0, \quad (x,v,W,\f_1, \f_2) \mapsto (x,v,W), \\
\s_1&: \CH_m^0 \to G_1 \times G_2^{\io\th}, 
        (x,v,W, \f_1, \f_2) \mapsto 
             (\f_1(x|_W)\f_1\iv, \f_2(x|_{W^{\perp}/W})\f_2\iv). 
\end{align*}

Extending the discussion in [SS, 4.8], we define a subset 
$\wt\CX_{I,J}^{\spade}$ of $(\wt\CX_m^+)'$ as follows;
For $(x, v) \in G^{\io\th} \times M_{n'}$ with 
$x = su \in G^{\io\th}$ such that $s \in T^{\io\th}$, 
the subsets $K_i, K_i', K_i''$ of $[1,n]$ are defined as before. 
We define $\wt\CX_{I,J}^{\spade}$ as the set of all $H$-conjugates 
of $(x,v, w_{\g}B^{\th}) \in (\wt\CX_m^+)'$ such that 
$|w_{\g}\iv(K'_i) \cap I| = m_i'$ and 
$|w_{\g}\iv(K''_i) \cap J| = m_i''$ for each $i$ and for all the 
possible choice of $(x,v)$ and $w_{\g} \in W_H$, where  
$w_{\g}$ is the distinguished representative of the coset 
$\g  \in \vG = W_{H,s}\backslash W_H$.  
Then $\wt\CX_{I,J}^{\spade}$ is a closed subset of $(\wt\CX_m^+)'$ 
containing $\wt\CX_{I,J}$. 
($\wt\CX_{I,J}$ coincides with the $H$-conjugates of $(x,v, w_{\g}B^{\th})$
such that $w_{\g}\iv(K'_i) \cap I$ consists of the first $m_i'$ letters
in $w_{\g}\iv (K_i')$ 
and that $w_{\g}\iv(K''_i) \cap J$ consists of the first $m_i''$ letters 
in $w_{\g}\iv(K_i'')$).
We define a map $\vf_{I,J} : \wt\CX_{I,J}^{\spadesuit} \to \CG_m^0$ by 
$(x,v, gB^{\th}) \mapsto (x,v, W(x,v))$. 
Define a variety
\begin{equation*}
\CZ_{I,J}^{\spadesuit} = 
  \{ (x,v, gB^{\th}, \f_1, \f_2) \mid 
    (x,v,gB^{\th}) \in \wt\CX_{I,J}^{\spadesuit}, 
   \f_1: U \isom W_0, \f_2: U^{\perp}/U \isom \ol V_0\},
\end{equation*}
where $U = W(x,v)$, and define maps 
\begin{align*}
q_{I,J} &: \CZ_{I,J}^{\spadesuit} \to \wt\CX_{I,J}^{\spadesuit},  
(x,v, gB^{\th}, \f_1, \f_2) \mapsto (x,v, gB^{\th}), \\ 
\wt\vf_{I,J}&: \CZ_{I,J}^{\spadesuit} \to \CH_m^0,  
(x,v,gB^{\th}, \f_1, \f_2) \mapsto (x,v, W(x,v), \f_1, \f_2).
\end{align*}
We define a map $\s_{I,J}: \CZ_{I,J}^{\spadesuit} \to 
    \wt G_1 \times \wt G_2^{\io\th}, 
(x,v, gB^{\th}, \f_1, \f_2) \mapsto ((x_1, g_1B_1), (x_2, g_2B^{\th}_2))$ 
as follows;  
let $U = W(x,v)$, and put $x_1 = \f_1(x|_U)\f_1\iv \in G_1$, 
$x_2 = \f_2(x|_{U^{\perp}/U})\f_2\iv \in G_2^{\io\th}$. 
Let $(M_k)_{0 \le k \le 2n}$ be the total flag in $V$ as before, 
and let $(U_i)$ be the total flag in $U$ obtained from 
$(U \cap g(M_k))_{0 \le k \le 2n}$.  This defines an $x_1$-stable
total flag in $W_0$ via $\f_1$, and we denote the corresponding element 
in $G_1/B_1$ by $g_1B_1$.  Thus $(x_1, g_1B_1) \in \wt G_1$.     
Contrast to the case in [SS, 4.5], $g(M_n)$ is not necessarily 
contained in $U^{\perp}$, but one can check that 
$(g(M_n) \cap U^{\perp}) + U/U$ is a maximal isotropic subspace in 
$\ol U = U^{\perp}/U$.  Hence we obtain an isotropic flag $(\ol U_j)$ on 
$\ol U $ from $((g(M_k) \cap U^{\perp})+U/U)_{0 \le k \le n}$.
This defines an $x_2$-stable isotropic flag in $\ol V_0$ via $\f_2$,
and $(x_2, g_2B_2^{\th}) \in \wt G_2^{\io\th}$ is determined from 
this.  Thus $\s_{I,J}$ is defined. 
\par
We have a commutative diagram
\begin{equation*}
\tag{1.9.1}
\begin{CD}
\wt G_1 \times \wt G_2^{\io\th} @<\s_{I,J}<<  \CZ_{I,J}^{\spade} 
             @>q_{I,J}>> \wt\CX_{I,J}^{\spade}  \\
@V\pi^1 \times \pi^2 VV   @VV\wt\vf_{I,J}V   @VV\vf_{I,J} V                 \\
   G_1 \times G_2^{\io\th} @<\s_1<<  \CH_m^0  @> q_1>>  \CG_m^0, 
\end{CD}
\end{equation*}
where the right hand square is a cartesian square. 
We consider the fibre product $\ZC_m$ of
$\wt G_1 \times \wt G_2^{\io\th}$ with $\CH_m^0$ over 
$G_1 \times G_2^{\io\th}$,  and let 
$c_{I,J}: \CZ_{I,J}^{\spadesuit} \to \ZC_m$ be the natural map 
obtained from the diagram (1.9.1).  
Note that $\ZC_m$ is the same variety given in [SS, 4.8].
Let $\ZC_m = \coprod_{\a} \ZC_{\a}$ be the $\a$-partition of 
$\ZC_m$ defined in [loc. cit.].
We show by a similar argument as in the proof of [SS, (4.8.2)] that 
\par\medskip\noindent
(1.9.2) \ 
The restriction of $c_{I,J}$ on $c_{I,J}\iv(\ZC_{\a})$ is 
a locally trivial fibration over $\ZC_{\a}$.
\par\medskip
In fact, under the notation in [SS, 4.8], for 
$z = (x,v, \f_1, \f_2, (\ol U_i)) \in \ZC_{\a}$ with $x$ : unipotent, 
the fibre $c_{I,J}\iv(z)$ is given as 
\begin{equation*}
c_{I,J}\iv(z) \simeq \{ (V_k) \in \CF_x^{\th}(V) \mid 
              V_{n_0} \subset U^{\perp},  
             ((V_k\cap U^{\perp}) + U/U) \to (\ol U_j) \}.
\end{equation*}
(In this case, $I = [1,m_1']$ and $J = [n_0+1, n_0 + m_1'']$ 
with $m = m_1' + m_1''$.)  Then as in [SS, 4.8], 
$c_{I,J}\iv(\ZC_{\a,\unip}) \to \ZC_{\a,\unip}$ 
gives a locally trivial fibration, where $\ZC_{\a,\unip}$ is the 
restriction of $\ZC_{\a}$ to its unipotent part. 
On the other hand, under the general setting, 
take $(x,v, U, \f_1, \f_2) \in \CH_m^0$
and $(x_1, g_1B_1) \in \wt G_1$, 
$(x_2,g_2B_2^{\th}) \in \wt G_2^{\io\th}$ so that it gives 
an element $z \in \ZC_{\a}$.  Then $\g_1 \in \vG_1$ is determined 
from $(x_1, g_1B_1)$ and $\g_2 \in \vG_2$ is determined from 
$(x_2, g_2B_2^{\th})$, where $\vG_1, \vG_2$ are given by 
the semisimple part $s_1, s_2$ of $x_1, x_2$.  
Now $\g \in \vG$ is determined from $(\g_1, \g_2)$ as follows;
we assume that the semisimple part $s$ of $x$ is contained in 
$T^{\io\th}$.  The sets $K_i', K_i''$ are determined as before. 
Put $K_I = \coprod_iK_i' \cap I$ and $K_J = \coprod_iK_i'' \cap J$. 
Put $K^+ = K_I \cup K_J$ and let $K^-$ be the complement of $K^+$ in $[1,n]$.  
Then $w_{\g_1}$ 
(resp. $w_{\g_2}$) is regarded via $\f_1$ (resp. $\f_2$) 
as a permutation on $K^+$ (resp. $K^-$).  We define a permutation 
$w \in S_n$ by
\begin{equation*}
w\iv(j) = \begin{cases}
             w_{\g_1}\iv(j) &\text{ if } j \in K^+, \\
             w_{\g_2}\iv(j) &\text{ if } j \in K^-.
           \end{cases}
\end{equation*}
We take $\g = W_{H,s}w \in W_{H,s}\backslash W_H$.  
Then $w_{\g}$ satisfies the condition that 
$|w_{\g}\iv(K_i) \cap I| = m_i'$ and 
$|w_{\g}\iv(K_i) \cap J| = m_i''$ for each $i$.
We have a bijection between $(\g_1, \g_2)$ arising from
$\wt G_1 \times \wt G_2^{\io\th}$ and $\g$ arising from 
$\CZ_{I,J}^{\spadesuit}$.
Then for such $z \in \ZC_{\a}$, $c_{I,J}\iv(z)$ is contained 
in $\CF_u^{\g}(V)$, and by using the argument in the unipotent 
case, we see that $c_{I,J}\iv(\ZC_{\a}) \to \ZC_{\a}$ is 
a locally trivial fibration. Hence (1.9.2) is proved.  
\par
By using a similar argument as in the proof of Proposition 4.8 
in [SS1], it follows from (1.6.3) that
\par
\begin{prop} %%%  Prop. 1.10
$(\pi_m')_*(\a')^*\CE|_{(\wt\CX_m^+)'}[d_m]$ is a semisimple complex
and is decomposed as 
\begin{equation*}
(\pi_m')_*(\a')^*\CE|_{(\wt\CX_m^+)'}[d_m] \simeq
\bigoplus_{\substack{I,J \\ |I| + |J| = m}}\bigoplus_{\r \in \CW_{I\cup J,\CE}\wg}
      H^{\bullet}(\BP_1^{I'})\otimes\r
               \otimes \IC(\CX_m^0, \CL_{\r})[d_m].
\end{equation*}  
\end{prop}

We can now prove

\begin{thm}  %%%%%  Theorem 1.11
Assume that $n' > n$ and let 
$K_{n', T,\CE} = \pi'_*(\a')^*\CE[d_{n'}]$.  
Then $K_{n',T,\CE}$ is a semisimple complex in $\DD(\CX)$, 
and its simple summands (up to shift) are
contained in the set 
$\bigcup_{0 \le m \le n}\{ \IC(\CX_m,\CL_{\r})[d_m] \mid \r \in 
            \CW_{\Bm,\CE}\wg \}$. 
\end{thm}

\begin{proof}
For each $0 \le m \le n$, let $\ol\pi'_m$ be the restriction of 
$\pi'$ on $(\pi')\iv(\CX_m)$. 
Since $\ol\pi'_m$ is proper, $(\ol\pi_m')_*(\a')^*\CE[d_m]$ is a 
semisimple complex. Since 
$(\ol\pi'_{n'})_*(\a')^*\CE[d_{n'}]$ coincides with $K_{n',T,\CE}$
(note that $d_n' = d_n$ if $n' > n$), it is enough to show 
that its simple summands (up to shift) are contained 
in the set 
$\bigcup_{0 \le m' \le m}
   \{ \IC(\CX_{m'},\CL_{\r})[d_{m'}] \mid \r \in \CW_{\Bm',\CE}\wg\}$. 
But this is proved by a similar discussion as in [SS, 4.9] (see also 
Appendix, II),
thanks to Proposition 1.10. 
\end{proof}  

Although the explicit decomposition of $K_{n',T,\CE}$ is complicated
in general, $K_{n',T,\CE}$ has a simple description for $n' = 2n$ as 
follows.

\begin{prop}  %%%%%  Prop. 1.12
\begin{equation*}
K_{2n,T,\CE} \simeq H^{\bullet}(\BP_1^n)\otimes \bigoplus_{\r \in \CW_{\CE}\wg}
         \r \otimes \IC(\CX_n, \CL_{\r})[\dim \CX_n],
\end{equation*}
where $\CW_{\CE}$ is the stabilizer of $\CE$ in $\CW \simeq S_n$, 
and $\CL_{\r}$ is a simple local system on an open dense subset of 
$\CX_n = \CX$. Moreover, we have
\begin{equation*}
\IC(\CX_n,\CL_{\r}) \simeq \IC(G^{\io\th}, \CL'_{\r})\boxtimes (\Ql)_V,
\end{equation*}
where $\CL'_{\r}$ is a simple local system on an open dense subset of $G^{\io\th}$, 
and $(\Ql)_V$ is a constant sheaf $\Ql$ on $V$.  
\end{prop}

\begin{proof}
We consider the map $\pi^{(2n)}: \wt\CX_{2n} \to \CX_{2n} = \CX$ 
as in 1.1.  We consider a variety
\begin{equation*}
\wt G^{\io\th} = \{ (x, gB^{\th}) \in G^{\io\th} \times H/B^{\th}
      \mid g\iv xg \in B^{\io\th}\}
\end{equation*}
and the diagram 
\begin{equation*}
\begin{CD}
T^{\io\th} @< \a^{(0)}<<  \wt G^{\io\th} @>\pi^{(0)}>>  G^{\io\th},
\end{CD}
\end{equation*}
where $\pi^{(0)} : (x,gB^{\th}) \mapsto x, \a^{(0)}: (x, gB^{\th}) \mapsto p(g\iv xg)$.
Then $\wt\CX_{2n} \simeq \wt G^{\io\th} \times V$, and $\pi^{(2n)}$ can be identified 
with the map $\pi^{(0)} \times \id$. 
Hence $K_{2n, T,\CE} \simeq K_{0,T,\CE} \boxtimes (\Ql)_V[\dim V]$, where 
$K_{0,T,\CE} = \pi^{(0)}_*(\a^{(0)})^*\CE[\dim G^{\io\th}]$. 
By Grojnowski [Gr, Lemma 7.4.4], we know the decomposition of 
the semisimple complex $K_{0,T,\CE}$ (see Theorem 1.16 and (1.15.2) in [SS]).
In particular, we have
\begin{equation*}
\tag{1.12.1}
K_{2n,T,\CE} \simeq H^{\bullet}(\BP_1^n)\otimes \bigoplus_{\r \in \CW_{\CE}\wg}
        \r \otimes \IC(G^{\io\th}, \CL'_{\r})[\dim G^{\io\th}]\boxtimes (\Ql)_V[\dim V].
\end{equation*}
It follows that $K_{2n, T,\CE}$ is a semisimple complex such that
each simple summand (up to shift) has its support $\CX = \CX_n$. 
By Theorem 1.11, any simple summand in $K_{2n, T, \CE}$ 
is of the form $\IC(\CX_m, \CL_{\r})$ up to shift for some $0 \le m \le n$.  
Hence it should be $\IC(\CX_n, \CL_{\r})$.  
Under the notation in [SS, 3.4], we have a commutative diagram 
\begin{equation*}
\begin{CD}
\wt\CY_n^+ @>\psi_n>>  \CY_n^0  \\
@V\wt f VV            @VV f V      \\
\wt G^{\io\th}\reg @>\psi_0>>  G^{\io\th}\reg,   
\end{CD}
\end{equation*}
where $\wt f, f$ are natural projections.  
Since $\psi_n, \psi_0$ are Galois coverings with Galois group $S_n$, 
this square is a cartesian square. 
It follows that 
\begin{equation*}
f^*(\IC(G^{\io\th},\CL'_{\r})|_{G^{\io\th}\reg}) \simeq \IC(\CX_n, \CL_{\r})|_{\CY_n^0}.
\end{equation*}
This shows that $\IC(G^{\io\th},\CL'_{\r}) \boxtimes (\Ql)_V|_{\CY_n^0} \simeq 
                 \IC(\CX_n, \CL_{\r})|_{\CY_n^0}$, 
which implies the second assertion.  
The first assertion follows from this by (1.12.1).  The proposition is proved.
\end{proof}

\para{1.13.}
We fix $s \in T^{\io\th}$.  Let $U$ be the unipotent radical of $B$.  We consider a variety 
\begin{equation*}
\begin{split}
\CZ = \{ (x,v, &gB^{\th}, g'B^{\th}) \in G^{\io\th} \times V \times H/B^{\th} 
              \times H/B^{\th}  \\
         & \mid g\iv xg \in (sU)^{\io\th}, g\iv v \in M_n, 
                {g'}\iv xg' \in (sU)^{\io\th}, {g'}\iv v \in M_n \},  
\end{split}
\end{equation*}
where $(sU)^{\io\th} := sU \cap G^{\io\th}$. 
We consider a partition $H/B^{\th} \times H/B^{\th} = \coprod_{w \in W_n}X_w$ into 
$H$-orbits, and put $\CZ _w = p\iv (X_w)$, where 
$p: \CZ \to H/B^{\th} \times H/B^{\th}$
is the projection onto the last two factors. 
Recall that $\nu_H = \dim U^{\th}$. 
The following result will be used later.

\begin{prop}  %%%%  Prop. 1.14
Let $\CO$ be an $H$-orbit in $\CX$, and put $c = \dim \CO$.
\begin{enumerate}
\item
We have $\dim (\CO \cap ((sU)^{\io\th} \times M_n)) \le c/2$.
\item
For $z = (x,v) \in \CO$, we have
\begin{equation*}
\dim \{gB^{\th} \in H/B^{\th} \mid g\iv xg \in (sU)^{\io\th}, 
       g\iv v \in M_n\} \le \nu_H - c/2.
\end{equation*}
\item
$\dim \CZ_w \le 2\nu_H$ for all $w \in W_n$.  Hence 
$\dim \CZ \le 2\nu_H$. 

\end{enumerate}
\end{prop}
\begin{proof}
First we show (iii). 
The projection $p$ is $H$-equivariant. A representative of the $H$-orbit $X_w$ 
is given by $(B^{\th}, wB^{\th})$.  Hence in order to show (iii), it is enough to see that 

\begin{equation*}
\tag{1.14.1}
\begin{split}
\dim \{ (x,v ) \in (sU)^{\io\th} \times &M_n \mid (\dw\iv x\dw, \dw\iv v) \in 
               (sU)^{\io\th} \times M_n \} \\ 
               &\le 2\nu_H -  \dim X_w 
\end{split}
\end{equation*}
for each $w \in W_n$, where $\dw$ is a representative of $w$ in $H$.  
Then $x \in B^{\io\th} \cap {}^wB^{\io\th}$ can be written as 
$x = su = s'u'$ with $u \in U$, $s' = wsw\iv \in T$, $u' \in {}^{w}U$.
Hence $s' = s$ and $u = u' \in U^{\io\th} \cap {}^wU^{\io\th}$. 
Moreover $v \in M_n \cap w(M_n)$.  
It follows that the left hand side of (1.14.1) is less than or equal to 
$\dim  (U^{\io\th} \cap {}^wU^{\io\th}) + \dim (M_n  \cap w(M_n))$.
One can check that 
$\dim (M_n \cap w(M_n)) = \sharp \{i \in [1,n] \mid w\iv(i) > 0 \}$. 
On the other hand, if we denote the set of positive roots of type $C_n$
by $\vD^+ = \{ \ve_i \pm \ve_j (1 \le i< j \le n), 2\ve_i (1 \le i \le n)\}$, 
we see that $\dim (U \cap {}^wU)^{\th} = \dim (U \cap {}^wU)^{\io\th} + b_w$,
where 
\begin{equation*}
b_w = \sharp\{ 2\ve_i \in \vD^+ \mid w\iv(2\ve_i) > 0\}.
\end{equation*}
It is easy to see that $b_w = \dim (M_n \cap w(M_n))$.  Hence we have  
$\dim (U \cap {}^wU)^{\io\th} + \dim (M_n \cap w(M_n)) = \dim (U \cap {}^wU)^{\th}$.
Since $\dim X_w = 2\nu_H - \dim (U \cap {}^wU)^{\th}$, we obtain (1.14.1).  
Thus (iii) follows.  
\par
Next we show (ii). Let $q: \CZ \to G^{\io\th} \times V$ be the projection to the first
two factors. For each $H$-orbit $\CO$, we consider $q\iv (\CO)$. 
We may assume that $q\iv(\CO)$ is non-empty, since otherwise the variety in (ii)
is empty.  By (iii), we have $\dim q\iv(\CO) \le 2\nu_H$. On the other hand, 
let $Y_z$ be the variety given in (ii).  Then for each 
$z \in \CO$, $q\iv(z) \simeq Y_z \times Y_z$. Hence 
$\dim Y_z = (\dim q\iv(\CO) - \dim \CO)/2 \le \nu_H  - c/2$, 
and (ii) follows.  
\par
Finally we show (i).  Consider the variety
\begin{equation*}
\tag{1.14.2}
\CX_{\CO} = \{ (x,v), gB^{\th}) \in \CO \times H/B^{\th} 
                   \mid g\iv xg \in (sU)^{\io\th}, g\iv v \in M_n \}, 
\end{equation*}
and let $\a : \CX_{\CO} \to \CO$ be the projection onto $\CO$-factor, and 
$\b: \CX_{\CO} \to H/B^{\th}$ the projection onto $H/B^{\th}$. 
Then for each $z \in \CO$, the fibre
$\a\iv(z)$ is isomorphic to the variety $Y_z$ given in (ii).  
Hence $\dim \CX_{\CO} = \dim Y_z + \dim \CO   \le \nu_H + c/2$. 
On the other hand, each fibre of $\b$ is isomorphic to the variety 
$\CO \cap ((sU)^{\io\th} \times M_n)$. 
Hence 
\begin{equation*}
\dim (\CO \cap ((sU)^{\io\th} \times M_n)) = \dim \CX_{\CO} - \dim H/B^{\th} 
     \le c/2, 
\end{equation*}
and (i) holds.  The proposition is proved. 
\end{proof}

\begin{cor}  %%%%  Cor. 1.15.
Let $\CO$ be an $H$-orbit in $\CX$ containing $z_0 = (x_0,v_0)$. Let 
$x_0 = su$ be the Jordan decomposition. 
Assume that $s \in T^{\io\th}, u \in U^{\io\th}$ and that 
$v_0 \in M_n$.  Then we have
\begin{equation*}
\dim (\CO \cap ((sU)^{\io\th} \times M_n)) = \frac{1}{2}\dim \CO.
\end{equation*}
\end{cor}

\begin{proof}
Let $\CX_{\CO}$ be the variety defined in (1.14.2).  We follow
the notation in the proof of Proposition 1.14. 
In particular, $c = \dim \CO$.  
By Proposition 1.14 (ii),
we have $\dim Y_z \le \nu_H - c/2$ for any $z = (x,v) \in \CO$. 
We want to show that 
\begin{equation*}
\tag{1.15.1} 
\dim Y_z = \nu_H - c/2.
\end{equation*}
Let $L = Z_G(s)$, and consider $\CX_L = L^{\io\th} \times V$. 
Let $\pi^L_1: (\wt\CX_L)_{\unip} \to (\CX_L)_{\unip}$ be the map 
$\pi_1$ defined in [SS, 2.4] with respect to $\CX_L$. 
We have $\CX_L \simeq \prod_i G_i^{\io\th} \times V_i$, where 
$G_i = GL_{2n_i}$ and $\dim V_i = 2n_i$. 
Here $(u,v_0) \in (\CX_L)_{\unip}$, and we consider 
$(\pi_1^L)\iv (u, v_0)$.  Let $\CO_0$ be the $L^{\th}$-orbit of 
$(u, v_0)$ in $(\CX_L)\uni$. 
In the proof of Theorem 5.4 in [SS], it was shown that 
$\dim \pi_1\iv(z) = \nu_H - \dim \CO'/2$ for any $H$-orbit $\CO'$ 
in $\CX\uni$ containing $z$.  
Applying this to our setting, we have
\begin{align*}
\dim (\pi_1^L)\iv(u,v_0) &= \nu_{L^{\th}} - \dim \CO_0/2  \\
                         &= (\dim Z_{L^{\th}}(u,v_0) - \rk L^{\th})/2 \\
                         &= (\dim Z_H(z_0) - \rk H)/2 \\
                         &= \nu_H - \dim \CO/2.
\end{align*}
Here $Y_z$ contains a variety isomorphic to $(\pi_1^L)\iv(u,v_0)$ for 
$z = z_0$.  Since $\dim Y_z$ is constant for any $z \in \CO$, this implies   
that $\dim Y_z \ge \nu_H - c/2$.  Hence (1.15.1) holds.  
Now the corollary follows by a similar argument as in the proof of 
Proposition 1.14 (i).  
\end{proof}
\par\bigskip

\par\bigskip
\section{Intersection cohomology on $GL(V) \times V$}
\para{2.1.} 
In this section, we assume that $V$ is an $n$-dimensional 
vector space over $\Bk$, and $G = GL(V)$.  
We consider the variety $\CX = G \times V$ on which $G$ acts 
diagonally. Put $\CX\uni = G\uni \times V$.  Then $\CX\uni$ 
is a closed $G$-stable subset of $\CX$ isomorphic to the enhanced 
nilpotent cone $\Fg\nil \times V$ studied extensively by 
Achar and Henderson [AH] and Travkin [T]. The following fact was
proved independently by [AH] and [T]. 
%%%
\begin{prop}[{[AH, Proposition 2.3], [T, Theorem 1]}]  %%%  Prop. 2.2. 
The set of $G$-orbits of $\Fg\nil \times V$ is in bijection with 
$\CP_{n,2}$, the set of double partitions of $n$.  
\end{prop}

\para{2.3.}
The explicit correspondence is given as follows; for 
$x \in \Fg\nil$, put $E^x = \{ g \in \End(V) \mid gx = xg\}$.
Then $E^x$ is an $x$-stable subspace of $\End(V)$.  
For $(x,v) \in \Fg\nil \times V$, $E^xv$ is an $x$-stable subspace of 
$V$.  Assume that the Jordan type   
of $x$ is $\nu$, a partition of $n$.  Let $\la^{(1)}$ be 
the Jordan type of $x|_{E^xv}$ and $\la^{(2)}$ the Jordan type of 
$x|_{V/E^xv}$.  Then $\nu = \la^{(1)} + \la^{(2)}$, and we have
$\Bla = (\la^{(1)}, \la^{(2)}) \in \CP_{n,2}$. 
The correspondence $(x,v) \mapsto \Bla$ gives the above bijection.
We denote by $\CO_{\Bla}$ the $G$-orbit in $\Fg\nil \times V$ 
(or the $G$-orbit in $\CX\uni$) corresponding to $\Bla \in \CP_{n,2}$.
Note that the orbit containing $(x,0)$ corresponds to 
$\Bla = (\emptyset, \la^{(2)})$.  In that case, $\CO_{\Bla}$ coincides with 
$\CO_{\nu}$, the $G$-orbit containing $x$ in $\Fg\nil$, with 
$\nu = \la^{(2)}$.    
\par
\begin{prop}[{[AH, Theorem 3.9]}] %%%% Prop.2.4
For $\Bla, \Bmu \in \CP_{n,2}$, $\CO_{\Bmu} \subseteq \ol\CO_{\Bla}$ 
if and only if $\Bmu \le \Bla$, where $\Bmu \le \Bla$ is the partial
order on $\CP_{n,2}$ given in {\rm [SS, 1.7]}. 
\end{prop} 
\par
Recall that $a(\Bla) = 2\cdot n(\Bla) + |\la^{(2)}|$ (cf. [SS2, 5.1]).

%%%%
\begin{prop}[{[AH, Proposition 2.8]}]  %%% Proposition 2.5
Let $\Bla = (\la^{(1)}, \la^{(2)}) \in \CP_{n,2}$, and 
put $\nu = \la^{(1)} + \la^{(2)}$.   Let $(x,v) \in \CO_{\Bla}$. 
Then we have 
\begin{enumerate}
\item 
$Z_G(x,v)$ is a connected algebraic group of dimension $a(\Bla)$.
\item
$\dim \CO_{\Bla} = \dim \CO_{\nu} + |\la^{(1)}| = n^2 - a(\Bla)$, 
where $\CO_{\nu}$ is the $G$-orbit in $\Fg\nil$ containing $x$.
\end{enumerate}
\end{prop}

\para{2.6.}
Let $B = TU$ be a Borel subgroup of $G$, with a maximal torus $T$ and 
the unipotent radical $U$.
Let $M_0 = \{0\} \subset M_1 \subset M_2 \subset \cdots 
                \subset M_{n-1} \subset M_n = V$  
be the total flag stabilized by $B$. 
We fix a basis $\{ e_1, \dots e_n\}$ of $V$ such that 
$M_m = \lp e_1, \dots, e_m\rp$ for any $m$, 
and that $e_i$ are weight vectors for 
$T$.  Let $W = N_G(T)/T$ be the Weyl group of $G$.  Then $W \simeq S_n$
is  identified with the permutation group of the basis 
$\{ e_1, \dots, e_n\}$.   
We define a subset $M_{[1,m]}$ of $M_m$ as before, i.e., 
\begin{equation*}
M_{[1,m]} = \{ v = \sum_{j = 1}^ma_je_j \in M_m \mid a_j \ne 0
  \text{ for any } j\}.   
\end{equation*}
For any $0 \le m \le n$, we define varieties
\begin{align*}
\wt\CX_m &= \{ (x,v, gB) \in G \times V \times G/B 
                 \mid g\iv xg \in B, g\iv v \in M_m\}, \\
\CX_m &= \bigcup_{g \in G}g(B \times M_m), \\
\wt\CY_m &= \{ (x,v,gT) \in G \times V \times G/T 
                  \mid g\iv xg \in T\reg, g\iv v \in M_{[1,m]}\}, \\
\CY_m &= \bigcup_{g \in G}g(T\reg \times M_{[1,m]}), 
\end{align*}
where $T\reg$ is the set of regular semisimple elements in $T$.
In the rest of this section, we fix $m$, and 
define maps $\pi : \wt\CX_m \to \CX_m$ by  
$\pi(x,v,gB) = (x,v)$, $\p: \wt\CY_m \to \CY_m$
by $\p(x,v, gT) = (x,v)$.
Since $\wt\CY_m \simeq G \times{}^T(T\reg \times M_{[1,m]})$,  
$\wt\CY_m$ is smooth and irreducible.
Let $W_{\Bm}$ be the stabilizer of $\{ e_1, \dots, e_m\}$ in $W$.
Hence $W_{\Bm} \simeq S_m \times S_{n-m}$.  Then the map 
$\p : \wt\CY_m \to \CY_m$  is a finite
Galois covering with group $W_{\Bm}$, thus $\CY_m$ is also 
smooth irreducible. We have 
\begin{equation*}
\tag{2.6.1}
\dim \wt\CY_m = \dim \CY_m = \dim G + m.
\end{equation*}    
Note that $\wt\CX_m \simeq G \times^B(B \times M_m)$ is smooth and 
irreducible, and the map $\pi : \wt\CX_m \to \CX_m$ is proper, surjective.
Hence $\CX_m$ is a $G$-stable irreducible closed subvariety of $\CX$.  Since 
$\CX_m \supset \CY_m$, and $\dim \CY_m = \dim \wt\CX_m$, we see that 
\par\medskip\noindent
(2.6.2) \ The closure $\ol \CY_m$ of $\CY_m$ coincides with $\CX_m$.
\par\medskip
Let $\CX_{m,\unip} = \CX_m \cap \CX\uni$, and  put 
$\wt\CX_{m,\unip} = \pi\iv(\CX_{m,\unip})$. 
Let $\pi_1: \wt\CX_{m,\unip} \to \CX_{m,\unip}$ be the restriction of 
$\pi$.
Since $\wt\CX_{m,\unip} \simeq G \times^B(U \times M_m)$, 
$\wt\CX_{m,\unip}$ is smooth, irreducible with 
$\dim \wt\CX_{m,\unip} = \dim G\uni + m$.  Moreover, 
$\pi_1$ is surjective. We note that
\begin{equation*}
\tag{2.6.3}
\CX_{m,\unip} = \ol\CO_{\Bla} \quad\text{ for } \Bla = ((m), (n-m)).
\end{equation*}
\par\medskip\noindent
In fact, by the explicit correspondence given in 2.3, we see that 
$\CO_{\Bla} \subset \CX_{m,\unip}$.  Since $\nu = (m) + (n-m) = (n)$, 
$x$ is a regular unipotent element.  Thus by Proposition 2.5 (ii), 
$\dim \CO_{\Bla} = \dim G\uni + m$.
This implies that 
$\dim \wt\CX_{m,\unip}  = \dim \CO_{\Bla} \ge \dim \CX_{m,\unip}$, 
and so $\dim \CO_{\Bla} = \dim \CX_{m,\unip}$. 
Since $\CX_{m,\unip}$ is irreducible, the claim follows.

\begin{prop}  %%%  Proposition 2.7.
Let $\Bmu = (\mu^{(1)}, \mu^{(2)}) \in \CP_{n,2}$ be such that 
$|\mu^{(1)}| = m, |\mu^{(2)}| = n-m$. 
Then $\CO_{\Bmu} \subset \CX_{m,\unip}$, and for $(x,v) \in \CO_{\Bmu}$,  
we have
\begin{align*}
\dim \pi_1\iv(x,v) &= \frac{1}{2}(\dim \CX_{m,\unip} - \dim \CO_{\Bmu}) \\
                   &= \frac{1}{2}(\dim G\uni + m - \dim \CO_{\Bmu}).
\end{align*} 
\end{prop}
\begin{proof}
If we put $\Bla = ((m), (n-m))$, then $\Bmu \le \Bla$. Hence 
by Proposition 2.4 and by (2.6.3) 
we have $\CO_{\Bmu} \subset \ol\CO_{\Bla} = \CX_{m,\unip}$.
Here we have $\mu^{(1)} \le (m), \mu^{(2)} \le (n-m)$ with respect 
to the dominance order of partitions.  
Note that our map $\pi_1: \wt\CX_{m,\unip} \to \CX_{m,\unip}$ 
coincides with the map $\pi_{\Bla} : \wt\CF_{\Bla} \to \ol\CO_{\Bla}$
for $\Bla = ((m), (n-m))$
given in [AH, 3.2].  Hence by Proposition 4.4 (1) in [AH]
we have $\dim \pi_1\iv(x,v) = (\dim \CO_{\Bla} - \dim \CO_{\Bmu})/2$.
The proposition follows. 
\end{proof}

\para{2.8.}
We consider the diagram 
\begin{equation*}
\begin{CD}
T @<\a_0<< \wt\CY_m @> \p>> \CY_m,
\end{CD}
\end{equation*}
where $\a_0(x,v, gT) = p(g\iv xg)$ ($p: B \to T$ is the projection). 
Let $\CE$ be a tame local system on $T$.
Let $W_{\Bm,\CE}$ be the stabilizer of $\CE$ in $W_{\Bm}$.  Since 
$\p$ is a finite Galois covering with Galois group $W_{\Bm,\CE}$, 
$\p_*\a_0^*\CE$ is a semisimple 
local system on $\CY_m$, 
and is decomposed as 
\begin{equation*}
\tag{2.8.1}
\p_*\a_0^*\CE \simeq \bigoplus_{\r \in W_{\Bm,\CE}\wg}\r \otimes \CL_{\r},
\end{equation*} 
where $\CL_{\r}$ is a simple local system 
on $\CY_m$. 
\par
We consider the diagram 
\begin{equation*}
\begin{CD}
T @<\a<< \wt\CX_m @>\pi >>  \CX_m,
\end{CD}
\end{equation*}
where $\a(x,v, gB) = p(g\iv xg)$.  
We consider the complex $K_{m, T,\CE} = \pi_*\a^*\CE[\dim \CX_m]$.
In the case where $\CE = \Ql$, the structure of $K_{m,T,\CE}$ 
was described by Finkelberg and Ginzburg [FG, Corollary 5.4.2].  
The general case is done in a similar way, namely we have the following 
result.
\begin{thm} %%%% Theorem 2.9
$K_{m,T,\CE}$ is a semisimple perverse sheaf on $\CX_m$ equipped with 
$W_{\Bm,\CE}$-action, and is 
decomposed as
\begin{equation*}
\pi_*\a^*\CE[\dim \CX_m] \simeq \bigoplus_{\r \in W_{\Bm,\CE}\wg}
              \r \otimes \IC(\CX_m, \CL_{\r})[\dim \CX_m].
\end{equation*}
\end{thm}
\begin{proof}
$\CY_m$ is an open dense smooth subset of $\CX_m$, and 
$\pi\iv(\CY_m) \simeq \wt\CY_m$.   Hence 
the restriction of $\pi_*\a^*\CE$ on $\CY_m$ coincides with 
the local system $\p_*\a_0^*\CE$. 
In Proposition 5.4.1 (i) in [FG], it was proved that 
the map $\pi: \wt\CX_m \to \CX_m$ is small, in the sense 
that there exists a stratification $\CX_m = \coprod_{i \ge 0}X_i$
such that $X_0 = \CY_m$ and that $\dim \pi\iv(x,v) < (\codim X_i)/2$
for $(x,v) \in X_i$ if $i > 0$. 
Then $\pi_*\a^*\CE[\dim \CX_m]$ is a semisimple perverse sheaf, and
is isomorphic to the intersection 
cohomology $\IC(\CX_m, \p_*\a_0^*\CE)[\dim \CX_m]$.
Thus the theorem follows from (2.8.1).
\end{proof}

\para{2.10.}
We consider the case where $\CE = \Ql$.  Then 
$\pi_*\a^*\Ql = \pi_*\Ql$, and 
$\pi_*\Ql|_{\CX_{m,\unip}} \simeq (\pi_1)_*\Ql$. 
Moreover $W_{\Bm,\CE} = W_{\Bm}$. 
Since $\pi_1$ is semismall by [FG, Corollary 5.4.2], and 
$\pi_1$ is $G$-equivariant,  
$(\pi_1)_*\Ql[d'_m]$ ($d'_m = \dim \CX_{m,\unip}$) is 
a $G$-equivariant semisimple perverse sheaf on 
$\CX_{m,\unip}$.  
Since the number of $G$-orbits in $\CX_{m,\unip}$ is finite, 
and the isotropy subgroup of each orbit is trivial, 
we can write as 
\begin{equation*}
\tag{2.10.1}
(\pi_1)_*\Ql[d'_m] \simeq \bigoplus_{\CO \subset \CX_{m,\unip}}
                        \r_{\CO} \otimes A_{\CO}, 
\end{equation*}
where $A_{\CO} = \IC(\ol\CO, \Ql)[\dim \CO]$ and 
$\r_{\CO} = \Hom (A_{\CO}, (\pi_1)_*\Ql[d_m'])$ is a $W_{\Bm}$-module. 
Note that $W_{\Bm} \simeq S_m \times S_{n-m}$, and $W_{\Bm}\wg$ 
is parametrized by the set 
$\CP_{n,2}(m) = \{ (\la, \mu) \in \CP_{n,2} \mid |\la| = m, 
|\mu| = n-m\}$.
 We denote by $V_{(\la,\mu)} = V_{\la}\otimes V_{\mu}$
the (standard) irreducible $W_{\Bm}$-module corresponding to 
$(\la, \mu) \in \CP_{n,2}(m)$, 
We have the following result
as stated in 3.9 (11) in [FGT].
%%%%
\begin{thm}[Springer correspondence]   %%%%  Theorem 2.11
\begin{enumerate}
\item 
$(\pi_1)_*\Ql[d'_m]$ is a semisimple perverse sheaf on $\CX_{m,\unip}$ 
equipped with 
$W_{\Bm}$-action, and is decomposed as 
\begin{equation*}
(\pi_1)_*\Ql[d'_m] 
     \simeq \bigoplus_{\Bmu \in \CP_{n,2}(m)}V_{\Bmu} \otimes 
                \IC(\ol\CO_{\Bmu},\Ql)[\dim \CO_{\Bmu}]
\end{equation*}
\item  For each $\Bmu \in \CP_{n,2}(m)$, 
let $\CL_{\Bmu} = \CL_{\r}$ be the simple local system on $\CY_m$ 
corresponding to $\r = V_{\Bmu} \in W_{\Bm}\wg$. 
Then we have 
\begin{equation*}
\IC(\CX_m, \CL_{\Bmu})|_{\CX_{m,\unip}} \simeq \IC(\ol\CO_{\Bmu}, \Ql)[a],
\end{equation*}
where $a = \dim \CO_{\Bmu} - \dim \CX_{m, \unip}$. 
\end{enumerate}
\end{thm}

\begin{proof}
The following argument was inspired by [FGT, Theorem 1] and 
[AH, Proposition 4.6].
Put 
\begin{equation*}
\CG_m = \{ (x,v,W) \mid W \subset V, \dim W = m, v \in W, 
                 x(W) \subset W \}.
\end{equation*}
Then $\pi: \wt\CX_m \to \CX_m$ is factored as 
$\pi = \pi''\circ \pi'$, where 
$\pi': \wt\CX_m \to \CG_m, (x,v,gB) \mapsto (x,v, g(M_m))$, 
$\pi'': \CG_m \to \CX_m, (x,v,W) \mapsto (x,v)$. 
Put $\ol V = V/M_m$, and $G_1 = GL(M_m)$, $G_2 = GL(\ol V)$.
Let $B_1$ be a Borel subgroup of $G_1$ which is the stabilizer of the 
flag $M_1 \subset \cdots \subset M_m$ in $G_1$, and 
$B_2$ a Borel subgroup of $G_2$ which is the stabilizer 
of the flag $M_{m+1}/M_m \subset \cdots \subset V/M_m = \ol V$
in $G_2$.  Put $\wt G_i = \{ (x, gB_i) \in G_i \times G_i/B_i 
    \mid g\iv xg \in B_i\}$ 
and $p_i : \wt G_i \to G_i, (x, gB_i) \mapsto x$ for $i = 1,2$. 
We consider the commutative diagram

\begin{equation*}
\tag{2.11.1}
\begin{CD}
\wt G_1 \times \wt G_2 @<<<  \CZ_m  @>>> \wt\CX_m \\
    @V p_1\times p_2 VV                       @VVV          @VV\pi' V   \\
   G_1 \times G_2   @<s<<     \CH_m   @>q>>   \CG_m   \\
     @.                        @.              @VV\pi''V   \\
                    @.                @.      \CX_m, 
\end{CD}
\end{equation*}
where
\begin{equation*}
\begin{split}
\CH_m = \{ (x,v, W, &\f_1, \f_2) \mid (x,v,W) \in \CG_m, \\
                  &\f_1: W \isom M_m, \f_2: V/W \isom \ol V\},
\end{split}
\end{equation*}
and $q$ is the projection on the first three factors, $s$ 
is the map defined by 
\begin{equation*}
s: (x,v,W, \f_1, \f_2) \mapsto (\f_1(x|_W)\f_1\iv, \f_2(x|_{V/W})\f_2\iv).
\end{equation*}
$\CZ_m$ is the fibre product of $\wt\CX_m$ and $\CH_m$ over $\CG_m$, 
which is isomorphic to the fibre product of $\wt G_1 \times \wt G_2$ 
and $\CH_m$ over $G_1 \times G_2$. 
One can check that $q$ is a principal bundle with group 
$G_1 \times G_2$, and that $s$ is a locally trivial fibration 
with smooth connected fibre of dimension $n^2 + m$.
Let $K_i = (p_i)_*\Ql[\dim G_i]$ for $i = 1,2$.  Then it is well-known 
that 
\begin{align*}
\tag{2.11.2}
K_1 &\simeq \bigoplus_{\la \in \CP_m}V_{\la}\otimes 
            \IC(G_1, \CL_{\la})[\dim G_1] \\
K_2 &\simeq \bigoplus_{\mu \in \CP_{n-m}}V_{\mu}\otimes
             \IC(G_2, \CL_{\mu})[\dim G_2], 
\end{align*}
where $V_{\la}$ is the irreducible $S_m$-module, and $\CL_{\la}$ 
is a simple local system defined on $(G_1)\reg$ corresponding to 
$\la \in \CP_m$, and similarly for $G_2$. We consider the complex 
$K_1\boxtimes K_2$ on $G_1\times G_2$. By the property of the maps 
$s,q$ (see 0.2, 0.3), there exists a unique semisimple perverse 
sheaf $A_m$ on $\CG_m$
such that 
\begin{equation*}
\tag{2.11.3}
q^*(A_m)[m^2 + (n-m)^2] \simeq s^*(K_1 \boxtimes K_2)[n^2 + m].
\end{equation*}
Then by using the fact that $\pi$ is small, the following formula
can be proved by a standard argument.
\begin{equation*}
\tag{2.11.4}
\pi''_*(A_m) \simeq \pi_*\Ql[\dim \CX_m].
\end{equation*} 
Now $A_m$ can be decomposed by (2.11.2) as
\begin{equation*}
A_m \simeq \bigoplus_{\Bmu \in \CP_{n,2}(m)}V_{\Bmu} \otimes A_{\Bmu}, 
\end{equation*}
where $A_{\Bmu}$ is a simple perverse sheaf on $\CG_m$ 
defined as in (2.11.3), but  by replacing 
$K_i$ by $\IC(G_i, \CL_{\mu^{(i)}})[\dim G_i]$. 
By Theorem 2.9, we have 
\begin{equation*}
\pi_*\Ql[\dim \CX_m] \simeq \bigoplus_{\Bmu \in \CP_{n,2}(m)}
               V_{\Bmu} \otimes \IC(\CX_m, \CL_{\Bmu})[\dim \CX_m].
\end{equation*}
By considering the restriction of the diagram (2.11.1) to the
regular semisimple part (cf. [SS, Lemma 4.7]), (2.11.4) implies 
that 
\begin{equation*}
\tag{2.11.5}
(\pi'')_*A_{\Bmu} \simeq \IC(\CX_m, \CL_{\Bmu})[\dim \CX_m]
\end{equation*}
for any $\Bmu \in \CP_{n,2}(m)$.
\par
On the other hand, by considering the restriction of the diagram
(2.11.1) to the unipotent part, one gets the diagram 
\begin{equation*}
\begin{CD}
(G_1)\uni \times (G_2)\uni @<s_1<< \CH_{m,\unip} @>q_1>> \CG_{m,\unip}
            @>\pi_1''>> \CX_{m,\unip}.
\end{CD}
\end{equation*}
As in the above case, for 
$\Bmu = (\mu^{(1)},\mu^{(2)}) \in \CP_{n,2}(m)$ one can 
define a simple perverse sheaf $B_{\Bmu}$ on $\CG_{m,\unip}$ by 
the condition that 
\begin{equation*}
q_1^*B_{\Bmu}[m^2 + (n-m)^2] \simeq 
   s_1^*\bigl(\IC(\ol\CO_{\mu^{(1)}}, \Ql)[c_{\mu^{(1)}}]
      \boxtimes \IC(\ol\CO_{\mu^{(2)}},\Ql)[c_{\mu^{(2)}}])\bigr)
                      [n^2 + m],
\end{equation*}
where $c_{\la} = \dim \CO_{\la}$ for a partition $\la$. 
It was proved by [AH, Proposition 4.6] that
\begin{equation*}
\tag{2.11.6}
(\pi''_1)_* B_{\Bmu} \simeq \IC(\ol\CO_{\Bmu}, \Ql)[\dim \CO_{\Bmu}].
\end{equation*}
By the Springer correspondence for $GL_n$,
the restriction of $\IC(G_i, \CL_{\mu^{(i)}})$ on $(G_i)\uni$
coincides with $\IC(\ol\CO_{\mu^{(i)}},\Ql)$, up to shift, 
for $i = 1,2$. 
Thus by comparing (2.11.5) with (2.11.6), one sees that 
the restriction of $\IC(\CX_m, \CL_{\Bmu})$ on $\CX_{m,\unip}$
coincides with $\IC(\ol\CO_{\Bmu}, \Ql)$, up to shift. 
This proves (ii), up to shift.  Then (i) follows. 
The degree shift $a$ in (ii) can be easily computed from (i).  
\end{proof}

\remark{2.12.}
By applying Theorem 4.5 in [AH] to the case 
$\pi_{\Bla}: \wt\CF_{\Bla} \to \CX_{m,\unip}$ with $\Bla = ((m), (n-m))$, 
one gets 
\begin{equation*}
(\pi_1)_*\Ql[d'_m] \simeq \bigoplus_{\Bmu\in \CP_{n,2}(m)}
           V_{\Bmu} \otimes \IC(\ol\CO_{\Bmu}, \Ql)[\dim \CO_{\Bmu}],
\end{equation*}
where $V_{\Bmu}$ is regarded as a multiplicity space ignoring the
$W_{\Bm}$-action. By comparing this with (2.10.1), we obtain 
\begin{equation*}
(\pi_1)_*\Ql[d'_m] \simeq \bigoplus_{\Bmu \in \CP_{n,2}(m)}
           \r_{\CO_{\Bmu}} \otimes \IC(\ol\CO_{\Bmu}, \Ql)[\dim \CO_{\Bmu}],
\end{equation*}
where $\dim \r_{\CO_{\Bmu}} = \dim V_{\Bmu}$. 
However, in order to show that $\r_{\CO_{\Bmu}} = V_{\Bmu}$, 
one needs some additional arguments.  

\para{2.13.}
For $z = (x,v) \in \CX_m$, put
\begin{equation*} 
\CB^{(m)}_z = \{ gB \in G/B \mid x\iv gx \in B, g\iv v \in M_m\}.
\end{equation*}
Then $\pi\iv(z) \simeq \CB^{(m)}_z$. Since 
$\CH^i_z(\pi_*\Ql) \simeq H^i(\CB^{(m)}_z, \Ql)$,  
$H^i(\CB^{(m)}_z,\Ql)$ is equipped with the $W_{\Bm}$-module structure,
called the Springer representation of $W_{\Bm}$.   
By taking the stalk in the formula in Theorem 2.11 (i), we have, 
for $z \in \CX_{m,\unip}$,  
\begin{equation*}
H^i(\CB^{(m)}_z,\Ql) \simeq \bigoplus_{\Bmu \in \CP_{n,2}(m)}
     V_{\Bmu} \otimes \CH_z^{i - \dim \CX_{m,\unip} + \dim \CO_{\Bmu}}
                           \IC(\ol\CO_{\Bmu},\Ql).
\end{equation*}
By Proposition 2.7, $\dim \CB^{(m)}_z = 
(\dim \CX_{m,\unip} - \dim \CO_{\Bmu})/2$ for $z \in \CO_{\Bmu}$ with 
$\Bmu \in \CP_{n,2}(m)$. . 
Put $d_{\Bmu} = \dim \CB^{(m)}_z$. It follows from the above formula 
that 
\begin{equation*}
H^{2d_{\Bmu}}(\CB^{(m)}_z,\Ql) \supset V_{\Bmu} \otimes 
                \CH^0_z\IC(\ol\CO_{\Bmu}, \Ql) \simeq V_{\Bmu}
\end{equation*}
since $\CH^0_{(x,v)}\IC(\ol\CO_{\Bmu},\Ql) \simeq \Ql$.
Let $c_{\Bmu}$ be the number of irreducible components of $\CB_z^{(m)}$, for 
$z \in \CO_{\Bmu}$, of maximum dimension.  By the above formula, we have 
$\dim V_{\Bmu} \le c_{\Bmu}$.  By a similar argument as in 
[SS, Lemma 3.5 (iii)], one can show 
\begin{equation*}
\sum_{\Bmu \in \CP_{n,2}(m)}c_{\Bmu}^2 \le |W_{\Bm}|.
\end{equation*}
It follows that $\dim V_{\Bmu} = c_{\Bmu}$.   
Thus we have an analogue of the original Springer correspondence.
\par\medskip\noindent
(2.13.1) \ The top cohomology $H^{2d_{\Bmu}}(\CB^{(m)}_z, \Ql)$ 
gives rise to an irreducible $W_{\Bm}$-module $V_{\Bmu}$, and 
the map $z \mapsto V_{\Bmu}$ gives a bijective correspondence
\begin{equation*}
\CX_{m,\unip}/G \lra (S_m \times S_{n-m})\wg.
\end{equation*}
\par
We consider the special case where $\Bmu = ((1^m),(1^{n-m}))$. 
Take $z = (x,v) \in \CO_{\Bmu}$.  We may assume that 
$z \in U \times M_m$. 
Then $(x-1)|_{M_m} = 0, (x-1)|_{\ol V} =0$ and $(x-1)V = M_m$. 
It follows that $\CB^{(m)}_z$ turns out to be 
a closed subvariety of $G/B$  consisting of flags  
$\{ (V_1 \subset \cdots \subset V_{n-1} \subset V) \}$
such that $V_m = M_m$, which is isomorphic to    
$\CB_1 \times \CB_2$, where $\CB_i = G_i/B_i$ 
is the flag variety of $G_i$ under the notation in the proof 
of Theorem 2.11.
Now $H^{\bullet}(\CB^{(m)}_z, \Ql)$ is a graded $W_{\Bm}$-module, 
and $H^{\bullet}(\CB_i, \Ql)$ are graded $S_m, S_{n-m}$-modules
for $i = 1,2$.  Then we have an isomorphism of graded 
$S_m \times S_{n-m}$-modules,
\begin{equation*}
\tag{2.13.2}
H^{\bullet}(\CB^{(m)}_z, \Ql) \simeq 
          H^{\bullet}(\CB_1,\Ql) \otimes H^{\bullet}(\CB_2, \Ql).
\end{equation*}

\para{2.14.}
We fix $s \in T$ and an integer $m \ge 0$.  We consider a variety 
\begin{equation*}
\begin{split}
\CZ = \{ (x,v, &gB, g'B) \in G \times V \times G/B \times G/B  \\
         & \mid g\iv xg \in sU, g\iv v \in M_m, 
                {g'}\iv xg' \in sU, {g'}\iv v \in M_m \}.  
\end{split}
\end{equation*}
We consider a partition $G/B \times G/B = \coprod_{w \in S_n}X_w$ into 
$G$-orbits, and put $\CZ _w = p\iv (X_w)$, where $p: \CZ \to G/B \times G/B$
is the projection onto the last two factors. 
Recall that $\nu_G = \dim U$. 
\par
The following result is an analogue of Proposition 1.14 to the enhanced case.

\begin{prop}  %%%%  Prop. 2.15
Let $\CO$ be a $G$-orbit in $\CX$, and put $c = \dim \CO$.
\begin{enumerate}
\item
We have $\dim (\CO \cap (sU \times M_m)) \le (c + m)/2$.
\item
For $z = (x,v) \in \CO$, we have
\begin{equation*}
\dim \{gB \in G/B \mid g\iv xg \in sU, g\iv v \in M_m\} \le \nu_G - (c - m)/2.
\end{equation*}
\item
$\dim \CZ_w \le 2\nu_G + m$ for all $w \in S_n$.  Hence 
$\dim \CZ \le 2\nu_G + m$. 

\end{enumerate}
\end{prop}

\begin{proof}
First we show (iii). 
The projection $p$ is $G$-equivariant. A representative of the $G$-orbit $X_w$ 
is given by $(B, wB)$.  Hence in order to show (iii), it is enough to see that 

\begin{equation*}
\tag{2.15.1}
\begin{split}
\dim \{ (x,v ) \in sU \times &M_m \mid (\dw\iv x\dw, \dw\iv v) \in sU \times M_m \} \\ 
               &\le 2\nu_G + m - \dim X_w 
\end{split}
\end{equation*}
for each $w \in S_n$, where $\dw$ is a representative of $w$ in $G$.  
Then $x \in B \cap {}^wB$ can be written as 
$x = su = s'u'$ with $u \in U$, $s' = wsw\iv \in T$, $u' \in {}^{w}U$.
Hence $s' = s$ and $u = u' \in U \cap {}^wU$. 
Moreover $v \in M_m \cap w(M_m)$.  
It follows that the left hand side of (2.15.1) is less than or equal to 
$\dim  (U \cap {}^wU) + \dim (M_m  \cap w(M_m)) \le \dim (U \cap {}^wU) + m$. 
Since $\dim X_w = 2\nu_G - \dim (U \cap {}^wU)$, we obtain (2.15.1).  
Thus (iii) follows.  
\par
The proof of (ii) and (i) is completely similar to the proof of 
(ii) and (i) of Proposition 1.14.
\end{proof}

\para{2.16.}
Let $\CO$ be a $G$-orbit containing $(su,v_0) \in \CX$, where 
$su = us$, $s$: semisimple, $u$: unipotent. 
Let $L = Z_G(s)$.  Then $L \simeq \prod_{i = 1}^t G_i$, 
and $L \times V \simeq \prod_i (G_i \times V_i)$, where  
$V = V_1\oplus \cdots \oplus V_t$ is a decomposition of $V$ to eigenspaces of 
$s$, and $G_i = GL(V_i)$ for each $i$.  
Let $\CO_0$ be the $L$-orbit containing $(u,v_0)$ in $L \times V$.   
Then under the decomposition of $L \times V$ as above, $\CO_0$ is isomorphic to
$\prod_i\CO_i$, where $\CO_i$ is a $G_i$-orbit in $(G_i)\uni \times V_i$. 
Under the natural parametrization in Proposition 2.2, one can 
write $\CO_i = \CO_{\Bla_i}$ for $\Bla_i \in \CP_{n_i,2}$, where $n_i = \dim V_i$.
Recall that $m(\Bla) = |\la^{(1)}|$ for 
$\Bla = (\la^{(1)}, \la^{(2)}) \in \CP_{n,2}$ ([SS, 5.3]).
For each $G$-orbit $\CO$ in $\CX$ as above, we define an integer $\mu(\CO)$ by 
$\mu(\CO) = \sum_im(\Bla_i)$. 
As a corollary to Proposition 2.15, we have the following.

\begin{cor}  %%%%  Cor. 2.17.
Let $\CO$ be a $G$-orbit in $\CX$ containing $(su,v_0)$ as in 2.16.  
Assume that $s \in T, u \in U$ and that $v_0 \in M_m$ for $m = \mu(\CO)$.  Then we have
\begin{equation*}
\dim (\CO \cap (sU \times M_m)) = (\dim \CO + m)/2.
\end{equation*}
\end{cor}

\begin{proof}
The proof is similar to the proof of Corollary 1.15.
Let $\CX_{\CO}$ be the variety defined similar to (1.14.2).  
Let $c = \dim \CO$.  
By Proposition 2.15 (ii),
we have $\dim Y_z \le \nu_G - (c - m)/2$ for any $z = (x,v) \in \CO$, 
where $Y_z$ is the variety appearing there. 
We want to show that 
\begin{equation*}
\tag{2.17.1} 
\dim Y_z = \nu_G - (c - m)/2.
\end{equation*}
Let $\pi^L_1: (\wt\CX_L)_{m,\unip} \to (\CX_L)_{m,\unip}$ be the map 
$\pi_1$ defined in 2.6 with respect to $\CX_L = L \times V$.
Since $\CX_L \simeq \prod_i(GL_{n_i} \times V_i)$, Proposition 2.7 implies 
that 
$\dim (\pi_1^L)\iv(u,v_0) = \nu_G - (c - m)/2$.  
Here $Y_z$ contains a variety isomorphic to $(\pi_1^L)\iv(u,v_0)$ for 
$z = (su, v_0)$.  
This implies that $\dim Y_z \ge \nu_G - (c - m)/2$.  Hence (2.17.1) holds.  
Now the corollary follows by a similar argument as in the proof of 
Proposition 2.15 (i).  
\end{proof}
\par\bigskip

\section{Definition of character sheaves}

%%%%
%%%%
\para{3.1.}
Let $\wt V$ be a vector space over $\Bk$ of the form 
$\wt V = V_0 \oplus \bigoplus_{i=1}^k(V_i \oplus V_i)$, where 
$\dim V_0 = 2n_0$ and $\dim V_i = n_i$ for $ i \ge 1$.
We consider a subgroup 
$G = G_0 \times \prod_{i=1}^k(G_i \times G_i)$ of 
$GL(\wt V)$, where $G_i = GL(V_i)$
for $i \ge 0$.
Let $\th$ be an involutive automorphism of $G$, where  
$\th$ preserves $G_0$ such that $G_0^{\th} \simeq Sp_{2n_0}$ and that 
$\th$ acts as a permutation on the factor $G_i \times G_i$ for each $i$. 
 Let $B$ be a $\th$-stable Borel 
subgroup of $G$ containing a $\th$-stable maximal torus $T$. We denote by
$U$ the unipotent radical of $B$. 
Let $V = V_I$ be a subspace of $\wt V$ of the form 
$V = \bigoplus_{i \in I}V_i$, where $I$ is a subset of 
$[0,k]$.  
Let $H = G^{\th}$. 
We have  
$H \simeq Sp_{2n_0} \times \prod_{i=1}^kGL_{n_i}$.  
Put $G^{\io\th} = \{ g \in G \mid \th(g) = g\iv\}$, which is
isomorphic to $G/H$.
Let 
$\CX = G^{\io\th} \times V \simeq G/H \times V$, on which 
$H$ acts diagonally.  
We call $\CX$ an exotic symmetric space.
We have
\begin{equation*} 
\tag{3.1.1}
\CX \simeq (GL_{2n_0}/Sp_{2n_0} \times V^{\ve}_0) \times 
    \prod_{i =1}^k (GL_{n_i} \times V^{\ve}_i),  
\end{equation*}  
where $V_i^{\ve} = V_i$ or $\{0\}$ according to 
 $i \in I$ or not. 
We say that $\CX$ is of pure exotic type 
if $\CX \simeq GL_{2n_0}/Sp_{2n_0} \times V_0$, 
and of enhanced type or of mirabolic type if 
$\CX \simeq GL_{n_i} \times V_i$.    
Note that the pure exotic type $\CX$ was studied by Kato [Ka1], [Ka2] 
in connection with the exotic nilpotent cone, and by [SS], [SS2] in 
connection with character sheaves, and 
the enhanced type $\CX$ was studied by
Achar and Henderson [AH],  
Finkelberg, Ginzburg and Travkin [FGT] in connection with 
the enhanced nilpotent cone and mirabolic character sheaves. 

\para{3.2.}
The notion of character sheaves was generalized to the case of 
symmetric spaces by Ginzburg [Gi].  Grojnowski [Gr] studied 
the character sheaves on $\CX = GL_{2n}/Sp_{2n}$ in connection 
with the representation theory of the associated Hecke algebra 
due to Bannai, Kawanaka and Song [BKS].  The character sheaves 
on the mirabolic type $\CX$ were introduced by [FG], [FGT]. 
In [HT], Henderson and Trapa proposed to define (unipotent) 
character sheaves on the pure exotic type $\CX$.  In what follows, 
we shall define character sheaves on $\CX$ based on those ideas. 
\par 
Put $\CB = G/B, \wt\CB = G/U$.  Then $T$ acts freely on 
$\wt\CB$ by $(t, gU) \mapsto gt\iv U$, and the natural map 
$q_0: \wt\CB \to \CB$ is a principal $T$-bundle.  
We consider a diagram 
\begin{equation*}
\tag{3.2.1}
\begin{CD}
(\wt\CB \times V) \times \wt\CB @<r<< (G \times V) \times \wt\CB 
    @>q>>  (G \times V) \times \CB @>p>> G \times V,
\end{CD}
\end{equation*}  
where $p$ is the projection on the first two factors, 
$q = \id \times q_0$, and $r$ is defined by 
$r(x,v,gU) = (xgU, v, gU)$.  Then $p$ is proper, $q$ is a principal
$T$-bundle and $r$ is smooth with fibre isomorphic to 
$U$.  Following [HT], we define an action of $H \times H$ on 
these varieties; for $(h_1,h_2) \in H \times H$, 
\par\medskip
on $G \times V$ by 
$(x,v) \mapsto (h_1xh_2\iv, h_1v)$,  
\par
on $(G \times V) \times \CB$ by 
      $(x,v,gB) \mapsto (h_1xh_2\iv, h_1v, h_2gB)$, 
\par
on $(G \times V) \times \wt\CB$ by $(x,v, gU) \mapsto 
                    (h_1xh_2\iv, h_1v, h_2gU)$, 
\par
on $(\wt\CB \times V) \times \wt\CB$ by 
$(g'U,v,gU) \mapsto (h_1g'U,h_1v,h_2gU)$. 
\par\medskip
\noindent 
Then the maps $p,q,r$ are $H \times H$-equivariant.
\par
We define an action of $T$ on $\wt\CB \times V$ so that 
$T$ acts trivially on $V$ (and acts on $\wt\CB$ as before).
Let $\CL$ be a tame local system on $T$.  
We define an action of $T$ on $(G \times V) \times \wt \CB$ so that
$T$ acts on $\wt\CB$ as before, and acts trivially on other factors. 
Then $r$ is equivariant with respect to the action of $T$ on 
$(G \times V) \times \wt\CB$, and the action of $\vD(T)$ on 
$(\wt \CB\times V) \times \wt\CB$, where $\vD(T)$ is the diagonal subgroup 
of $T \times T$. 
\par
We now consider the category 
$\DD_{\CL \boxtimes\CL\iv}(\wt\CB \times V \times \wt\CB)$.   
By the above observation, $r^*$ gives a functor from 
$\DD_{\CL\boxtimes \CL\iv}(\wt\CB \times V \times \wt\CB)$ to
$\DD^T(G \times V \times \wt\CB)$. 
Thus one can define a functor
$\Ch : \DD_{\CL\boxtimes\CL\iv}(\wt\CB \times V \times \wt\CB)
  \to \DD(G \times V)$
by $\Ch(K) = p_!q_{\flat}r^*K[\dim U - \dim T]$ for 
$K \in \DD_{\CL\boxtimes\CL\iv}(\wt\CB \times V \times \wt\CB)$.
Since the diagram (3.2.1) is compatible with the action of 
$H \times H$, $\Ch$ lifts uniquely to a functor 
\begin{equation*}
\tag{3.2.2}
\Ch : \DD_{\CL \boxtimes\CL\iv}^{H \times H}(\wt\CB \times V \times \wt\CB) 
\to \DD^{H \times H}(G \times V).
\end{equation*}
\par
Let $\vT: G \times V \to G^{\io\th} \times V = \CX$ 
be a map defined by $(x,v) \mapsto (x\th(x)\iv, v)$. Then 
$\vT$ is smooth with connected fibre, where the fibre is isomorphic 
to $H$.
Thus $\wt\vT = \vT^*[\dim H]$ gives rise to a fully faithful functor 
$\CM^H(\CX) \to \CM^{H \times H}(G \times V)$.
For a tame local system $\CL$ on $T$, we define $\wh\CX_{\CL}$
as the set of isomorphism classes of $H$-equivariant simple perverse
sheaves $A$ such that $\wt\vT A$ is isomorphic to a perverse constituent 
of $\Ch K$ for some 
$K \in \DD^{H \times H}_{\CL \boxtimes\CL\iv}(\wt\CB \times V \times \wt\CB)$.
We define a set $\wh\CX$ as a union $\bigcup_{\CL}\wh\CX_{\CL}$, 
where $\CL$ runs over the set of tame local systems on $T$. 
An $H$-equivariant simple perverse sheaf $A$ on $G^{\io\th} \times V$ 
is called a character sheaf if $A \in \wh\CX$.

\para{3.3.}
Assume that $\CL$ is a tame local system on $T$ such that 
$\CL^{\otimes m} \simeq \Ql$.
Then $\CL$ is $T$-equivariant with respect to the action of 
$T$ on itself by $t_0: t \mapsto t_0^mt$.  
Let $Z(G) = Z^0(G)$ be the center of $G$.
We define an action of $Z(G)$ on $(\wt\CB \times V) \times \wt\CB$ 
by $z : (g_1U, v, g_2U) \mapsto (z^mg_1U, v, g_2U)$.  We also define 
an action of $Z(G)$ on $G$ by $z : g \mapsto z^mg$, and define 
actions on $(G\times V) \times \wt\CB$,  etc. 
so that $Z(G)$ acts trivially on factors different from $G$.  Then 
all the maps in (3.2.1) are $Z(G)$-equivariant, which commute with 
the action of $H \times H$.    
Since $Z(G)$ is a subtorus of $T$, $\CL$ is $Z(G)$-equivariant.  
Hence any object in 
$\DD^{H \times H}_{\CL\boxtimes\CL\iv}(\wt\CB \times V \times \wt\CB)$
is $Z(G) \times (H \times H)$-equivariant, and we have a refinement 
of (3.2.2), 
\begin{equation*}
\tag{3.3.1}
\Ch : \DD^{H \times H}_{\CL\boxtimes\CL\iv}(\wt\CB \times V \times \wt\CB)
       \to \DD^{Z(G) \times H \times H}(G \times V).
\end{equation*}

\par
We define an action of $Z(G)^{\io\th} \times H$ on $\CX$ by 
$(z, g) : (x,v) \mapsto (z^mgxg\iv, gv)$. (Note that $Z(G)^{\io\th}$ 
is a subtorus of $Z(G)$).  Then $\vT$ is 
$Z(G)^{\io\th} \times H$-equivariant, and we have the following.
\par\medskip\noindent
(3.3.2) \ Any character sheaf $A \in \wh \CX_{\CL}$ 
is $Z(G)^{\io\th} \times H$-equivariant.

\para{3.4.}
Assume that $G = T$ is a torus 
such that $T \simeq T_0 \times T_0$ and $\th$ is the permutation of 
two factors.  
Thus $T$ is a maximal torus of $GL(\wt V)$, where 
$\wt V \simeq V \oplus V$ for a $T_0$-stable subspace $V$.  
We consider the variety 
$\CX = \CX_{T,V} = T^{\io\th} \times V$ with diagonal $T^{\th}$-action.  
Let $\CE$ be a tame local system on 
$T^{\io\th}$, and put $\CL = \vT_T^*\CE$, where 
$\vT_T: T \to T^{\io\th}, t \mapsto t\th(t)\iv$.
Then $\CL$ is a tame local system on $T$.  
For each $T^{\th}$-stable subspace $V_1$ of $V$, 
let $\wt\CE_{V_1}$ be the constructible sheaf on 
$\CX$ which is the extension by 
zero of the local system $\CE \boxtimes \Ql$ on 
$\CX_{T, V_1} = T^{\io\th}\times V_1$.  
We note that
\par\medskip
\begin{equation*}
\tag{3.4.1}
\wh\CX_{\CL} = 
\{ \wt\CE_{V_1}[\dim \CX_{T,V_1}]  \mid \vT_T^*\CE \simeq \CL, 
                 V_1 \subset V : \text{$T^{\th}$-stable} \}.
\end{equation*}
\par\medskip

In fact, in the case where $G = T$, $\wt\CB \times V$ 
coincides with $T \times V$ on which $T^{\th}$ acts by 
the left multiplication on $T$ and by the action of $T^{\th} \simeq T_0$ 
on $V$. 
Similarly, $T^{\th}$ acts on $\wt\CB = T$ by the left multiplication.
Take $A \in \wh\CX_{\CL}$. By definition, $\wt\vT A$ is a perverse constituent
of $\Ch K$ for some 
$K \in \DD_{\CL\boxtimes\CL\iv}^{T^{\th} \times T^{\th}}(T \times V \times T)$.
By (0.4.1), we may assume that $K$ is a $T^{\th} \times T^{\th}$-equivariant 
simple perverse sheaf on $(T \times V) \times T$. Thus $K$ can be written as 
$K = K_1 \boxtimes K_2$, where $K_1$ (resp. $K_2$) is a simple perverse sheaf 
contained in $\DD_{\CL}^{T^{\th}}(T \times V)$ (resp. in 
$\DD_{\CL\iv}^{T^{\th}}(T)$).  
Since $\CL$ is a $T^{\th}$-equivariant sheaf on $T$, $K_2$ coincides with 
$\CL\iv[\dim T]$. Since $K_1 \in \DD_{\CL}^{T^{\th}}(T \times V)$, the support of 
$K_1$ is given by $T \times V_1$ for a $T^{\th}$-stable subspace $V_1$ of $V$.
It follows that $K_1$ is a constructible sheaf, up to shift,  obtained 
by the extension by zero from the local system  $\CL\boxtimes \Ql$ on $T \times V_1$.
By applying $\Ch$ following (3.2.1), we see that 
$\Ch K$ is an extension by zero of the local system 
$\CL \boxtimes \Ql$ on $T \times V_1$, up to shift. 
(3.4.1) follows from this. 
\par
If we replace $\CX$ by $\CX' = T^{\io\th} \times V'$, where 
$V'$ is a $T^{\th}$-stable subspace of $V$, 
$\CX'$ also fits to the definition of (3.1.1), and 
the set of 
character sheaves $\wh \CX'_{\CL}$ is given by all those 
$\wt\CE_{V_1}[\dim \CX_{T,V_1}]$ such that $V_1 \subset V'$.     
The variety $\CX$ in (3.1.1) for the case $G =T$ coincides with 
one of those $\CX'$.   
\par
For $\CX = \CX_{T,V}$, we define a subset $\wh\CX^0$ of $\wh\CX$ by
\begin{equation*}
\tag{3.4.2}
\wh\CX^0 = \{ A \in \wh\CX \mid \supp A = T^{\io\th} \times \{ 0\}\}.
\end{equation*}

\bigskip
\section{Induction}

\para{4.1.}
Induction functors and restrictions functors for character 
sheaves in the case of symmetric spaces are introduced by 
Grojnowski [Gr, Section 2]. Henderson reformulated the induction functors 
in [H1].
In this section, we generalize the induction functors so that it fits 
to our setting.
\par
Let $P = LU_P$ be a $\th$-stable parabolic subgroup of $G$ containing 
$B$, where $L$ is a $\th$-stable Levi subgroup of $P$ containing $T$
and $U_P$ is the unipotent radical of $P$.
Then $L$ is a group of the same type as $G$ given in 3.1.
Let $V_L$ be a $L^{\th}$-module such that 
$\CX_L = L^{\io\th} \times V_L$ gives rise to a similar variety as $\CX$.
We assume that $V_L = V'/V_L^-$, where $V'$ is a $P^{\th}$-stable subspace of 
$V$ and $V_L^-$ is a  $P^{\th}$-stable subspace of $V'$.  Hence the action of 
$L^{\th}$ on $V_L$ can be extended to the action of $P^{\th}$.   
\para{4.2.} 
Put 
\begin{equation*}
\wt\CX^P = \{ (x,v, gP^{\th}) \in \CX \times H/P^{\th}
           \mid g\iv xg \in P^{\io\th}, g\iv v \in V' \}. 
\end{equation*}
We consider a diagram 
\begin{equation*}
\tag{4.2.1}
\begin{CD}
\CX_L  @<\t << H \times P^{\io\th} \times V' 
     @>\s >> \wt\CX^P  @>\r >> \CX,
\end{CD}
\end{equation*}
where $\t$ is the map $(g, p, v) \mapsto (\ol p, \ol v)$
($p \mapsto \ol p$ is the natural projection 
$P^{\io\th} \to L^{\io\th}$, and $\ol v$ is the projection $V' \to V_L$). 
$\s$ is the map $(g,p,v) \mapsto (gpg\iv,  gv, gP^{\th})$, and 
$\r$ is the projection to the first and second factors.  
Note that $\wt\CX^P$ is canonically isomorphic to 
$H \times^{P^{\th}}(P^{\io\th} \times V')$, 
and $\s$ gives rise to a principal $P^{\th}$-bundle.  
Now 
\par\medskip
$H\times L^{\th}$ acts on $\CX_L$ by $(h,\ell) :
         (x,v) \mapsto (\ell x\ell\iv, \ell v)$, 
\par
$H \times P^{\th}$ acts on $H \times P^{\io\th} \times V'$ 
by $(h, p): (g, x, v) \mapsto (hgp\iv, p xp\iv, p v)$,  
\par
$H$ acts on $\wt\CX^P$ by $h: (x,v,gP^{\th}) \mapsto 
               (hxh\iv, hv, hgP^{\th})$.
\par\medskip
\noindent
Here $\s, \r$ are 
$H$-equivariant. 
The action of $L^{\th}$ on $L^{\io\th}$ can be lifted to the 
action of $P^{\th}$ by the projection $P^{\th} \to L^{\th}$.  We consider 
the diagonal action of $P^{\th}$ on $\CX_L = L^{\io\th} \times V_L$. 
Then $\t$ is $H \times P^{\th}$-equivariant. 
We define a functor 
$\Ind = \Ind_{\CX_L,V', P}^{\CX}: 
   \DD^{P^{\th}}(\CX_L) \to \DD^{H}(\CX)$
by 
\begin{equation*}
\tag{4.2.2}
\Ind K = \r_!\s_{\flat}\t^*K[a] \qquad (K \in \DD^{P^{\th}}(\CX_L)),
\end{equation*}
with $a = \dim U_P + \dim V_L^-$. 
(The funtor $\Ind$ depnends on the choice of $V'$ even if $\CX_L, P$ are fixed.
However, we often write it as $\Ind_{\CX_L, P}^{\CX}$ if there is no fear
of confusion.) 
We note that
\begin{align*}
\dim(\text{fibre of }\t) - \dim(\text{fibre of }\s) 
  &=  \dim H + \dim U_P^{\io\th}  + \dim V_L^- - \dim P^{\th} \\
  &=  \dim U_P + \dim V_L^-.
\end{align*}
We also define a functor $\Ind^{\bullet} K = \r_!\s_{\flat}\tau^*K$
by removing the degree shift.  
\par
Let $P \subset Q$ be two $\th$-stable parabolic subgroups 
of $G$ containing $B$, with $\th$-stable Levi subgroups 
$L  \subset M$ containing $T$.  Then $P_M = P \cap M$ is 
a $\th$-stable parabolic subgroup of $M$ with Levi subgroup $L$.
Put $V_M = V'/V_M^-$ for a $Q^{\th}$-subspace $V'$ of $V$, and let 
$V_L$ be the quotient of $V_M'$, where $V_M'$ is a 
a $P_M^{\th}$-stable subspace of $V_M$.  Then $V_M' = V''/V^-_M$ for a
$P^{\th}$-stable subspace $V''$ of $V'$. 
We consider the varieties $\CX_L = L^{\io\th} \times V_L$, 
$\CX_M = M^{\io\th} \times V_M$ as in the case of $\CX$. 
One can define functors 
$\Ind_{\CX_L, V_M', P_M}^{\CX_M} :
         \DD^{P_M^{\th}}(\CX_L) \to \DD^{M^{\th}}(\CX_M)$, 
and $\Ind_{\CX_M, V', Q}^{\CX} : \DD^{Q^{\th}}(\CX_M) \to \DD^H(\CX)$.
Since $V_L$ is a quotient of $V''$, 
$\Ind_{\CX_L, V'', P}^{\CX} : \DD^{P^{\th}}(\CX_L) \to \DD^H(\CX)$
can be defined. 
We obtain the following transitivity of inductions.
The proof is similar to the proof of the transitivity 
of induction for character sheaves [L3, Proposition 4.2], 
and we omit it. 
%%%
\begin{prop}  %%%  Prop. 4.3
Let $K \in \DD^{P^{\th}}(\CX_L)$ be such that 
$\Ind_{\CX_L, V_M', P_M}^{\CX_M}K \in \DD^{Q^{\th}}(\CX_M)$. 
Then there is an isomorphism 
\begin{equation*}
{\Ind^{\bullet}}_{\CX_L, V'',P}^{\CX}K
    \simeq ({\Ind^{\bullet}}_{\CX_M,V',  Q}^{\CX}
                \circ {\Ind^{\bullet}}_{\CX_L,V'_M, P_M}^{\CX_M})K.
\end{equation*}
In particular,  a similar transitivity holds for 
$\Ind_{\CX_L,P}^{\CX}$.
\end{prop}
%%%%
%%%%
\para{4.4.} We consider a diagram 
\begin{equation*}
\tag{4.4.1}
\begin{CD}
L \times V_L @<\t'<< H \times P \times V' \times H @>\s'>> 
                       \wt\CZ^P  @>\r'>> G\times V \\
            @V\vT_L VV    @VV\vT_1V      @VV\vT_2V  @VV\vT V  \\
\CX_L  @<\t<<  H \times P^{\io\th} \times V' @>\s>> \wt\CX^P  @>\r>>  \CX,
\end{CD}
\end{equation*}
where $P^{\th} \times P^{\th}$ acts on $H\times P \times V' \times H$
by 
\begin{equation*}
(p_1, p_2) : (g_1, x, v, g_2) \mapsto 
     (g_1p_1\iv, p_1xp_2\iv, p_1 v,  p_2g_2),
\end{equation*} 
and $\wt\CZ^P = H \times^{P^{\th}}(P \times V')\times^{P^{\th}}H$ 
is the quotient of $H \times P \times V' \times H$ 
by $P^{\th} \times P^{\th}$.  Here the maps are defined by 

\par\medskip
$\t': (g_1, x, v, g_2) \mapsto (\ol x, \ol v)$, where $\ol x$ is the image 
of $x$ under $P \to L$,
\par
$\s'$ : the quotient map by $P^{\th} \times P^{\th}$,
\par
$\r' : g_1*(x,v)*g_2 \mapsto (g_1xg_2, g_1v)$, 
\par
$\vT_1: (g_1,x,v,g_2) \mapsto (g_1, x\th(x)\iv, v)$, 
\par
$\vT_2: g_1*(x,v)*g_2 \mapsto 
    g_1*(x\th(x)\iv, v)$,
\par\medskip
\noindent
where we denote by $g*(x,v) \in \wt\CX^P$ the image of 
$(g,x,v) \in H \times P^{\io\th} \times V'$ 
by the quotient by $P^{\th}$, and we use a similar notation 
also for the quotient 
of $H \times P \times V' \times H$ 
by $P^{\th} \times P^{\th}$. 
\par
Now $H \times H$ 
acts on $H \times P \times V' \times H$ by 
$(h_1, h_2) : 
(g_1, x, v, g_2) \mapsto 
   (h_1g_1, x, v, g_2h_2\iv)$,
and this action commutes with the action of $P^{\th} \times P^{\th}$. 
The action of $H \times H$ on $\wt\CZ^P$ is defined similarly.  
Then $\s', \r'$ are 
$H\times H$-equivariant, 
and $\t'$ is $(H \times H) \times (P^{\th}\times P^{\th})$-equivariant
($L^{\th} \times L^{\th}$ acts on $L$ as in the case 
of $G$, which can be lifted to the action of $P^{\th} \times P^{\th}$ by 
the projection $P^{\th} \to L^{\th}$, and the action of $L^{\th}$ on 
$V_L$ can be lifted to the action of $P^{\th}$ as in 4.2),  
and $H \times H$ acts trivially on $L \times V_L$. 
Then $\vT_L, \vT_1, \vT_2, \vT$ are compatible 
with those actions on the upper row and the previous actions 
on the lower row. 
\par 
Since $\s'$ is a principal $(P^{\th} \times P^{\th})$-bundle,
one can define in a similar way as the functor $\Ind$, a functor 
$\wt\Ind = \wt\Ind_{L \times V_L, P}^{G \times V}
  : \DD^{P^{\th} \times P^{\th}}(L \times V_L) \to 
       \DD^{H \times H}(G \times V)$ by 
\begin{equation*}
\tag{4.4.2}
\wt\Ind \, K = \r'_!\s'_{\flat}(\t')^*K[a'] \qquad 
     (K \in  \DD^{P^{\th} \times P^{\th}}(L \times V_L)),
\end{equation*}
where 
$a' = 2(\dim H - \dim P^{\th}) + \dim U_P + \dim V_L^-$.
Here we have
\begin{equation*}
\begin{split}
\dim (&\text{fibre of }  \t') - \dim (\text{fibre of }\s')  \\ 
          &= 2(\dim H - \dim P^{\th}) + \dim U_P + \dim V_L^-.   
\end{split}
\end{equation*}
\par
One can check that the rightmost square in the diagram (4.4.1)
is a cartesian square.  Since $\s, \s'$ are both principal
bundles, and $\vT_1, \vT_2$ are smooth with connected fibre, 
we have an isomorphism of functors
\begin{equation*}
\tag{4.4.3}
\wt\Ind\circ \wt\vT_L \simeq \wt\vT \circ \Ind. 
\end{equation*}

\para{4.5.}
Put $B_L = B\cap L, U_L = U \cap L$. Then $B_L$ is a $\th$-stable
Borel subgroup of $L$, and $U_L$ is the unipotent radical of $B_L$
such that $B_L = TU_L$.   Put $\CB_L = L/B_L$ and  
$\wt\CB_L = L/U_L$. 
The action of $L$ on $\wt\CB_L$ can be lifted to the action of 
$P$ by the projection $P \to P/U_P \simeq L$.  Hence we have actions 
of $P^{\th}$ on $\wt\CB_L$ and on $\wt\CB_L \times V'$. 
We consider a diagram
\begin{equation*}
\tag{4.5.1}
\begin{CD}
(\wt\CB_L \times V_L) \times \wt\CB_L  @<\t''<<
   (H\times \wt\CB_L \times V') \times (H \times \wt\CB_L) @>\s''>> 
       \wt\ZC_L @>\r''>> (\wt\CB \times V) \times \wt\CB,
\end{CD}
\end{equation*}
where $\wt\ZC_L = (H \times^{P^{\th}}(\wt\CB_L \times V')) \times
                   (H \times^{P^{\th}}\wt\CB_L)$  
with the action of $P^{\th} \times P^{\th}$, 
\begin{equation*}
(p_1, p_2) : ((g_1, \ell_1 U_L,v), (g_2, \ell_2U_L)) 
      \mapsto ((g_1p_1\iv, p_1\ell_1U_L, p_1v), (g_2p_2\iv,p_2\ell_2U_L))  
\end{equation*}
and $\s''$ is the quotient map by $P^{\th} \times P^{\th}$,
$\t''$ is the projection on the factors except $H \times H$, 
$\r''$ is defined by 
\begin{equation*}
(g_1*(\ell_1U_L,v), g_2*\ell_2U_L) \mapsto (g_1\ell_1U, g_1v, g_2\ell_2U). 
\end{equation*}
Then $\s''$ is a $(P^{\th} \times P^{\th})$-principal bundle,  
$\t''$ is 
$(H \times H) \times (P^{\th} \times P^{\th})$-equivariant, 
and $\s'', \r''$ are $H \times H$-equivariant with respect to the 
obvious actions of $L^{\th} \times L^{\th}$ and of $H \times H$.
Moreover, all the maps in the diagram are $(T \times T)$- equivariant
with respect to the actions induced from the action of $T$ on $\wt\CB_L$ 
and $\wt\CB$.  Thus as in the previous cases, one can define a functor
$I: \DD^{P^{\th} \times P^{\th}}_{\CL\boxtimes\CL\iv}
(\wt\CB_L \times V_L \times \wt\CB_L) \to 
   \DD^{H \times H}_{\CL\boxtimes\CL\iv}(\wt\CB \times V\times \wt\CB)$
by
\begin{equation*}
\tag{4.5.2}
I(K) = \r''_!\s''_{\flat}(\t'')^*K[a''], \qquad 
              (K \in  \DD^{P^{\th} \times P^{\th}}_{\CL\boxtimes\CL\iv}
(\wt\CB_L \times V_L \times \wt\CB_L)), 
\end{equation*} 
where $a'' = a' - \dim U_P$.  Note that 
in this case, 
$\dim(\text{fibre of }\t'') - \dim(\text{fibre of }\s'')
           = a''$.
\par
We define a functor 
$\Ch_L: \DD_{\CL\boxtimes\CL\iv}^{L^{\th} \times L^{\th}}
           (\wt\CB_L \times V_L \times \wt\CB_L) 
             \to \DD^{L^{\th}\times L^{\th}}(L \times V_L)$ 
similar to $\Ch$.  
Then $\Ch_L$ induces a functor 
$\DD_{\CL\boxtimes\CL\iv}^{P^{\th} \times P^{\th}}
           (\wt\CB_L \times V_L \times \wt\CB_L) 
             \to \DD^{P^{\th}\times P^{\th}}(L \times V_L)$
in an obvious way, which we denote by the same symbol $\Ch_L$.

\begin{prop} %%%%  Prop. 4.6.
Under the notation above, we have 
an isomorphism of functors
\begin{equation*}
\Ch\circ I \simeq \wt\Ind\circ \Ch_L.
\end{equation*}
\end{prop} 

\begin{proof}
The degree shift for the functor $\Ch$ is given 
by $b = \dim U - \dim T$, and that for $\Ch_L$ is given 
by $b_L = \dim U_L - \dim T$.  Since 
$a'' + b = a' + b_L$, 
we may ignore the degree shift for proving the formula. 
We have the following commutative diagram
\begin{equation*}
\begin{CD}
(\wt\CB_L \times V_L) \times \wt\CB_L @<\t''<< 
      (H \times \wt\CB_L \times V') \times (H \times \wt\CB_L)  @>\s''>>
     \wt\ZC_L @>\r''>> (\wt\CB \times V) \times \wt\CB \\
      @Ar_LAA    @AAr_1A     @AAr_2A   @AA r A  \\
(L \times V_L)\times \wt\CB_L @<\t_2<<
       (H \times P \times V') \times (H \times \wt\CB_L)  @>\s_2>>
             \wt\ZC^P_L @>\r_2>> (G \times V) \times \wt\CB \\
@Vq_LVV    @VVq_1V    @VVq_2V   @VV q V  \\
(L \times V_L) \times \CB_L @<\t_1<< 
     (H\times P \times V') \times (H \times \CB_L) @>\s_1>>
              \wt\CZ^P_L @>\r_1>>  (G \times V) \times \CB \\
@Vp_LVV   @VVp_1V    @VVp_2V    @VV p V    \\
(L \times V_L) @<\t'<<  (H \times P \times V') \times H @>\s'>>
         \wt\CZ^P  @>\r'>>  G\times V. 
\end{CD}
\end{equation*}
Here $P^{\th} \times P^{\th}$ acts on 
$(H \times P \times V') \times (H \times  \wt\CB_L)$ by
\begin{equation*}
(p_1, p_2): ((g_1,x,v),(g_2,\ell U_L)) \mapsto 
              ((g_1p_1\iv, p_1xp_2\iv, p_1v), (p_2g_2, p_2\ell U_L)),   
\end{equation*}
and $\wt\ZC^P_L = H \times^{P^{\th}}(P \times V') 
       \times^{P^{\th}}(H \times \wt\CB_L)$ is the quotient 
of $(H \times P \times V') \times (H \times \wt\CB_L)$ 
by $P^{\th} \times P^{\th}$,  
and $\wt\CZ^P_L = H \times^{P^{\th}}(P \times V')
                   \times^{P^{\th}}(H \times \CB_L)$
is defined similarly.
The first row is the diagram (4.5.1), the fourth row is 
the upper row of the diagram (4.4.1).  The fourth column 
is the diagram (3.2.1), and the first column is the diagram 
(3.2.1) defined by replacing $G$ by $L$. The map $r_1$ is defined 
by 
\begin{equation*}
r_1: ((g_1, x,v), (g_2,\ell U_L)) \mapsto 
          ((g_1, \ol x\ell U_L,v), (g_2\iv, \ell U_L)).
\end{equation*}
$r_1$ is $(P^{\th} \times P^{\th})$-equivariant 
with respect to the above action on 
$(H \times P \times V') \times (H \times \wt\CB_L)$ and the 
action on $(H \times \wt\CB_L \times V') \times (H \times \wt\CB_L)$
defined in 4.4.  Hence it induces the map $r_2: \wt\ZC_L^P \to \wt\ZC_L$. 
The maps $\s_1,\s_2$ are the quotient maps by $P^{\th} \times P^{\th}$.
The map $\r_2$ is defined by 
\begin{equation*}
\r_2 : g_1*(x, v)*(g_2,\ell U_L) 
              \mapsto (g_1xg_2, g_1v, g_2\iv\ell U),
\end{equation*}
and $\r_1$ is defined similarly. 
The maps $q_1, q_2$ are principal $T$-bundles, and the maps
$p_1, p_2$ are obvious projections. 
One can check that the bottom left, bottom center, center right, 
and top right squares are cartesian squares. 
Then the proposition follows from the standard diagram chase.
\end{proof}

\para{4.7.}
In the rest of this section, we assume that $\CX$ is of pure exotic type 
or of enhanced type.  
We consider the varieties $\CX_L = L^{\io\th} \times V_L$ as in the following 
special type. In the pure exotic case, take $L$ such that 
$L^{\th} \simeq Sp_{2n'} \times (GL_1)^{n-n'}$.  
Let $V = V_L^- \oplus V_L^+ \oplus (V_L'\oplus V_L'')$ 
be a decomposition of $V$ into $L^{\th}$-stable subsapces, 
where $V_L^+$ is a subspace with $\dim V_L^+ = 2n'$ corresponding to $Sp_{2n'}$, 
and $V_L^-, V_L'\oplus V_L''$ are subspaces 
with $\dim V^-_L = \dim V'_L \oplus V_L''= n-n'$ corresponding to
$(GL_1)^{n-n'}$. We assume that
$V' = V^-_L \oplus V_L^+ \oplus V'_L$ and $V^-_L$ are $P^{\th}$-stable.
Put $V_L = V'/V^-_L$.  
\par
In the enhanced case, we take $L$ such that 
$L^{\th} \simeq GL_{n'} \times (GL_1)^{n-n'}$.  
Let $V = V_L^-\oplus V_L^+ \oplus (V_L' \oplus V_L'')$ be a decomposition of $V$ into
$L^{\th}$-modules, where $V_L^+$ is a subspace with $\dim V_L^+ = n'$ 
corresponding to $GL_{n'}$, and $V_L^- \oplus V_L' \oplus V_L''$ is a subspace with 
$\dim V_L^- + \dim V_L' + \dim V_L'' = n-n'$ corresponding to $(GL_1)^{n-n'}$.
We assume that $V' = V_L^- \oplus V_L^+ \oplus V_L'$ and $V_L^-$ are $P^{\th}$-stable.
Put $V_L = V'/V_L^-$.  
\par 
We denote by $\CS(\CX)$ the set of such $\CX_L$ for a given $\CX$. 
In either case, $V_L^+$ is regarded as a subspace of $V_L$.
If $V_L' = \{0\}$, 
then $V_L \simeq V_L^+$ and the induced action of $P^{\th}$ on $V_L$ coinicides 
with the lift of the $L^{\th}$-action on $V_L^+$, namely $U_P^{\th}$ acts trivially 
on $V_L$.
We define $\wh\CX_L^0$ as the set of $A \in \wh\CX_L$ 
such that the support of $A$ is contained in $L^{\io\th} \times V^+_L$.
In the case of $L = T$, $\wh\CX^0_L$ coincides with the previous 
definition in (3.4.2). 
Note that if $A \in \wh\CX_L^0$, $A$ turns out to be $P^{\th}$-equivariant. 
We prove the following proposition. 

\begin{prop}  %%%  Prop. 4.8.
Assume that $\CX_L \in \CS(\CX)$.  Then
for any $A \in (\wh\CX_L)^0_{\CL}$,  $\Ind A$ is a semisimple complex 
such that each direct summand is contained in $\wh\CX_{\CL}$ up to shift.  
\end{prop}

\begin{proof}
Take $A \in (\wh\CX_L)^0_{\CL}$.  
By definition, 
$\wt\vT_L A$ is a perverse constituent of $\Ch_L(K)$ for 
some $K \in \DD^{L^{\th} \times L^{\th}}_{\CL\boxtimes\CL\iv}
               (\wt\CB_L \times V_L \times \wt\CB_L)$.
By our assumption, we may assume that 
$K \in \DD_{\CL\boxtimes\CL\iv}^{L^{\th} \times L^{\th}}
 (\wt\CB_L \times V^+_L \times \wt\CB_L)$.  Hence $K$ turns out 
to be $P^{\th} \times P^{\th}$-equivariant, 
and so $\Ch_L(K)$ is
$P^{\th} \times P^{\th}$-equivariant.  By (0.4.1), 
we may further assume that $K$ is a simple perverse sheaf.  
Since $p_L$ is proper, $\Ch_L(K)$ is a semisimple complex 
by the decomposition theorem, and 
$\wt\vT_LA$ is a direct summand of $\Ch_L(K)$, up to shift. 
Since an additive functor on additive categories preserves 
the direct sum decompositon, 
we see that $\wt\Ind\circ\wt\vT_L A$ is a direct summand of 
$\wt\Ind\circ\Ch_L(K)$, up to shift. 
This implies that  
if $A_1$ is a perverse constituent 
of $\wt\Ind\circ\wt\vT_L A$, then $A_1$ is a perverse
constituent of $\wt\Ind \circ \Ch_L(K)$.  
By Proposition 4.6, we have $\wt\Ind\circ\Ch_L(K) \simeq \Ch\circ I(K)$.
Hence $A_1$ is a perverse constituent of 
$\Ch(K_1)$ with $K_1 = I(K) \in 
    \DD^{H\times H}_{\CL\boxtimes\CL\iv}(\wt\CB \times V \times \wt\CB)$.
\par
On the other hand,  
by (4.4.3) we have $\wt\vT\circ \Ind A \simeq \wt\Ind\circ \wt\vT_L A$.
Hence 
$\wt\vT\circ \Ind A$ is a semisimple complex whose simple objects 
are of the form $A_1 = \wt\vT A'$ for a simple perverse sheaf $A'$ 
contained in $\Ind A$ up to shift. 
Then the above argument shows that $A' \in \wh\CX_{\CL}$. 
The semisimplicity of $\Ind A$ is clear from the decomposition theorem 
since $\r$ in (4.2.1) is proper.
This proves the proposition. 
\end{proof}

\para{4.9.} We consider the case where $L = T$.  We follow 
the notation in 3.4.
A character sheaf $A$ on $\CX_T$ can be written as 
$A = \wt\CE_{V_1}$ with $V_1 \subset V$. 
If $V_1$ is $B^{\th}$-stable, then $A$ turns out to be 
$B^{\th}$-equivariant, and so $\Ind_{\CX_T,B}^{\CX}A$ can be defined. 
We denote by $\wh\CX^{\ps}$ the set of character sheaves in 
$\wh\CX$ appearing in the perverse constituents 
of $\Ind_{\CX_T,B}^{\CX}A$ for various $A = \wt\CE_{V_1}$ 
($V_1$ : $B^{\th}$-stable), and call 
them principal series character sheaves.   
Clearly, the determination of $\wh\CX^{\ps}$ is reduced to the case 
where $\CX$ is of pure exotic type or of enhanced type. 
We shall 
describe $\wh\CX^{\ps}$ in each case. 
\par
First assume that $\CX = G^{\io\th} \times V$ is 
of pure exotic type with $\dim V = 2n$.
(We also consider the case where $\CX = G^{\io\th}$, in which 
case we use the notation $\wh G^{\io\th}_{\CL}, (\wh G^{\io\th}_{\CL})^{\ps}$, 
etc.)
We consider $\CX_T = \CX_{T,V_T}= T^{\io\th} \times V_T$, 
where $V_T = M_{m'}/M_n$ for $m' \ge n$.  
Consider the diagram (4.2.1) for the case $P = B$, 

\begin{equation*}
\begin{CD}
\CX_T @<\t<< H \times B^{\io\th} \times M_{m'} @>\s>> \wt\CX^B @>\r>> \CX.
\end{CD}
\end{equation*}

We consider the character sheaf 
$A = \wt\CE_{V_1}[\dim \CX_{T,V_1}]$ for a $B^{\th}$-stable subspace 
$V_1 = M_{m}/M_n$ for some $n \le m \le m'$. 
Let $\wt\CX_m$ be as in 1.1.  Then $\wt\CX_m$ is a closed subvariety of 
$\wt\CX^B$. We see that
the support of $\s_{\flat}\t^*A$ is contained in $\wt\CX_m$, 
and $\s_{\flat}\t^*A$ coincides with 
$(\a^{(m)})^*\CE[\dim \CX_{T,V_1}]$ under the notation in 1.1.
In this case, $a = \dim U + n = \dim \wt\CX_m - m$.
Since $\dim\wt\CX_m = \dim \CX_m + (m -n)$,
we have
\begin{equation*}
\dim \CX_{T,V_1} + a = \dim \CX_m + (m -n).
\end{equation*}
Hence we see that $\Ind_{\CX_T,B}^{\CX}A \simeq K_{m, T, \CE}[m -n]$ 
under the notation in Theorem 1.11.  Note that $m = n$ if $A \in \wh\CX_T^0$. 
$K_{m,T, \CE}$ is a semisimple complex on $\CX$, and 
the decomposition of $K_{m, T, \CE}$ into
simple constituents was given in Theorem 1.11. 
In the case where $m = n$, $K_{n, T,\CE} = K_{T,\CE}$ is a semisimple
perverse sheaf by Theorem 4.2 in [SS].
In particular, 
by comparing Theorem 1.11 with Theorem 4.2 in [SS], we see 
that any simple component appearing in $K_{m,T,\CE}$, up to shift, 
is already contained in $K_{T,\CE}$.  
It follows that
\par\medskip\noindent
(4.9.1) \ Assume that $\CX$ is of pure exotic type.  
For any $B^{\th}$-stable $A \in \wh\CX_T$,  
$\Ind_{\CX_T,B}^{\CX}A$ is a semisimple complex in 
$\DD^H(\CX)$.  If $A \in \wh\CX_T^0$, 
then $\Ind_{\CX_T,B}^{\CX}A \in \CM(\CX)$. 
Moreover 
$\wh\CX_{\CL}^{\ps}$ coincides with the union of simple constituents
of $K_{T,\CE}$ for various $\CE$ such that $\vT^*_T\CE = \CL$, namely
\begin{equation*}
\wh\CX^{\ps}_{\CL} = \{ \IC(\CX_m, \CL_{\r})[\dim \CX_m]
   \mid \r \in \CW_{\Bm,\CE}\wg, 0 \le m \le n, \vT^*_T\CE \simeq \CL\}
\end{equation*}     
under the notation in Theorem 1.5.
Moreover the set $(\wh G^{\io\th}_{\CL})^{\ps}$ coincides with 
the subset of $\wh\CX_{\CL}^{\ps}$ consisting of those 
$\IC(\CX_m,\CL_{\r})[\dim \CX_m]$ such that $m = 0$.

\par\medskip
Next assume that $\CX$ is of enhanced type, namely, 
$\CX = G^{\io\th} \times V$, where $G = GL(V) \times GL(V)$  with $\dim V = n$.
(We also consider the case where $\CX = G^{\io\th}$, in which case, 
we use the notation $\wh G_{\CL}, \wh G_{\CL}^{\ps}$, etc.).
Let $\CX_T = T^{\io\th} \times V_T$ with $V_T = M_{m''}/M_{m'}$ for $m'' \ge m'$.
A $B^{\th}$-stable character sheaf on $\CX_T$ is given as $A = \wt\CE_{V_1}[\dim \CX_{T,V_1}]$
with $V_1 = M_m/M_{m'}$ with $m' \le m \le m''$. 
Let $\wt\CX_m$ be the variety defined in 2.6.  
Then 
$\wt\CX_m$ is a closed subvariety of $\wt\CX^B$.  
Since $\dim \wt\CX_m = \dim \CX_m$ by 2.6, we have
\begin{equation*}
\dim \CX_{T,V_1} + a = \dim \wt\CX_m  = \dim \CX_m .
\end{equation*}
Let $K_{m,T,\CE}$ be the complex on $\CX_m$ defined in 2.8.
By a similar 
argument as before, we see that 
$\Ind_{\CX_T}^{\CX}A \simeq K_{m,T,\CE}$. 
Note that if $m' = m''$, we have $\CX_T = T^{\io\th} \times \{0\}$.
Thus for each $0 \le m \le n$, $K_{m,T,\CE}$ is 
realized as $\Ind_{\CX_T, V',B}^{\CX}A_0$ for $A_0 \in \wh\CX_T^0$ by choosing 
$V'$ so that $m' = m'' = m$.  
Now by Theorem 2.9, $K_{m,T,\CE}$ is a semisimple perverse sheaf, 
and its simple components are completely described there.  Thus we have
\par\medskip\noindent
(4.9.2) \ Assume that $\CX$ is of enhanced type.  Then 
for any $B^{\th}$-stable $A \in \wh\CX_T$, $\Ind_{\CX_T,B}^{\CX}A$ is a semisimple 
perverse sheaf on $\CX$.  
Under the notation of Theorem 2.9, we have 
\begin{equation*}
\wh\CX^{\ps}_{\CL} = \{ \IC(\CX_m, \CL_{\r})[\dim \CX_m]
  \mid \r \in W_{\Bm,\CE}\wg, 0 \le m \le n, \vT_T^*\CE \simeq \CL\}.   
\end{equation*}
Any element in $\wh\CX^{\ps}_{\CL}$ is obtained as a simple component of 
some $\Ind_{\CX_T}^{\CX}A_0$ for $A_0 \in \wh\CX_T^0$. 
Moreover, (as is well-known) the set $\wh G_{\CL}^{\ps}$ is given as 
the subset of $\wh\CX_{\CL}^{\ps}$ consisting of those 
$\IC(\CX_m, \CL_{\r})[\dim \CX_m]$ with $m = 0$.
\par

\remark{4.10.}
In the pure exotic case, the set of character sheaves 
$\wh\CX$ was defined in [SS, 4.1], [SS2, 6.4].  In fact, the set
$\wh\CX$ given there is nothing but $\wh\CX^{\ps}$ in the present
notation.   Later in Section 6, we prove that those two definitions 
coincide with each other, namely, we show that $\wh\CX = \wh\CX^{\ps}$.      
In the enhanced case, $\wh\CX^{\ps}$ coincides with a subset of mirabolic
character sheaves introduced in [FGT, Conjecture 1].

\par\bigskip
%%%%
%%%%
\section{Restriction}

\para{5.1.}
In this section we shall construct the restriction functors, 
and show the main properties based on the idea of [Gr, Section 2].
Here we follow the notation in 4.1, i.e., $V_L = V'/V_L^-$ for 
$P^{\th}$-stable subspaces $V', V_L^-$.
Let $j : P^{\io\th} \times V' \hra \CX$  be the inclusion map, 
and $\pi_P: P^{\io\th} \times V' \to \CX_L$ be the projection 
$(p,v) \mapsto (\ol p, \ol v)$.  Then $j$ and $\pi_P$ are 
$P^{\th}$-equivariant with respect to the obvious actions of $P^{\th}$.
We define a restriction functor 
$\Res = \Res_{\CX_L,V', P}^{\CX}:
    \DD^H(\CX) \to \DD^{P^{\th}}(\CX_L)$ by 
\begin{equation*}
\tag{5.1.1}
\Res K = (\pi_P)_!j^*K \qquad (K \in \DD^H(\CX)).
\end{equation*}
\par
We consider $\th$-stable parabolic subgroups $P \subset Q$ of $G$.
We follow the notation in 4.2.  In particular, $V_M = V'/V_M^-$ for a $Q^{\th}$-subspace
$V'$ of $V$,  $V_L$ is a quotient of $V'_M$ for a $P_M^{\th}$-stable subspace $V_M'$ of $V_M$, and   
$V_M' = V''/V_M^-$ for a $P^{\th}$-stable subspace $V''$ of $V'$. Then $V_L$ is a quotient 
of $V''$ as $P^{\th}$-module.  The following transitivity holds. 

\begin{prop}  %%%%  Prop. 5.2.
Let $K \in \DD^H(\CX)$.  Then there is an isomorphism  in $\DD^{P^{\th}}(\CX_L)$
\begin{equation*}
\Res^{\CX}_{\CX_L,V'',P}K \simeq (\Res_{\CX_L,V'_M,P_M}^{\CX_M}\circ \Res_{\CX_M,V',Q}^{\CX})K.
\end{equation*}
\end{prop}

\begin{proof}
We consider the 
following commutative diagram.
                                                                                                    
\begin{figure}[h]
\setlength{\unitlength}{1mm} 
\begin{picture}(80,60)(6,40)  
\put(0,90) {$M^{\io\th} \times V_M$} 
\put(32, 91) {\vector (-1,0) {10}}
\put(36,90){$Q^{\io\th} \times V'$}
\put(54, 91) {\vector (1,0) {10}} 
\put(68, 90){$G^{\io\th} \times V$} 
\put(44, 76) {\vector (0,1) {10}} 
\put(46, 81){$_{j'}$} 
\put(2, 81){$_{j_{P_M}}$}
\put(26, 93){$_{\pi_Q}$}
\put(58, 93){$_{j_Q}$}
\put(2, 62){$_{\pi_{P_M}}$}
\put(26, 75){$_{\pi_Q'}$}  
\put(2, 70) {$P^{\io\th}_M \times V'_M$}
\put(10, 76) {\vector (0,1) {10}}
\put(32, 72){\vector (-1,0) {10}} 
\put(36, 70){$P^{\io\th} \times V''$}                                                                  
\put(10, 66){\vector (0,-1) {10}}
\put(2, 50){$L^{\io\th} \times V_L$}
\put(54, 74) {\vector (3,2) {18}}
\put(40, 66){\vector (-3,-2) {18}}
\put(64, 78){$_{j_P}$}
\put(32, 58){$_{\pi_P}$}

\end{picture}
\end{figure}   
\par\noindent
where $j'$ is the inclusion map, and $\pi_Q'$ is the 
projection.  
Note that the upper left square is a cartesian square.  Then the proposition 
follows from a simple diagram chase. 
\end{proof}

\para{5.3.}
We consider a commutative diagram 
\begin{equation*}
\tag{5.3.1}
\begin{CD}
L \times V_L @<\pi_P'<< P \times V' @ >j'>>   G \times V  \\
    @V\vT_LVV                 @VV\vT_PV               @VV\vT V   \\
\CX_L @<\pi_P<<  P^{\io\th} \times V' @>j>>   \CX,
\end{CD}
\end{equation*}
where $j'$ is the inclusion map, $\pi_P'$ is the projection 
$(p,v) \mapsto (\ol p,\ol v)$, and $\vT_P$ is the map 
$(p,v) \mapsto (p\th(p)\iv,v)$.  $H \times H$ acts on $G \times V$,
and $L^{\th} \times L^{\th}$ acts on $L \times V_L$ as in 3.2. 
$P^{\th} \times P^{\th}$ also acts on $P \times V'$ by the restriction 
of the action on $G \times V$. 
Then $j', \pi_P'$ are $P^{\th} \times P^{\th}$-equivariant. 
We define $\wt\Res = \wt\Res_{L\times V_L, P}^{G \times V} :
  \DD^{H \times H}(G \times V) \to 
     \DD^{P^{\th} \times P^{\th}}(L \times V_L)$ 
by
$\wt \Res\, K = (\pi_P')_!(j')^*K$.  
As in the case of induction functors, we have 
%%%
\begin{lem}  %%% Lemma 5.4
There is an isomorphism of functors
\begin{equation*}
\wt\vT_L\circ \Res \simeq \wt\Res \circ \wt\vT.
\end{equation*}
\end{lem}
\begin{proof}
Since the square in the left hand side of (5.3.1) 
is not a cartesian square, we need a special care.
Let $Z$ be the fibre product of $L \times V_L$ and 
$P^{\io\th} \times V'$ over $\CX_L$, and 
$\a: P \times V' \to Z$ be the natural projection. 
We have $\b\circ\a = \pi_P'$ and $\g\circ \a = \vT_P$. 
\begin{equation*}
\begin{CD}
L \times V_L @<\b<< Z @<\a<< P \times V'  \\
  @V\vT_LVV            @VV\g V                 \\
\CX_L   @<\pi_P<<   P^{\io\th} \times V'   
\end{CD}
\end{equation*}
Since the projection $P^{\io\th} \to L^{\io\th}$ 
is a vector bundle with fibre $U_P^{\io\th}$, 
$\a$ turns out to be a vector bundle with fibre isomorphic 
to $U_P^{\th}$. Hence $\a_!\a^* \simeq \id[-2d]$ with 
$d = \dim U_P^{\th}$, and so
$\a_!\vT_P^* \simeq \a_!\a^*\g^* \simeq \g^*[-2d]$.
It follows that 
\begin{equation*}
(\pi_P')_!\vT_P^* \simeq \b_!\a_!\vT_P^* 
\simeq \b_!\g^*[-2d] \simeq \vT_L^*(\pi_P)_![-2d].
\end{equation*}
The lemma follows from this.
\end{proof}

\para{5.5.} 
For $K \in \DD^{H \times H}_{\CL\boxtimes\CL\iv}
           (\wt\CB \times V \times \wt\CB)$, we want to 
investigate the perverse constituents of $\wt\Res\circ \Ch(K)$. 
Consider a commutative diagram 

\begin{equation*}
\begin{CD}
(\wt\CB \times V) \times \wt\CB @<r << (G\times V) \times \wt\CB 
             @>q>> (G \times V)\times\CB @>p>> G \times V  \\
     @.        @Aj_1AA     @AAj_0A   @AAj'A    \\
     @.      (P \times V')\times \wt\CB @>q_0>> (P \times V') \times \CB
                    @>p'>> P \times V'  \\
     @.         @.        @.     @VV\pi'_PV   \\
     @.         @.        @.      L \times V_L,            
\end{CD}
\end{equation*}
where $j_0,j_1$ are inclusions, 
and put $r_0 = r\circ j_1, p_0 = \pi'_P\circ p'$. 
Since both squares are cartesian squares, we see that 

\begin{equation*}
\tag{5.5.1}
\wt\Res\circ \Ch(K) \simeq (p_0)_!(q_0)_{\flat}r_0^*K, 
\qquad \text{(up to shift)}.
\end{equation*}
\par\medskip

Let $W = N_G(T)/T$ be the Weyl group of $G$, and $W_P$
the Weyl subgroup of $W$ corresponding to $P$. We consider 
the partition $\CB = \coprod_{w \in W_P\backslash W}\CB^P_w$, 
$\wt\CB = \coprod_{w \in W_P\backslash W}\wt\CB^P_w$ into $P$-orbits, 
where $\CB^P_w = PwB/B, \wt\CB^P_w = PwU/U$.  We consider a diagram
\begin{equation*}
\begin{CD}
(\wt\CB \times V)\times \wt\CB @<r_w<< (P \times V') \times \wt\CB^P_w
           @>q_w>>   (P \times V') \times \CB^P_w @>p_w>> L \times V_L, 
\end{CD}
\end{equation*} 
where $r_w, q_w,p_w$ are the maps induced from $r_0, q_0, p_0$
by the inclusions $\CB_w^P \hra \CB, \wt\CB_w^P \hra \wt\CB$.
By considering the perverse cohomology long exact sequence, we see
that 
\par\medskip\noindent
(5.5.2) \ The perverse constituents of $(p_0)_!(q_0)_{\flat}r_0^* K$
are contained in the perverse constituents of 
$(p_w)_!(q_w)_{\flat}r_w^*K$ for some $w \in W_P\backslash W$.
\par\medskip
We show the following lemma. 
%%%%%
\begin{lem}  %%% Lemma 5.6
For $K \in \DD^{H \times H}_{\CL\boxtimes\CL\iv}
  (\wt\CB \times V\times\wt\CB)$, the perverse constituents 
of $\wt\Res\circ\Ch(K)$ are contained in the perverse constituents
of $\Ch_L(K_1)$ for some 
$K_1 \in \DD^{L^{\th}\times L^{\th}}_{w^*\CL\boxtimes w^*\CL\iv}
                  (\wt\CB_L \times V_L \times \wt\CB_L)$  
with $w \in W_P\backslash W$. 
\end{lem}

\begin{proof}
Let $B_L' = wBw\iv \cap L$ be a Borel subgroup of $L$,
and $U_L' = wUw\iv \cap L$ be its unipotent radical, for a fixed 
$w \in W_P\backslash W$.
Put $\CB_L' = L/B_L'$ and $\wt\CB_L' = L/U_L'$.
We consider a commutative diagram
\begin{equation*}
\tag{5.6.1}
\begin{CD}
      @.     (P \times V')\times \wt\CB_w^P  @.  \\
      @.         @VVf_1 V             @.   \\
      Z_w  @>\a>>  (\wt\CB_w^P \times V') \times \wt\CB_w^P  
                    @>i_w>>  (\wt\CB \times V) \times \wt\CB \\   
      @V\b VV                      @VV\d_wV          @.  \\
 (L \times V_L) \times \wt\CB'_L @>r_L>>  
               (\wt\CB'_L \times V_L)\times \wt\CB'_L,  
\end{CD}
\end{equation*}
where 
\begin{align*}
f_1 &: (p_1, v, p_2wU) \mapsto (p_1p_2wU, v, p_2wU),  \\ 
i_w &: \text{ the inclusion map}, \\
\d_w &: (p_1wU, v, p_2wU) \mapsto (\ol p_1U_L',\ol v, \ol p_2U_L'), \\
Z_w &: \text{ the fibre product of $(L \times V_L)\times \wt\CB_L'$ 
and $\wt\CB_w^P \times V \times \wt\CB_w^P$} \\
    &\text{ \phantom{**}over $(\wt\CB'_L \times V_L) \times \wt\CB_L'$}.
\end{align*}
Let $f_2: (P \times V')\times \wt\CB_W^P \to (L \times V_L) \times \wt\CB_L'$
be the map $(p_1,v,p_2wU) \mapsto (\ol p_1, \ol v, \ol p_2U_L')$.
Then $\d_w\circ f_1 = r_L\circ f_2$, and so we have a map 
$f : (P\times V') \times \wt\CB_w^P \to Z_w$ such that 
$f_1 = \a\circ f, f_2 = \b\circ f$.
We note that $f$ is an affine space bundle with fibre isomorphic to
$U_P \cap wUw\iv$.
\par
One can check that the following diagram commutes.
\begin{equation*}
\tag{5.6.2}
\begin{CD}
L \times V_L @<p_w<<  (P \times V') \times \CB_w^P @<q_w<< 
                             (P \times V')\times \wt\CB_w^P \\
   @V\id VV               @V\b' VV                  @VV f_2V   \\
L \times V_L @<p_L<< (L \times V_L) \times \CB'_L @<q_L<< 
                             (L \times V_L) \times \wt\CB_L',
\end{CD}
\end{equation*}
where $\b': (p_1,v, p_2wB) \mapsto (\ol p_1, \ol v, \ol p_2B_L')$.
Moreover, the right square is a cartesian square.  
Note that the maps $p_L, q_L$ and $r_L$ are independent of the choice of 
$B_L', U_L'$ under the isomorphism 
$\wt\CB'_L \simeq \wt\CB_L, \CB_L' \simeq \CB_L$.  
\par
Take $K \in \DD^{H \times H}_{\CL\boxtimes\CL\iv}
             (\wt\CB \times V \times \wt\CB)$.   
Since $r_w = i_w\circ f_1$, we have $r_w^*K \simeq f^*\a^*i_w^*K$ 
by (5.6.1).  By (5.6.2), we have
\begin{align*}
(p_w)_!(q_w)_{\flat}r_w^*K &\simeq (p_L)_! (q_L)_{\flat}
                               (f_2)_!r_w^*K \\
                    &\simeq (p_L)_!(q_L)_{\flat}\b_!f_!(f^*\a^*i_w^*K).
\end{align*}
 Since $f$ is an affine space bundle of rank $d = \dim (U_P \cap wUw\iv)$, 
we have 
\begin{align*}
\b_!f_!f^*\a^*i_w^*K &\simeq \b_!\a^*i_w^*K[-2d] \\
                     &\simeq r_L^*(\d_w)_!i_w^*K[-2d].
\end{align*}
It follows that 
\begin{align*}
\tag{5.6.3}
(p_w)_!(q_w)_{\flat}r_w^*K &\simeq 
     (p_L)_!(q_L)_{\flat}r_L^*((\d_w)_!i_w^*K[-2d]) \\
      &= \Ch_L((\d_w)_!i_w^*K[-2d]).
\end{align*}
Note that $(\d_w)_!i_w^*$ is a functor 
from $\DD^{H \times H}(\wt\CB \times V \times \wt\CB)$ 
   to $\DD^{L^{\th} \times L^{\th}}(\wt\CB_L \times V_L \times \wt\CB_L)$
since $i_w, \d_w$ are $L^{\th} \times L^{\th}$ equivariant.  
Moreover, $T\times T$ acts, for $(t_1, t_2) \in T \times T$,  
\par\medskip
on $(\wt\CB \times V) \times \wt\CB$ by 
$(g_1U, v, g_2U) \mapsto (g_1t_1\iv U, v, g_2t_2\iv U)$,
\par 
on $(\wt\CB_w^P \times V') \times \wt\CB_w^P$ by 
$(p_1wU, v, p_2wU) \mapsto (p_1w(t_1\iv)wU, v, p_2w(t\iv)wU)$,
\par
on $\wt\CB'_L \times V_L \times \wt\CB'_L$ by 
$(\ell_1U_L', v, \ell_2U_L) \mapsto 
                 (\ell_1w(t_1\iv)U_L', v, \ell_2w(t_2\iv)U_L')$,
\par\medskip\noindent
and $i_w, \d_w$ are $T \times T$-equivariant with respect to those
actions. 
This shows that $(\d_w)_!i_w^*: 
\DD^{H \times H}_{\CL\boxtimes\CL\iv}(\wt\CB \times V \times \wt\CB)
\to \DD^{L^{\th} \times L^{\th}}_{w^*\CL\boxtimes w^*\CL\iv}
             (\wt\CB_L \times V_L \times \wt\CB_L)$.
Now the lemma follows from (5.5.1), (5.5.2) and (5.6.3). 
\end{proof}

\begin{prop}  %%%%  Prop. 5.7
For any $A \in \wh\CX$, the perverse constituents of $\Res A$ are all contained 
in $\bigcup_{w \in W_P\backslash W}(\wh\CX_L)_{w^*\CL}$. 
\end{prop}

\begin{proof}
Take $A \in \wh\CX_{\CL}$ and let $A'$ be a perverse constituent 
of $\Res A$.  Then by Lemma~5.4,  $\wt\vT_L A'$ is a perverse constituent 
of $\wt\Res(\wt\vT A)$.  By definition, $\wt\vT A$ is a perverse constituent
of $\Ch K$ for some $K \in \DD^{H \times H}_{\CL\boxtimes\CL\iv}
            (\wt\CB \times V \times \wt\CB)$.  By (0.4.1), we may assume that
$K$ is a simple perverse sheaf. Since $p$ is proper, $\Ch K$ is a semisimple 
complex by the decomposition theorem. 
Since $\wt\vT A$ is a simple constituent
of $\Ch K$, it is a direct summand of $\Ch K$, up to shift. 
Hence $\wt{\Res}(\wt\vT A)$ is a direct summand
of $\wt{\Res}(\Ch K)$, up to shift.    
Thus $\wt\vT_LA'$ is a perverse constituent of  
$\wt\Res(\Ch K)$. The proposition now follows from  Lemma 5.6.
\end{proof}

\para{5.8.}
We consider a special case where $\CX$ is of pure exotic type 
or of enhanced type, and assume that $\CX_L \in \CS(\CX)$ as in 4.7.
Under the notation in 4.7,  we further assume that $V_L = V_L^+$, i.e., 
$V_L' = \{ 0\}$. 
Thus $V$ can be decomposed as $V = V_L^- \oplus V_L^+ \oplus V_L''$
with $V' = V_L^- \oplus V_L^+$.   
Write $T$ as $T = T_1 \times T_2$, where 
$T^{\th}_1$ (resp. $T^{\th}_2$) is a maximal torus of 
$GL(V_L^+)$ (resp. $GL(V_L^-\oplus V_L'')$ 
according to the decomposition of $V$. 
We define a cocharacter 
$\vf = (\vf_1, \vf_2) : \BG_m \to T^{\th} \simeq T^{\th}_1 \times T^{\th}_2$ 
by the condition that $\vf_1 : \BG_m \to T^{\th}_1$ is trivial, and 
$\vf_2: \BG_m \to T^{\th}_2$
is regular with respect to $P^{\io\th}$, namely, the weights of $\BG_m$ on 
$V_L^-$ (resp. on $V_L''$) are strictly positive (resp. stricitly negative) 
and that it gives strictly positive weights on the space $U_P^{\io\th}$. 
(In fact, in the enhanced case, it is easy to find such $\vf_2$.  In the pure exotic case,
since $\Im \vf_2 \subset T_2^{\th}$, if we define $\vf_2$ so that $\vf_2(a)$ acts
on $V_L^-$ with strictly positive weights, then $\vf_2(a)$ automatically acts on 
$V_L''$ with strictly negative weights). 
Now $T^{\th}$ acts on $\CX$ as the restriction of the action of $H$.  Through the action of 
$T^{\th}$, $\BG_m$ acts on $\CX$.
Let $P^-$ be the opposite parabolic subgroup of $P$ so that $P \cap P^- = L$. 
We consider the setting for Braden's theorem in 0.5 for the $\BG_m$-variery $X = \CX$.
It follows from  the definition of the cocharacter $\vf$, $X^0$ coincides with 
$L^{\io\th} \times V_L$. Moreover, we have 
$X^+ = P^{\io\th} \times (V_L^- \oplus V_L^+)$ and 
$X^- = (P^-)^{\io\th} \times (V^+_L \oplus V_L'')$.
We show the following result.

\begin{prop} %%%%  Prop. 5.9
Assume that $\CX_L \in \CS(\CX)$ and that $V_L = V_L^+$.  
Then for any $A \in \wh\CX$, $\Res A$ is a semismple complex. 
More precisely, $\Res A$  is a direct sum of complexes of the form
$A_1[i]$ with $A_1 \in \wh\CX^0_L$.  
\end{prop}

\begin{proof}
Since $A$ is $H$-equivariant,  $A$ is $\BG_m$-equivariant.
By Braden's theorem  (Theorem 0.6), $i^!j^*A$ is a semisimple complex, 
where $i: L^{\io\th} \times V_L \to G^{\io\th} \times V'$ is the inclusion map
for $V' = V_L^-\oplus V_L^+$. 
On the other hand, $\pi_P : P^{\io\th} \times V' \to L^{\io\th} \times V_L$ is a vector bundle 
with $\BG_m$-action, where $\BG_m$ acts trivially on $L^{\io\th} \times V_L$, 
and acts on each 
fibres with strictly positive weights. Since $j^*A$ is $\BG_m$-equivariant, 
by Lemma 0.7, $i^!j^*A \simeq (\pi_P)_!j^*A = \Res A$. 
The proposition is proved.
\end{proof}

\par\bigskip
\section{Classification of character sheaves}

\para{6.1.}
In this section, we assume that $\CX$ is of pure exotic type or 
of enhanced type.
Following the idea in [L3, I, 3.10], we introduce the notion of
cuspidal character sheaves on $\CX$.
\par
We consider the variety $\CX_L = L^{\io\th} \times V_L$ 
associated to a $\th$-stable parabolic subgroup $P$ and its 
$\th$-stable Levi subgroup
$L$ as before. 
Recall that $\CS(\CX)$ is as in 4.7.
For $\CX_L \in \CS(\CX)$, we define an integer $c_{\CX_L}$ by
\begin{equation*}
\tag{6.1.1}
c_{\CX_L} = \dim U_P^{\th} - \dim U_P^{\io\th} - \dim V_L^-.
\end{equation*}
Hence 
$c_{\CX_L} = 0$ 
if $\CX$ is of pure exotic type, and $c_{\CX_L} = -\dim V_L^-$ 
if $\CX$ is of enhanced type.

\par
A character sheaf $A \in \wh\CX$ is said to be cuspidal if 
for any $\CX_L \in \CS(\CX)$ and for any $V'$, we have 
$\Res A[-c_{\CX_L}-1] \in \DD \CX_L^{\le 0}$, or equivalently, 
$\dim \supp \CH^i(\Res A[-c_{\CX_L}]) < -i$ for all $i$, where 
$\Res = \Res_{\CX_L,V', P}^{\CX}$.

\para{6.2.}
Let $x = x_sx_u$ be the Jordan decomposition of $x \in G^{\io\th}$, where 
$x_s$ is semisimple and $x_u$ is unipotent.
Then $x_s, x_u \in G^{\io\th}$, and $L_0 = Z_G(x_s)$ is a $\th$-stable
Levi subgroup of a (not necessarily $\th$-stable) parabolic subgroup of $G$.
Hence  $x_s \in Z(L_0)^{\io\th}$ 
and $x_u \in (L_0)^{\io\th}\uni$.
We define
$Z(L_0)\reg = \{ x \in Z(L_0) \mid Z_G(x) = L_0\}$. 
Let $L$ be the smallest $\th$-stable Levi subgroup of a $\th$-stable 
parabolic subgroup $P$ of $G$ containing $L_0^{\io\th}$ such that
$\CX_L \in \CS(\CX)$. 
Here we assume that $V = V^{\sharp}_L \oplus V^-_L$, where 
$V_L^{\sharp}$ is an $L^{\th}$-stable subspace, and $V^-_L$ is a 
$P^{\th}$-stable subspace of $V$, such that 
$V_L^{\sharp} \simeq V_L= V/V^-_L$ as $L^{\th}$-modules.  
(In the pure exotic case, $L_0 \simeq \prod_i GL_{2n_i}$ 
so that $L_0^{\th} = \prod_i Sp_{2n_i}$. Then 
$L = GL_{2n'} \times (GL_1)^{2n - 2n'}$, where $n'$ is the sum of $n_i$ 
such that $n_i > 1$, since $GL_2^{\io\th}$ is a torus.  
Here $V_L^- = M_{n-n'}$ and $\dim V^{\sharp}_L = n + n'$. 
In the enhanced case, $L_0^{\th} \simeq \prod_iGL_{n_i}$, and  
$L^{\th} \simeq GL_{n'} \times (GL_1)^{n-n'}$, where $n'$ is the sum of 
$n_i$ such that $n_i > 1$. Here $V_L^- = V_m$ for $m \le n-n'$ and 
$\dim V_L^{\sharp} = n-m$.)  
Let $\vS = Z(L_0)^{\io\th} \times \CO $ 
for an $L_0^{\th}$-orbit $\CO$
in $(L_0)^{\io\th}\uni \times V^{\sharp}_L$. (Note that $V^{\sharp}_L$ can be regarded as 
a vector space $V_{L_0}$ attached to $L_0$, and so 
$\CX_{L_0} = (L_0)^{\io\th} \times V_{L_0}$ is a similar variety as 
$\CX$.) For such a pair $(L,\vS)$,
we define 
$\vS\reg = Z(L_0)^{\io\th}\reg \times \CO$,  
and 
$Y_{(L,\vS)} = \bigcup_{g \in H}g(\vS\reg)$.
We claim that 
\par\medskip\noindent
(6.2.1) \ $\CX = \bigcup_{(L,\vS)}Y_{(L,\vS)}$, where $(L,\vS)$ runs over
all such pairs up to $H$-conjugacy, gives a partition 
of $\CX$ into finitely many locally closed smooth irreducible 
varieties stable by $H$. 
\par\medskip
In fact, since $Y_{(L,\vS)} \simeq H \times^{L_0^{\th}}\vS\reg$, 
$Y_{(L,\vS)}$ is smooth, irreducible and locally closed.  
It is clear from the definition that $Y_{(L,\vS)}$ are mutually 
disjoint.  So we have only to show that 
$\bigcup_{(L,\vS)} Y_{(L,\vS)} = \CX$. 
For this it is enough to consider the pure exotic case, i.e, 
where $G = GL_{2n}$, $H = Sp_{2n}$ and $\CX = G^{\io\th} \times V$.   
Take $(x,v) \in \CX$.
By (a variant of ) [SS, Lemma 2.1], we may assume that
$x_u \in U, v \in M'_n$, and $x_s \in T$, 
by replacing $(x,v)$ by its $H$-conjugate, 
where $M_n'$ is the maximal isotropic subspace of $V$ 
stable by $T$ such that $V = M_n \oplus M_n'$.
Put $L_0 = Z_G(x_s)$.  
Then $(x_u,v) \in (L_0)^{\io\th}\uni \times V^{\sharp}_L$ since  
$v \in M'_n \subset V^{\sharp}_L$ by our choice of $V^{\sharp}_L$.
This shows that $(x,v) \in \vS\reg = Z(L_0)^{\io\th}\reg \times \CO$ 
for an $L_0^{\th}$-orbit $\CO$ in $(L_0)^{\io\th}\uni \times V^{\sharp}_L$ containing 
$(x_u,v)$,
and the claim follows.  Thus (6.2.1) holds.   
\par
\begin{prop}  %%%% Prop.6.3.
Assume that $G$ is not a torus.  Then $\wh \CX$ does not
contain a cuspidal character sheaf. 
\end{prop}

\begin{proof}
Assume that $G$ is not a torus. Take $A \in \wh\CX$.
Let $\CU$ be a locally closed smooth irreducible subvariety of $\CX$
such that $A \simeq \IC(\CU, \CE)[\dim \CU]$. 
Since $A$ is $H$-equivariant, we may assume that $\CU$ is $H$-stable. 
Since $\CU$ is irreducible, by (6.2.1), there exists a unique piece 
$Y_{(L,\vS)}$ such that $\CU \cap Y_{(L,\vS)}$ is open dense in $\CU$.
Since $Y_{(L,\vS)}$ is $H$-stable, we may assume that $\CU \subset Y_{(L,\vS)}$.
Take $(x,v) \in \CU$, and let $x = x_sx_u$ be 
the Jordan decomposition of $x$.  Let $L_0 = Z_G(x_s)$ and 
$\CO_0$ be the $L_0^{\th}$-orbit of $(x_u,v)$ so that 
$\vS\reg = Z(L_0)\reg^{\io\th} \times \CO_0$.  
Let $\CO$ be the $H$-orbit of $(x,v)$.  
Put $\d = \dim \CO$ and $d = \dim \CU$. 
\par
First assume that 
$\CX$ is of pure exotic type $G^{\io\th}\times V$ with $G = GL_{2n}$
($n \ge 1$).  
We consider a restriction $\Res A = \Res_{\CX_T, V',B}^{\CX}A$, where
$\CX_T = T^{\io\th} \times \{0\}$ with $V' = V_L^- = M_n$. 
For $z_1 = (x_s, 0) \in \CX_T$, we have 
$\pi_P\iv(z_1) = (x_sU)^{\io\th} \times M_n$.
Then by Corollary 1.15, we have 
$H_c^{\d}(\pi_P\iv(z_1) \cap \CO, \Ql) \ne 0$. 
Since $\pi_P\iv(z_1) \cap \CU = \pi_P\iv(z_1) \cap \CO$,
and the restriction of $\CE$ on $\pi_P\iv(z_1) \cap \CO$ 
is a non-zero constant sheaf,   
we have $H_c^{\d}(\pi_P\iv(z_1) \cap \CU, \CE) \ne 0$,  
and so $\BH_c^{\d - d}(\pi_P\iv(z_1) \cap \ol\CU, A) \ne 0$. 
It follows that 
$\CH_{z_1}^{\d - d}((\pi_P)_!j^*A) \ne 0$. This implies that 
\begin{equation*}
T_{\CU} \times \{ 0\} \subset \supp \CH^{\d -d}(\Res A), 
\end{equation*}
where $T_{\CU} = \{ x_s \in T^{\io\th} \mid (x,v) \in \CU\}$. 
Let $\w_1 : \CX \to T^{\io\th}/S_n$ be the Steinberg map on 
$\CX$ (cf. [SS2, (1.6.2)]) which associates the semisimple class
of $x_s$ to $(x,v)$.  We denote by $\w$ the restriction of $\w_1$
on $\CU$. 
Then for each $x_s \in \w(\CU)$, $\w\iv(x_s)$ have the same 
dimension $\dim \CO$.  
Since $T_{\CU}/S_n = \w(\CU)$, we have 
\begin{equation*}
\dim T_{\CU} = \dim \w(\CU) = \dim \CU - \dim \CO = d - \d.
\end{equation*}
Hence we have $\dim \supp \CH^{\d-d}(\Res A) \ge d - \d$.
But if $A$ is cuspidal, $A$ must satisfy the condition  
\begin{equation*}
\dim\supp\CH^{\d - d}(\Res A) < d - \d.
\end{equation*}
(Here $c_{\CX_T} = 0$.)
Thus $A$ is not cuspidal.
\par
Next assume that $\CX = G \times V$ is of enhanced type 
with $G = GL_n$ $(n \ge 2)$. 
Put $m = \mu(\CO_0)$ (see 2.16 for the definition), and we consider
a restriction $\Res A = \Res_{\CX_T, V', B}^{\CX}$, where 
$\CX_T = T^{\io\th} \times \{ 0\}$ with 
$V' = V_L^- = M_m$.   
Hence for $z_1 = (x_s,0) \in \CX_T$,
we have $\pi_P\iv(z_1) = x_sU \times M_m$.
Thus by Corollary 2.17, we have 
$H_c^{\d + m}(\pi_P\iv(z_1) \cap \CO, \Ql) \ne 0$.  
As in the previous case, this implies that 
\begin{equation*}
d - \d \le \dim\supp\CH^{\d + m -d}(\Res A).
\end{equation*}
But if $A$ is cuspidal, we must have
\begin{equation*}
\dim\supp\CH^{\d + m - d}(\Res A) < (d- \d - m) - c_{\CX_T} = d - \d.
\end{equation*}
(Here $c_{\CX_T} = -m$). 
Hence $A$ is not cuspidal.
The proposition is proved.
\end{proof}

In the discussion below, we write $\Ind_{\CX_L, P}^{\CX}A$ as 
$\Ind A$ and similarly for $\Res A$ if there is no fear of confusion.
The following results are variants  of [L3, I, Theorem 4.4].
%%%%%
%%%%%
\begin{thm}  %%%%  Theorem 6.4.
Let $\CX$ be of pure exotic type or of enhanced type, and $\CX_L \in \CS(\CX)$. 
Recall the notation $\wh\CX_L^0$ in 4.7.
\begin{enumerate}
\item
Assume that $G \ne T$.  
For any $A \in \wh\CX$, there exists $A_0 \in \wh\CX_T^0$ 
such that $A$ is a direct summand of ${\Ind} A_0$.
\item
If $A_1 \in \wh \CX^0_L$,  then $\Ind  A_1 \in \CM(\CX)$.
\item
Assume that $V_L = V_L^+$.  
If $A \in \wh\CX$, then 
$\Res A[-c_{\CX_L}] \in \DD^{\le 0}(\CX_L)$. 
\item
Assume that $V_L = V_L^+$.
If $A \in \wh\CX$ and $K \in \CM^{L^{\th}}(\CX_L)$ is such that 
any simple constituent of $K$ is contained in $\wh\CX_L$, then 
\begin{equation*}
\Hom_{\DD(\CX_L)}(\Res A[-c_{\CX_L}], K) \simeq \Hom_{\DD(\CX)}(A, \Ind K).
\end{equation*}

\end{enumerate} 
\end{thm}

\remark{6.5.}
If $\CX_L$ satsifies the assumption $V_L = V_L^+$, 
we have $\wh\CX_L = \wh\CX_L^0$.
Hence the condition in (ii) is automatic.  

\par\medskip
As a corollary, we obtain a classification of character sheaves 
on $\CX$,  namely we have

\begin{cor}  %%%%% Corollary 6.6.
The set $\wh\CX_{\CL}^{\ps}$ coincides with $\wh\CX_{\CL}$ for 
each tame local system $\CL$ on $T$.  In particular, if $\CX$ 
is of pure exotic type or of enhanced type, the set of character 
sheaves $\wh\CX_{\CL}$ is given by (4.9.1) or (4.9.2), respectively.  
\end{cor}

\para{6.7.}
The rest of the paper is devoted to the proof of the theorem.
We prove it following the strategy of the proof of 
Theorem 4.4 in [L3].  
When $G$ is a torus, the theorem  is obvious.  
So we assume that 
$G$ is not a torus, and that the theorem  is already proved for 
$\CX_{G'}$ for a smaller rank group $G'$. 
We shall prove them for $\CX = \CX_G$ by using the inductive argument. 
\par
First we show that 
\par\medskip\noindent
(6.7.1) \ Theorem 6.4 (ii) holds for $\CX$. 
\par\medskip
We assume that $P \ne G$. 
If $L = T$, the assertions hold by (4.9.1) and (4.9.2).  So 
we may assume that $L \ne T$. 
We consider the case where $\CX$ is of pure exotic type.  
The proof for the enhanced case is similar.
Then $\CX_L \simeq \CX'_L \times \CX''_L$, where 
$\CX'_L = GL_{2n'}^{\io\th} \times V^+_L$ and 
$\CX''_L = (GL_1)^{n-n'} \times V_L'$
as in 4.7. $A_1 \in \wh\CX_L$ can be written as $A_1 = A'_1\boxtimes A_1''$ 
with $A_1' \in \wh\CX'_L, A_1'' \in \wh\CX''_L$.  Since $A_1 \in \wh\CX_L^0$, 
$A''_1 \simeq \CE[n-n']$ for a tame local system $\CE$ on $(GL_1)^{n-n'}$.
By applying Theorem 6.4 (i) to $\CX'_L$,
one can find $A_0' \in \wh\CX^0_{T_1}$ ($T_1$ is a maximal torus of $GL_{2n'}$) such that
$A_1'$ is a direct summand of $\Ind_{\CX_{T_1}}^{\CX'_L}A_0'$.  
Here $\CX''_L = \CX_{T_2}$ for $T_2 = (GL_1)^{n-n'}$, and for $A_1'' \in \wh\CX_{T_2}^0$,
put $A_0 = A_0'\boxtimes A_1''$.  Then $A_0 \in \wh\CX^0_T$ 
and $A_1$ is a direct summand of $\Ind_{\CX_T}^{\CX_L}A_0$. 
Since $K = \Ind_{\CX_T}^{\CX_L}A_0$ has its support in $L^{\io\th} \times V_L^+$,  
$K$ is $P^{\th}$-equivariant.  
By Proposition 4.3, we have
\begin{equation*} 
\Ind_{\CX_T}^{\CX}A_0 \simeq \Ind_{\CX_L}^{\CX} 
                (\Ind_{\CX_T}^{\CX_L}A_0).
\end{equation*}
By (4.9.1), $\Ind_{\CX_T}^{\CX}A_0$ is a semisimple perverse sheaf.
Since $\Ind_{\CX_L}^{\CX}A_1$ is a direct summand of $\Ind_{\CX_T}^{\CX}A_0$,
$\Ind_{\CX_L}^{\CX}A_1$ is also a semisimple perverse sheaf. 
Thus (6.7.1) holds. 

\para{6.8.}
Under the notation in 4.7, assume that $V_L = V_L^+$.
Thus $V = V_L^-\oplus V_L^+ \oplus V_L''$ and $V' = V_L^- \oplus V_L^+$.
We consider a commutative diagram

\begin{equation*}
\begin{CD}
  D  @<f<<   \wt\CX^P   @>\r >>   \CX   \\
  @A\b AA              @AA\s A    @.  \\
  D'  @<\f <<     H \times P^{\io\th} \times V'   \\
  @V\g VV      @VV\t_1 V      \\
 \CX_L   @<\pi_P<<   P^{\io\th} \times V'  @>j>>   \CX,
\end{CD}
\end{equation*}
where
\begin{align*}
D &= H \times^{P^{\th}}(L^{\io\th} \times V_L), \quad 
        (\text {$P^{\th}$ acts on $L^{\io\th}$ through the map 
            $P^{\th} \to P^{\th}/U_P^{\th} \simeq L^{\th}$)}  \\\
D' &= H \times L^{\io\th} \times V_L,  \\
\f &: (g,p,v) \mapsto (g, \ol p, \ol v), \quad \\  
\b &: \text{the quotient map by $P^{\th}$}, \\
f  &: \text{the map induced from $\f$ through the quotient by $P^{\th}$} \\
\g &: (g, \ell,v) \mapsto (\ell,v),  \\
\t_1&: (g, p, v) \mapsto (p,v),
\end{align*}
and the maps $\r, \s$ are as in (4.2.1), the maps $j,\pi_P$ are as in 5.1 
(with respect to $\CX_L, V'$).
Note that the map $\pi_P\circ \t_1 = \g\circ\f$ coincides with the map 
$\t : H \times P^{\io\th} \times V' \to \CX_L$ in (4.2.1).
It is clear that both squares are cartesian squares.
\par
Take $K  \in \CM^{L^{\th}}(\CX_L)$.
$K$ can be lifted to a $P^{\th}$-equivariant complex since 
$V_L = V_L^+$. 
Since $\g$ is a  principal $H$-bundle and $\b$ is a principal 
$P^{\th}$-bundle, one can define 
$K' = \b_{\flat}\wt\g K[-\dim P^{\th}]$, so that 
$\wt\g K = \wt\b K'$ with $K' \in \CM(D)$.
Since $\Ind K = \r_!\s_{\flat}(\g\circ\f)^*K[a]$
(see (4.2.2)), we see that 
\begin{equation*}
\tag{6.8.1}
\Ind K \simeq \r_!\wt f K'[a -c - \dim U_P^{\th}],
\end{equation*}
where $c = \dim U_P^{\io\th} + \dim V_L^-$.
Let $A \in \wh\CX$. Note that 
$\Res A = (\pi_P)_!j^*A$. 
We define $\Res^+ A = f_!\r^*A[d]$ with 
$d = \dim U_P^{\th}$.
We show that 
\begin{equation*}
\tag{6.8.2}
\wt\g(\Res A) \simeq \wt\b(\Res^+A).
\end{equation*}
In fact, since both squares in the diagram are cartesian squares, we have 
\begin{align*}
\wt\g(\Res A) [-\dim H]&\simeq \f_!(\t_1)^*j^*A, \\
\wt\b(\Res^+A)[-\dim H] &\simeq \f_!\s^*\r^*A. 
\end{align*}
Let $\z = \r\circ \s, \z' = j\circ \t_1$.
In order to show (6.8.2), it is enough to see that $\z^*A \simeq (\z')^*A$. 
But $\z$ is the composite of the inclusion 
$H \times P^{\io\th} \times V' \hra H \times \CX$ and the map 
$H \times \CX \to \CX$, $(g, x, v) \mapsto (gxg\iv, gv)$, while
$\z'$ is the composite of the inclusion and the map 
$H \times \CX \to \CX$, $(g,x,v) \mapsto  (x,v)$. Since
$A$ is an $H$-equivariant perverse sheaf, we have 
$\z^*A \simeq (\z')^*A$, and (6.8.2) follows.   
\par
Next we show, for any integer $i$, that  
\begin{equation*}
\tag{6.8.3}
\Hom_{\DD(D)}(\Res^+A[-c_{\CX_L}], K'[i]) \simeq 
        \Hom_{\DD(\CX)}(A, \Ind K[i]).
\end{equation*}
In fact,  
by 0.2, the left hand side is equal to
\begin{align*}
\Hom (f_!\r^*A[d-c_{\CX_L}], K'[i]) 
           &\simeq \Hom (\r^*A[d-c_{\CX_L}], f^!K'[i]) \\
           &\simeq \Hom (\r^*A, f^*K'[i + 2c-d +c_{\CX_L}])
\end{align*}
since $f^! = f^*[2c]$ 
as $f$ is smooth with connected fibres of dimension $c$. 
On the other hand, again by 0.2, the right hand side is equal to
\begin{align*}
\Hom (A, \r_!\wt fK'[i+ a -c - d]) 
    &\simeq \Hom (\r^*A, f^*K'[i + a - d])
\end{align*}
since $\r_* = \r_!$ as $\r$ is proper, and $\wt f = f^*[c]$.
It is enough to show that 
\begin{equation*}
\tag{6.8.4}
2c + c_{\CX_L} = a.
\end{equation*}
But (6.8.4) is equivalent to (6.1.1).  
Thus (6.8.4) holds, and (6.8.3) follows. 

\par
Here we note that 
\par\medskip\noindent
(6.8.5) \  Take $K \in \CM^{L^{\th}}(\CX_L)$, and assume that any simple constituent of $K$ 
is contained in $\wh \CX_L$.  Then $\Ind K \in \CM(\CX)$.
\par\medskip
In fact, by Theorem 6.4 (ii), $\Ind A_1 \in \CM(\CX)$ for any 
simple constituent $A_1$ of $K$.  Then by (0.1.2) we obtain the assertion. 
\par

Take $K$ as in (6.8.5).  
If $i < 0$, we have $\Hom_{\DD(\CX)}(A, \Ind K[i]) = 0$
since $A \in \CM(\CX)$, and $\Ind K [i] \in \DD^{\ge 1}(\CX)$ (see (0.1.3)). 
It follows from (6.8.3) that 
\begin{equation*}
\tag{6.8.6}
\Hom_{\DD(D)}(\Res^+A[-c_{\CX_L}], K'[i]) = 0  \qquad\text{ for } i < 0.  
\end{equation*}
\par
We show that 
\begin{equation*}
\tag{6.8.7}
{}^pH^i(\Res^+A[-c_{\CX_L}]) = 0 \qquad \text{ for } i > 0.
\end{equation*}
Suppose not, and let $i > 0$ be the largest integer such that 
${}^pH^i(\Res^+A[-c_{\CX_L}]) \ne 0$. 
Since $\wt\b, \wt\g$ commute with the perverse cohomology functor by 0.2,
$i$ is also the largest integer such that ${}^pH^i(\Res A[-c_{\CX_L}]) \ne 0$.
By Proposition 5.9, $\Res A$ is a semisimple complex, 
and so it is decomposed as
\begin{equation*}
\Res A[-c_{\CX_L}] = \bigoplus_j{}^pH^j(\Res A[-c_{\CX_L}])[-j].
\end{equation*}
Thus we have 
a non-zero morphism $\Res A[-c_{\CX_L}] \to {}^pH^i(\Res A[-c_{\CX_L}])[-i]$.
By (6.8.2), this implies that there exists a non-zero morphism
$\wt\b(\Res^+ A[-c_{\CX_L}+i]) \to \wt\b({}^pH^i(\Res^+ A[-c_{\CX_L}]))$.
Since $\Res A[-c_{\CX_L} + i] \in \DD^{\le 0}(\CX_L)$ by (0.1.1),  
we have $\Res^+A[-c_{\CX_L} + i] \in \DD^{\le 0}(D)$ by 0.2. 
Hence by (0.2.3), there exists a non-zero morphism 
\begin{equation*}
\Res^+ A[-c_{\CX_L}] \to {}^pH^i(\Res^+A[-c_{\CX_L}])[-i].
\end{equation*}
Put $K = {}^pH^i(\Res A[-c_{\CX_L}])$ and 
$K' = {}^pH^i(\Res^+A[-c_{\CX_L}])$.
Since the perverse cohomology functor commutes with $\wt\b, \wt\g$, 
(6.8.2) implies that 
$\wt\b K' \simeq \wt\g K$.
Thus, one can apply (6.8.6), which implies that there exists no 
non-zero morphism $\Res^+ A[-c_{\CX_L}] \to K'[-i]$ for $i > 0$. 
This is a contradiction, and (6.8.7) follows. 
\par
(6.8.7) implies, by (0.1.1),  that
\begin{equation*}
\tag{6.8.8} 
\Res^+A[-c_{\CX_L}] \in \DD^{\le 0}(D).
\end{equation*} 
Applying $\wt\b$ to (6.8.8), 
we have $\wt\b(\Res^+A[-c_{\CX_L}]) \in \DD^{\le}(D')$.
Thus by (6.8.2), we have $\wt\g(\Res A[-c_{\CX_L}]) \in \DD^{\le 0}(D')$, 
and so
$\Res A[-c_{\CX_L}] \in \DD^{\le 0}(\CX_L)$.  
 Hence we have proved
\par\medskip\noindent
(6.8.9) Theorem 6.4 (iii) holds for $A \in \wh\CX$.
\par\medskip
Returning to the original $K$ as in (6.8.5), we have the following.
\begin{equation*}
\tag{6.8.10}
\begin{aligned}
\Hom (A, \Ind K) &\simeq \Hom (\Res^+A[-c_{\CX_L}], K')  &\quad 
                       &(\text{by (6.8.3))} \\
                  &\simeq \Hom (\wt\b \Res^+A[-c_{\CX_L}], \wt\b K') 
                    &\quad    &(\text{by (6.8.8) and (0.2.3))} \\
                  &\simeq \Hom (\wt\g \Res A[-c_{\CX_L}], \wt\g K) 
                    &\quad &(\text{by (6.8.2)) }  \\
                  &\simeq \Hom (\Res A[-c_{\CX_L}], K) 
                    &\quad &(\text{by (6.8.9) and (0.2.3)}).
\end{aligned}
\end{equation*}
(Note that $\Res A[-c_{\CX_L}] \in \DD^{\le 0}(\CX_L)$ and 
$\Res^+A[-c_{\CX_L}] \in \DD^{\le 0}(D)$ by the previous discussion.)
It follows that 
\par\medskip\noindent
(6.8.11) \ Theorem 6.4 (iv) holds for $\CX$. 
\par\medskip

Finally we show that 
\par\medskip\noindent
(6.8.12) \ Theorem 6.4 (i) holds for any $A \in \wh\CX$.
\par\medskip
Assuming that $G$ is not a torus, take $A \in \wh\CX$. 
We may assume that $A$ is given as in the proof of Proposition 6.3.
Let $\Res A $ be the restriction of $A$ with respect to 
$\CX_T= T^{\io\th} \times \{ 0\}$ as given in the proposition.
Then the discussion there shows that $\Res A[-c_{\CX_T}] \not\in \DD^{\le 1}(\CX_T)$.
Since $\CX_T$ satsifies the condition $V_T = V_T^+ = \{ 0\}$, 
$\Res A [-c_{\CX_T}]\in \DD^{\le 0}(\CX_T)$ by (iii). 
Hence ${}^pH^0(\Res A[-c_{\CX_T}]) \ne 0$.
Since $\Res A$ is a semisimple perverse sheaf and its simple components 
are all contained in $\wh\CX^0_T$ by Proposition 5.7 and Proposition 5.9, 
there exists a non-zero morphism 
$\Res A \to A_0$ for some $A_0 \in \wh\CX^0_T$. 
By applying (6.8.10) for $K = A_0$, we see that 
there exists a non-zero morphism $A \to \Ind A_0$.
This shows that $A$ appears as a direct 
summand of $\Ind A_0$.  (6.8.12) is proved.  
\par
This completes the proof of Theorem 6.4. 

\remark{6.9.}
Let $\CX$ be of pure exotic type.
\par
(i) \ For $P^{\th}$-equivariant $A_1 \in \wh\CX_L$, 
 $\Ind_{\CX_L}^{\CX}A_1$ is not necessarily contained in $\DD(\CX)^{\le 0}$ if we drop 
the assumption $A_1 \in \wh\CX_L^0$.
In fact $\Ind_{\CX_T}^{\CX}A_1 = K_{2n, T,\CE}[n]$ for 
$A_1 = \wt\CE_{V_1}[\dim \CX_{T,V_1}]$ with $\dim V_1 = n$ 
by 4.9.  Hence $\Ind_{\CX_T}^{\CX}A_1 \notin \DD^{\le 0}(\CX)$ by 
Proposition 1.12. 
\par
(ii) \ For $A \in \wh\CX$,  $\Res_{\CX_L}^{\CX}A$ is not necessarily contained 
in $\DD(\CX_L)^{\le 0}$ if we drop the assumption $V_L = V_L^+$.  
In fact, consider the case where 
$G = GL_2$.  In this case $H = SL_2$,
$G^{\io\th} = T^{\io\th} = Z(G)$ and $\CX = Z(G) \times V$ with $\dim V = 2$.  
Then any character sheaf on $\CX$ is given, for a tame local system $\CE$ on $T^{\io\th}$, 
 as $A_{\CE} = (\CE\boxtimes (\Ql)_V)[3]$,  
or $A_{\CE}' = (\CE \boxtimes (\Ql)_{\{0\}})[1]$ on $T^{\io\th} \times \{0\}$.  
(Here we denote by $(\Ql)_X$ the constant sheaf $\Ql$ on the variety $X$).
Now $\CX_T = T^{\io\th} \times V_T$ with $\dim V_T = 1$. 
We have
$\Res_{\CX_T}^{\CX}A_{\CE} = (\CE\boxtimes (\Ql)_{V_T})[1]$, 
since $p_!(\Ql)_V = (\Ql)_{V_T}[-2]$ for the projection $p: V \to V_T$, 
and $\Res_{\CX_T}^{\CX}A_{\CE}' = (\CE \boxtimes (\Ql)_{\{0\}})[1]$. 
Since any character sheaf on $\CX_T$ is given as $(\CE\boxtimes (\Ql)_V)[2]$
or $(\CE\boxtimes(\Ql)_{\{0\}})[1]$, we see that 
$\Res_{\CX_T}^{\CX}A \in \DD(\CX)^{\le 1}$, but 
$\Res_{\CX_T}^{\CX}A \notin \DD(\CX)^{\le 0}$. 

\para{6.10.}
By making use of Braden's theroem, one can show a certain 
refinement of Theorem 6.4 (iii).
We follow the notation in 5.8 and Proposition 5.9. 
We define a diagonal action $\s$ on $\CX$ by an element 
$g \in H$ as follows; in the pure exotic case, 
$g$ is a permutation $(1, n+1)(2, n+2)\cdots (n, 2n)$ of 
the symplectic basis $\{ e_1, \dots, e_n, f_1, \dots, f_n\}$.
In the enhanced case, $g$ is a permutation of the basis 
$\{ e_1, \dots, e_n\}$ giving the longest element in $S_n$. 
Put $L_1 = \s(L)$. Let $P_1$ be a parabolic subgorup of $G$ containing 
$B$ and $L_1$, and $P_1^-$ be the opposite parabolic subgroup of $P_1$
such that $P_1 \cap P_1^- = L_1$.  Then $\s(P) = P_1^-$.  Note that 
in the pure exotic case, $L_1 = L$ and $P_1 = P$.  
Recall that $V = V_L^- \oplus V_L^+ \oplus V_L''$, with $V' = V_L^-\oplus V_L^+$.
We consider a similar decompostion $V = V^-_{L_1}\oplus V_{L_1}^+\oplus V''_{L_1}$
with respect to $L_1$, and put $V'_1 = V^-_{L_1} \oplus V_{L_1}^+$. 
We have $\s(V') = V_{L_1}^+ \oplus V''_{L_1}$ and $\s(V^+_L) = V^+_{L_1}$. 
Put $c'_{\CX_L} = c_{\CX_{L_1}}$.  Explicitly, $c'_{\CX_L} = c_{\CX_L}$ in the 
pure exotic case, and $c'_{\CX_L} = -\dim V_L''$ (here $c_{\CX_L} = -\dim V_L^-$) in 
the enhanced case. 
We have the following result.

\begin{prop}   %%%%  Prop. 6.11
Assume that $V_L = V_L^+$.  Then for any $A \in \wh\CX$, 
$\Res A[c'_{\CX_L}] \in \DD^{\ge 0}$. 
\end{prop}

\begin{proof}
The following proof was inspired by Achar [A, Theorem 3.1]. 
We have the following commutative diagram.

\begin{equation*}
\begin{CD}
L^{\io\th} \times V^+_L  @>i >> P^{\io\th} \times V' @>j >>  G^{\io\th} \times V  \\
     @V\s_L VV      @VV\s' V       @VV\s V   \\
L_1^{\io\th} \times V^+_{L_1}   @>\bar i_1 >>  (P_1^-)^{\io\th} \times \s(V') 
                     @>\bar j_1 >>  G^{\io\th} \times V,   
\end{CD}
\end{equation*}
where $\bar i_1, \bar j_1$ are defined similaly to $i, j$, and  
$\s'$ (resp. $\s_L$) is the restriction of $\s$ on $P^{\io\th} \times V'$
(resp. on $L^{\io\th} \times V^+_L$). 
Since $\s, \s', \s_L$ are isomorphisms, we have 
\begin{equation*}
\tag{6.11.1}
\s_L^*\bar i_1^*\bar j_1^! \simeq i^*j^!\s^! \simeq i^*j^!\s^*.
\end{equation*}
By Braden's theorem, applied to the maps $i_1, j_1, \bar i_1, \bar j_1$, 
we have a canonical isomorphism 
$i_1^!j_1^*K \simeq \bar i_1^*\bar j_1^!K$ for a weakly equivariant $K$
on $\CX$. 
Hence if $A \in \wh\CX$, we have
\begin{equation*}
\tag{6.11.2}
D(i^!j^*\s^*A) \simeq i^*j^!\s^*A  \simeq \s_L^*\bar i_1^*\bar j_1^!A 
    \simeq \s_L^*i_1^! j_1^*A.
\end{equation*}
Since $j^*\s^*A, j_1^*A$ are weakly equivariant, by Lemma 0.7, we have
\begin{equation*}
i^!j^*\s^*A \simeq (\pi_P)_!j^*\s^*A, \qquad 
    i_1^!j_1^*A \simeq  (\pi_{P_1})_!j_1^*A.  
\end{equation*}
Substituting this into (6.11.2), we have
\begin{equation*}
\tag{6.11.3}
D(\Res_{\CX_L,V',P}^{\CX}\s^*A) \simeq \s_L^*(\Res_{\CX_{L_1}, V_1', P_1}^{\CX}A).
\end{equation*}
By Theorem 6.4 (iii), we have 
$\Res_{\CX_{L_1}}^{\CX}A[-c'_{\CX_L}] \in \DD^{\le 0}(\CX_{L_1})$.
Hence (6.11.3) implies that
$\Res_{\CX_L}^{\CX}\s^*A [c'_{\CX_L}] \in \DD^{\ge 0}(\CX_L)$.
Since $\s^*$ gives a bijection on the set $\wh\CX$, the proposition follows
from this.

\end{proof}

\remark{6.12.}
In the case of original character sheaves in [L3], 
a similar argument as in the proof of Proposition 6.11 gives 
a simple proof based on Braden's theorem 
of the fact that $\Res A$ is a perverse sheaf 
(cf. [L3, Theorem 6.9 (a)]).

\para{6.13.}
Assume that $\CX$ is of pure exotic type or of enhanced type.
We consider a $\th$-stable parabolic subgroup $P$ of $G$, and 
its $\th$-stable Levi subgroup $L$, in general.  
For example, in the pure exotic case, consider 
$L^{\th} \simeq Sp_{2n_0} \times \prod_{i\ge 1}GL_{n_i}$, and 
correspondingly $V$ is decomposed as 
$V = V_0 \oplus \bigoplus_{i \ge 1}(V_i' \oplus V_i'')$, where 
$\dim V_0 = 2n_0$ and $\dim V_i' = \dim V_i'' = n_i$.
In that case, it is not known whether $\Res A$ is a semisimple 
complex or not.  The arguments for the proof of Proposition 5.9 
based on Braden's theroem  fails if $L^{\th}$ contains more than 
two factors. Also in that case, the problem of lifting of 
$L^{\th}$-equivaraint complex to $P^{\th}$-equivaraint complex 
in the definition of $\Ind A$ is delicate in general.  Our arguments 
can be generalized only in the 
following special cases;  let us
consider $V_L = V'/V_L^-$, where $V_L$ is a subquotient of $V_0$ or 
of $V'_i$ for some $i$ in the decomposition of $V$ so that 
$U_P^{\th}$ acts trivially on $V_L$.
Then Proposition 5.9 holds in this case, and properties (ii) $\sim$ (iv)
in Theorem 6.4 hold without change.

\par\bigskip

\begin{center}
{\sc Appendix} 
\par
\medskip
Corrections to ``Exotic symmetric space over a finite field, I, II''
\end{center}
\par\medskip
In this occasion, we give some corrections for [SS], [SS2].  We follow 
the notations in [SS], [SS2]. 
\par\medskip
{\bf A.  Corrections for [SS]}
\par\medskip
I. \   
The proof of Proposition 3.6 in [SS] contains a gap.  
In fact, in the proof of (3.6.4), $R^{2i}(\ol\psi_{m+1})_!\CE_{m+1}$ is 
claimed to be a perfect sheaf.  But this is not true since  
it does not satisfy the support condition. We give an alternate 
proof of Proposition 3.6 below. Note that the statement that 
$\IC(\CY_m, \CL_{\r})$ is a constructible sheaf in Proposition 3.6 in [SS] 
should be removed.
%%%%
\addtocounter{section}{-3}
\addtocounter{thm}{-8}
\begin{prop} %%%%  Prop. 3.6
$\psi_*\a_0^*\CE[d_n]$ is a semisimple perverse sheaf on $\CY$, equipped with
$\wt\CA_{\CE}$-action, and is decomposed as
\begin{equation*}
\psi_*\a_0^*\CE[d_n] \simeq \bigoplus_{0 \le m \le n}\bigoplus_{\r \in \CA_{\Bm,\CE}\wg}
   \wt V_{\r} \otimes \IC(\CY_m, \CL_{\r})[d_m],
\end{equation*} 
where $d_m = \dim \CY_m$.
\end{prop}

\begin{proof}
In the following discussion, we write the restriction of $\a_0^*\CE$ on 
$\wt\CY_m^+$, etc. simply as $\a_0^*\CE$ if there is no fear of confusion.
Recall the formula (3.5.2) in [SS],
\begin{equation*}
\tag{3.5.2}
(\psi_m)_*\a_0^*\CE \simeq \bigoplus_{\r \in \CA\wg_{\Bm,\CE}}
\Ind_{\wt\CA_{\Bm,\CE}}^{\wt\CA_{\CE}}(H^{\bullet}(\BP_1^{n-m})\otimes\r)\otimes \CL_{\r}.
\end{equation*}
This formula can be rewritten as 
\begin{equation*}
\tag{a}
(\psi_m)_*\a_0^*\CE \simeq 
     \biggl(\bigoplus_{\r \in \CA\wg_{\Bm,\CE}}\wt V_{\r} \otimes \CL_{\r}\biggr)
       [-2(n-m)] + \CN_m,
\end{equation*}  
where $\CN_m$ is a sum of various $\CL_{\r}[-2i]$ for $\r \in \CA\wg_{\Bm, \CE}$
with $0 \le i < n-m$.
Let $\ol\psi_m$ be as in the proof of Proposition 3.6.  
The statement (3.6.1) in [SS] should be replaced by the following formula.  
For each $m \le n$,  
\begin{equation*}
\tag{3.6.1*}
(\ol\psi_m)_*\a_0^*\CE \simeq \bigoplus_{0 \le m' \le m}\bigoplus_{\r \in \CA\wg_{\Bm',\CE}}
          \wt V_{\r}\otimes \IC(\CY_{m'}, \CL_{\r})[-2(n- m')] + \ol\CN_m, 
\end{equation*} 
where $\ol\CN_m$ is a sum of various $\IC(\CY_{m'}, \CL_{\r})[-2i]$ for 
$0 \le m' \le m$ and $\r \in \CA_{\Bm',\CE}\wg$ with $i < n - m'$. 
Note that (3.6.1*) will imply the proposition. 
In fact, $\ol\psi_m = \psi$ for $m = n$, and in that case, 
$d_n - d_{m'} = 2(n-m')$ by Lemma 3.3 (iii) 
in [SS].  
Since $\psi$ is semi-small by Lemma 3.3 (iv), $\psi_*\a_0^*\CE[d_n]$ is a semisimple 
perverse sheaf. 
But since $d_n-2i > d_{m'} = \dim \CY_{m'}$, $\IC(\CY_{m'}, \CL_{\r})[d_n-2i]$
is not a perverse sheaf for any $\IC(\CY_{m'}, \CL_{\r})[-2i]$ appearing in $\ol\CN_n$.
It follows that $\ol\CN_n = 0$ and we obtain the proposition. 
\par
We show (3.6.1*) by induction on $m$.  If $m = 0$, $(\ol\psi_m)_*\a_0^*\CE$ coincides with 
$(\psi_m)_*\a_0^*\CE$.  Hence it holds by (a).  
We assume that (3.6.1*) holds for any $m' < m$. 
Recall that $\CY_m \backslash \CY_{m-1} = \CY_m^0$ and $j : \CY_m^0 \to \CY_m$ is 
the open immersion. Since $\ol\psi_m$ is proper, 
$(\ol\psi_m)_*\a_0^*\CE$ is a semisimple complex on $\CY_m$. 
Here we note that $(\ol\psi_m)_*\a_0^*\CE \simeq (\ol\psi_m)_!\a_0^*\CE$ has a natural 
structure of $\wt\CA_{\CE}$-complexes.
In fact, $(\psi_m)_!\a_0^*\CE$ has a $\wt\CA_{\CE}$-action by (3.5.2).  
It induces an $\wt\CA_{\CE}$-action on $(j\circ\psi_m)_!\a_0^*\CE$, and hence on its
perverse cohomology ${}^pH^i((j\circ\psi_m)_!\a_0^*\CE)$. 
On the other hand, by induction $(\ol\psi_{m-1})_!\a_0^*\CE$ has a natural 
$\wt\CA_{\CE}$-action, which induces an $\wt\CA_{\CE}$-action on 
${}^pH^i((\ol\psi_{m-1})_!\a_0^*\CE)$.  
Thus by using the perverse cohomology long exact sequence, one can define 
an action of $\wt\CA_{\CE}$ on ${}^pH^i((\ol\psi_m)_!\a_0^*\CE)$.  
Since $(\ol\psi_m)_!\a_0^*\CE$ is a semisimple complex, in this way
the action of $\wt\CA_{\CE}$ can be defined. 
\par
Now since $(\ol\psi_m)_*\a_0^*\CE$ is a semisimple complex, it is a direct sum 
of the form $A[s]$ for a simple perverse sheaf $A$.  Suppose that $\supp A$ is not contained 
in $\CY_{m-1}$. Then $\supp A \cap \CY_m^0 \neq \emptyset$ and the restriction of 
$A$ on $\CY_m^0$ is a simple perverse sheaf on $\CY_m^0$. 
The restriction of $(\ol\psi_m)_*\a_0^*\CE$ on $\CY_m^0$ is isomorphic to 
$(\psi_m)_*\a_0^*\CE$.  Hence its decomposition is given by the formula 
(a).  It follows that $A|_{\CY_m^0} = \CL_{\r}$ for some $\r$, up to shift.
This implies that $A = \IC(\CY_m, \CL_{\r})[d_m]$ and that the direct sum of $A[s]$
appearing in $(\ol\psi_m)_*\a_0^*\CE$ such that $\supp A \cap \CY_m^0 \neq \emptyset$ 
is given by 
\begin{equation*}
K_1 = \bigoplus_{\r \in \wt\CA_{\Bm, \CE}\wg}\wt V_{\r} \otimes 
       \IC(\CY_m, \CL_{\r})[-2(n-m)] + \CN_m', 
\end{equation*} 
where $\CN_m'$ is a sum of various $\IC(\CY_m, \CL_{\r})[-2i]$ 
with $0 \le i < n - m$. 
\par
If $\supp A$ is contained in $\CY_{m-1}$, $A[s]$ appears as a direct summand of 
$(\ol\psi_{m-1})_*\a_0^*\CE$, which is decomposed as in (3.6.1*) by induction
hypothesis.  Thus if we exclude the contribution from the restriction of $K_1$ on $\CY_{m-1}$,
such $A[s]$ is determined from  $(\ol\psi_{m-1})_*\a_0^*\CE$.  So we consider the restriction
of $K_1$ on $\CY_{m-1}$.  Since $\CN_m'$ is contained in $\ol\CN_m$, we can ignore this part.
Hence it is enough to consider $A = \IC(\CY_m, \CL_{\r})[d_m]$. 
Note that the multiplicity space of $A$ in $K_1$ is $\wt V_{\r}$. 
Hence the multiplicity space of a simple perverse sheaf $A'$ appearing in the decomposition 
of $A|_{\CY_{m-1}}$, up to shift, has a structure of $\wt\CA_{\CE}$-module which is 
a sum of $\wt V_{\r}$.   
But by (3.6.1*) applied for $m-1$, the multiplicity space of a simple perverse sheaf $B$
appearing in the first term of (3.6.1*) is a sum of $\wt V_{\r'}$ 
with $\r' \in \CA_{\Bm', \CE}\wg$ for $m' < m$. 
Since $\wt V_{\r}$ and $\wt V_{\r'}$ are different simple $\wt\CA_{\CE}$-modules, 
$A|_{\CY_{m-1}}$ gives no contribution on the first terms in (3.6.1*). 
Since $\ol\CN_{m-1} \subset \ol\CN_m$, this proves (3.6.1*), and the proposition follows. 
\end{proof}

\par\medskip
II. \ Since (3.6.1) was modified, the formula (4.9.1) must be modified to the following form.

\begin{equation*}
\tag{4.9.1*}
\begin{split}
(\ol\pi_m)_*&(\a^*\CE)|_{\pi\iv(\CX_m)}[d_m] \\
    &\simeq 
      \bigoplus_{0 \le m' \le m}\bigoplus_{\r \in \CA_{\Bm',\CE}\wg}
        \wt V_{\r}\otimes \IC(\CX_{m'},\CL_{\r})[d_{m'} - 2(n-m)] + \CM_m,
\end{split}
\end{equation*}
where $\CM_m$ is a sum of various $\IC(\CX_{m'},\CL_{\r})[-2i]$ for 
$0 \le m' \le m$ and $\r \in \CA_{\Bm',\CE}\wg$ with 
$m' - m \le i < n - m$.  
\par
\medskip
By using a similar discussion as in I, one can define a natural 
$\wt\CA_{\CE}$-action on $(\ol\pi_m)_*\a^*\CE$.  Then (4.9.1*) is proved 
by a similar argument as in the proof of (3.6.1*), by using Proposition 4.8
instead of (3.5.2). 
\par
By applying (4.9.1*) to the case where $m = n$, we see that 
any simple perverse sheaf $A[i]$ appearing in the direct sum decompostion
of $\pi_*\a^*\CE[d_n]$ 
has the property that $\supp A \cap \CY \ne \emptyset$. Since $\CY$ is open dense 
in $\CX$, and the restriction of $\pi_*\a^*\CE$ on $\CY$ coincides with $\p_*\a_0^*\CE$, 
Theorem 4.2 in [SS] follows from  Proposition 3.6.

\par\medskip
III. \ In 5.6, the variety $\CZ$ is defined with respect to the map 
$\pi_P: P^{\io\th}\uni \times V \to L^{\io\th}\uni \times \wt V_L$
and an $L^{\th}$-orbit $\CO_L$.  But since the action of $U_P^{\th}$ on $\wt V_L$ is 
not necessarily trivial, $\pi_P\iv(\CO_L)$ is not stable by the action of 
$P^{\th}$, and $\CZ$ cannot be defined in this way.  We must replace $\CO_L$ 
by the $P^{\th}$-orbit $\wt\CO_L$ containing $\CO_L$.  Since the set of $L^{\th}$-orbits in 
$L^{\io\th} \times V_L$ is finite, $\wt\CO_L$ is a finite union of $L^{\th}$-orbits,
and so $\dim \wt\CO_L = \dim \CO'_L$ for some $L^{\th}$-orbit $\CO_L'$. 
Clearly $\pi_P\iv(\wt\CO_L)$ is stable by $P^{\th}$.
Thus in the definition of $\CZ$, $\CO_L$ should be replaced by $\wt\CO_L$.  
Accordingly, Proposition 5.7 should be modified to the following form.

\addtocounter{section}{2}
%\addtocounter{thm}{1}
\begin{prop} %%%  Prop. 5.7.
Put $c = \dim \CO, \ol c = \dim \CO_L$ and $\ol c' = \dim \wt\CO_L$.
\begin{enumerate}
\item
For $(\ol x, \ol v) \in \wt\CO_L$, we have $\dim (\CO \cap \pi_P\iv(\ol x, \ol v)) \le 
                (c - \ol c')/2 \le (c - \ol c)/2$.
\item
For $(x,v) \in \CO$, 
\begin{equation*}
\dim \{gP^{\th} \in H/P^{\th} \mid (g\iv xg, g\iv v) \in \pi_P\iv(\wt\CO_L)\}
                 \le (\nu_H -c/2) - (\ol\nu - \ol c'/2).
\end{equation*}
\item
Put $d_0 = 2\nu_H - 2\ol\nu + \ol c'$.  Then $\dim \CZ_{\w} \le d_0$ for all $\w$.  
Hence $\dim \CZ \le d_0$.
\end{enumerate}
\end{prop}
The arguments used to prove Proposition 5.7 in [SS] work well under this modification.
Note that Propostion 5.7 (i) in [SS] holds without change. 
\par
A similar correction should be made in Section 6.  In the discussion of 6.1, 
the $L^{\th}$-orbit $\CO'$ should be replaced by the $P^{\th}$-orbit $\wt\CO'$ 
so that Proposition 5.7 can be applied. 
In particular, the definition of $d_{z,z'}$ in 6.1 should be replaced by 
$d_{z,z'} = (\dim Z_H(z) - \dim Z_{P^{\th}}(z'))/2 + \dim U_P^{\th}$. 
\par
\medskip
IV. \ 
The followings are simple typos. 
\begin{itemize}
\item
The latter formula in the third line of the proof of Proposition 5.7 should be 
replaced by 
\begin{equation*}
V = \wt V_0 \oplus
       \bigoplus_{i=1}^kV_i.
\end{equation*} 

\item
(5.4.2) and (5.4.3) in Section 5 should be replaced by 
\begin{equation*}
\tag{5.4.2}
H^i(\pi_1\iv(x,v), \Ql) \simeq \bigoplus_{\Bmu \in \CP_{n,2}}
\r_{\Bmu} \otimes \CH^{i-\dim \CX\uni + \dim \CO_{\Bmu}}\IC(\ol\CO_{\Bmu}, \Ql).
\end{equation*}

\begin{equation*}
\tag{5.4.3}
H^{2d_{\Bmu}}(\pi_1\iv(x,v), \Ql) \supset 
   \r_{\Bmu}\otimes \CH^0_{(x,v)}\IC(\ol\CO_{\Bmu},\Ql) \simeq \r_{\Bmu}.
\end{equation*}

\end{itemize}

\par\bigskip
{\bf B. Corrections for [SS2]}
\par\medskip
Since the discussion in [SS2, 2.1] is based on (3.6.1) in [SS], it must be 
modified as follows.  
By (3.6.1*), $(\ol\psi_m)_*\wt\CE_m$ is a semisimple complex, whose simple 
components are perverse sheaves shifted by even degrees 
(here $\wt\CE_m = \a_0^*\CE|_{\psi\iv(\CY_m)}$).  It follows that 
${}^pH^i((\ol\psi_m)_*\wt\CE) = 0$ for odd $i$.  
By using a canonical distinguished triangle 
$((j_0)_!(\psi_m)_*\wt\CE_m, (\ol\psi_m)_*\wt\CE_m, (\ol\psi_{m-1})_*\wt\CE_{m-1})$,     
this implies that 
${}^pH^i((j_0)_!(\psi_m)_*\wt\CE_m)$ is a semisimple perverse sheaf.
Then instead of (2.1.1) in [SS2], we obtain 
\begin{equation*}
\tag{2.1.1*}
{}^pH^{2(n-m)}((j_0)_!(\psi_m)_*\wt\CE_m[d_m]) \simeq \bigoplus_{\r \in \CA_{m,\CE}\wg}
               \wt V_{\r}\otimes \IC(\CY_m, \CL_{\r})[d_m].
\end{equation*}
Similarly, by using (4.9.1*) we see that 
${}^pH^i(j_!(\pi_m)_*\wt\CE_m)$ is a semisimple perverse 
sheaf.  Thus instead of (2.1.2) in [SS2], we obtain

\begin{equation*}
\tag{2.1.2*}
{}^pH^{2(n-m)}(j_!(\pi_m)_*\wt\CE_m[d_m]) \simeq \bigoplus_{\r \in \CA_{m,\CE}\wg}
         \wt V_{\r} \otimes \IC(\CX_m, \CL_{\r})[d_m].
\end{equation*}  

By comparing (2.1.2*) with (1.3.1) in [SS2], we obtain a formula which is 
a replacement of (2.1.5) in [SS2]. 
\begin{equation*}
\tag{2.1.5*}
\x_{K,\vf} = \x_{T,\CL} = \sum_{m=0}^n\x_{{}^pH^{2(n-m)}K_{m, \vf_m}}. 
\end{equation*}

\par
The arguments in 2.6 remain valid if we replace (2.1.1), (2.1.2) and (2.1.5) by 
(2.1.1*), (2.1.2*) and (2.1.5*), (and by a suitable choice of the cohomology degree).

\par\vspace{1cm}
\noindent
T. Shoji \\
Department of Mathematics, Tongji University \\ 
1239 Siping Road, Shanghai 200092, P. R. China  \\
E-mail: \verb|shoji@tongji.edu.cn|
\par\bigskip\bigskip\noindent
K. Sorlin \\
L.A.M.F.A, CNRS UMR 7352, 
Universit\'e de Picardie-Jules Verne \\
33 rue Saint Leu, F-80039, Amiens, Cedex 1, France \\
E-mail : \verb|karine.sorlin@u-picardie.fr|


\bigskip 



\begin{thebibliography}{[DJMu]}

\bibitem [A] {A}  P. Achar; Green functions via hyperbolic localization; 
preprint, arXiv:1004.4412v1. 

\bibitem [AH] {AH} P. Achar and A. Henderson; 
Orbit closures in the enhanced nilpotent cone, 
Adv. in Math. {\bf 219} (2008), 27-62, Corrigendum, {\bf 228} (2011),
2984-2988.
\par
\bibitem [BBD] {BBD} A.A. Beilinson, J. Bernstein and P. Deligne;
Faisceaux pervers, Ast\'erisque {\bf 100} (1982).
\par
\bibitem [BKS] {BKS} E. Bannai, N. Kawanaka and S.-Y. Song;
The character table of the Hecke algebra 
$\CH(GL_{2n}(\Fq), Sp_{2n}(\Fq))$, J. Algebra, {\bf 129} (1990),
320--366.
\par
\bibitem
[B] {B} T. Braden; Hyperbolic localization of intersection cohomology, 
Transformation Groups, {\bf 8} (2003), 209-216.
\par\bibitem
[BL] {BL} J. Bernstein and P. Lunts; Equivariant sheaves and 
functors, Lecture Note in Math. {\bf 1578}, Springer-Verlag, 1994.
\par\bibitem
[FG] {FG} M. Finkelberg and V. Ginzburg; Cherednik algebras 
for algebraic curves, Progress in Math. {\bf 284}, (2010), 121-153
\par\bibitem
[FGT] {FGT} M. Finkelberg, V. Ginzburg and R. Travkin; 
Mirabolic affine Grassmannian and character sheaves, Selecta Math.
(N.S.) {\bf 14} (2009), 607-628.
\par
\bibitem
[Gi] {Gi} V. Ginzburg; Admissible modules on a symmetric space, 
Ast\'erisque {\bf 173-174} (1989), 199-255. 
\par
\bibitem
[Gr] {Gr} I. Grojnowski; Character sheaves on symmetric spaces, 
Ph.D. Thesis, MIT, 1992, available at 
\verb|www.dpmms.cam.ac.uk/~groj/thesis.ps| .
\par
\bibitem [H1] {H2} A. Henderson; Fourier transform, parabolic induction,
and nilpotent orbits, Transformation Groups, {\bf 6}, (2001), 
353-370.
\par
\bibitem [HT] {HT} A. Henderson and P.E. Trapa; The exotic Robinson-
Schensted correspondence, J. Algebra {\bf 370}, (2012), 32-45. 
                                                                                                   
\bibitem [Ka1] {K1} S. Kato; An exotic Deligne-Langlands correspondence                             
for symplectic groups, Duke Math. J. {\bf 148} (2009), 306-371.                                     
\par                                                                                                
\bibitem [Ka2] {K2} S. Kato; Deformations of nilpotent cones and                                    
Springer correspondence, Amer. J. of Math. {\bf 133} (2011), 519-553.                               
\par
\bibitem [L3] {L3} G. Lusztig; Character sheaves, I, Adv. in Math.,
{\bf 56} (1985), 193-227.
\par
\bibitem [Sp] {Sp} T.A. Springer; A purity result or fixed point varieties 
in flag manifolds, J. Fac. Sci. Univ. Tokyo Sect. IA Math. {\bf 31} (1984)
271-282.
\par
\bibitem [SS] {SS} T. Shoji and K. Sorlin; Exotic symmetric space over a
finite field, I, Transformation Groups {\bf 18} (2013), 877-929. 
\par
\bibitem [SS2] {SS2} T. Shoji and K. Sorlin; Exotic symmetric space
over a finite field, II,  Transformation Groups {\bf 19} (2014), 887-926.
\par
\bibitem [T] {T} R. Travkin; Mirabolic Robinson-Schensted-Knuth correspondence, 
Selecta Mathematica (New series) {\bf 14} (2009), 727-758.  
  
\end{thebibliography}
 \end{document}